\documentclass{amsproc}
\usepackage{amssymb}
\usepackage{amsmath}
\usepackage{amsfonts}
\usepackage{graphicx}
\providecommand{\U}[1]{\protect\rule{.1in}{.1in}}
\theoremstyle{plain}

\newtheorem{corollary}{Corollary}

\newtheorem{definition}{Definition}
\newtheorem{example}{Example}

\newtheorem{lemma}{Lemma}
\newtheorem{notation}{Notation}

\newtheorem{proposition}{Proposition}
\newtheorem{remark}{Remark}

\newtheorem{theorem}{Theorem}
\numberwithin{equation}{section}
\begin{document}
\title[Nonuniform Thickness]{Nonuniform Thickness and Weighted Distance}
\author{}
\address{Department of Mathematics, University of Iowa, Iowa City, Iowa 52242, U.S.A.}
\email{odurumer@math.uiowa.edu, Phone (319)335-0774, Fax (319)335-0627}
\urladdr{}
\author{}
\curraddr{}
\email{}
\urladdr{}
\author{Oguz C. Durumeric}
\address{\ }
\urladdr{}
\thanks{}
\thanks{}
\thanks{}
\date{May 11, 2007, Revised, July 20, 2008}
\dedicatory{ }
\begin{abstract}
Nonuniform tubular neighborhoods of curves in $\mathbf{R}^{n}$ are studied by
using weighted distance functions and generalizing the normal exponential map.
Different notions of injectivity radii are introduced to investigate singular
but injective exponential maps. A generalization of the thickness formula is
obtained for nonuniform thickness. All singularities within almost injectivity
radius are classified by the Horizontal Collapsing Property. Examples are
provided to show the distinction between the different types of injectivity
radii, as well as showing that the standard differentiable injectivity radius
fails to be upper semicontinuous on a singular set of weight functions.

\end{abstract}
\maketitle

\section{Introduction}

The uniform thickness of a knotted curve is the radius of the largest tubular
neighborhood around the curve without intersections of the normal discs. This
is also known as the normal injectivity radius $IR$ of the normal exponential
map of the curve $K$ in the Euclidean space $\mathbf{R}^{n}$. The ideal knots
are the embeddings of $S^{1}$ into $\mathbf{R}^{3},$ maximizing $IR$ in a
fixed isotopy (knot) class of fixed length. As noted in [Ka],
\textquotedblleft...the average shape of knotted polymeric chains in thermal
equilibrium is closely related to the ideal representation of the
corresponding knot type\textquotedblright. Uniform thickness has been studied
extensively in several articles including [BS] G. Buck and J. Simon, [CKS] J.
Cantarella, R. B. Kusner, and J. M. Sullivan, [Di] Y. Diao, [D1, D2, D3] O. C.
Durumeric, [GL]\ O. Gonzales and R. de La Llave, [GM]\ O. Gonzales and H.
Maddocks, [Ka]\ V. Katrich, J. Bendar, D. Michoud, R.G. Scharein, J. Dubochet
and A. Stasiak, [LSDR]\ A. Litherland, J Simon, O. Durumeric and E. Rawdon,
and [N]\ A. Nabutovsky. The following thickness formula was obtained earlier
in [LSDR] in the smooth case, and in [CKS] for $C^{1,1}$ curves in
$\mathbf{R}^{3}.$

\textbf{UNIFORM THICKNESS\ FORMULA }\emph{[D1, Theorem 1]}

\emph{For every complete smooth Riemannian manifold }$M^{n}$\emph{\ and every
compact }$C^{1,1}$\emph{\ submanifold }$K^{k}$\emph{\ }$(\partial
K=\emptyset)$\emph{\ of }$M,$\emph{\ }
\[
IR(K,M)=\min\{FocRad(K),\frac{1}{2}DCSD(K)\}.
\]

In this article, we study a ball-model to describe nonuniform thickness, which
allows a nonuniform distribution of the strength of the forces along a curve
in the Euclidean space. This model can help us to understand the local shape
of large polymers which do not have a uniform structure. Most of the results
of this article are true for surfaces or submanifolds of $\mathbf{R}^{n}$, but
the results about the focal points are qualitative and the proofs are
detailed. In order to have explicit expressions for the behavior and location
of the singular (focal) points, and to be able to obtain the rigidity in
Theorem 2, we concentrated on the curves in the Euclidean space. Even though
our motivation comes from examples in $\mathbf{R}^{3},$ all results are stated
and proved in $\mathbf{R}^{n}$ since our proofs are independent of the
dimension of the ambient space, and they do not simplify for $n=2,$ $3.$ In
our model, a curve $K$\ is a union of finitely many disjoint closed curves and
it is furnished with a weight function $\mu:K\rightarrow(0,\infty).$ The
nonuniform $R-$tubular neighborhood $O(K,\mu R)$ is the union of metric balls
of radius $R\mu(q)$ centered at each $q\in K.$ As $R$ increases, the size of
these balls increase at fixed rate at each point, but the rate differs from
point to point of $K.$ This model is different from the disc-model which
allows the growth of the normal discs at different rates. One of the reasons
that we chose to investigate the ball-model is that the physical forces, such
as electrical and magnetic forces have effects in every direction rather than
being restricted to chosen planes. Furthermore, the ball-model can be
investigated more thoroughly, since there is a natural potential function,
$\min_{q\in K}\frac{\left\Vert p-q\right\Vert }{\mu(q)}.$%

\begin{figure}
[ptb]
\begin{center}
\includegraphics[
height=1.6717in,
width=2.783in
]
{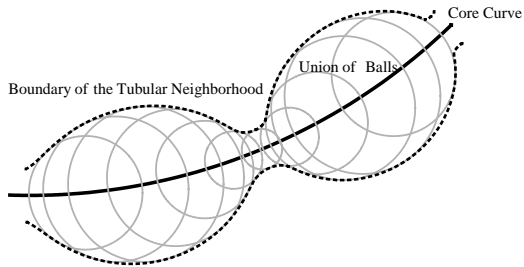}
\caption{{\protect\small A non-uniform }$\mu r${\protect\small -neighborhood
is shown as a union of balls of radii }$r\mu(s)${\protect\small \ centered at
}$\gamma(s)${\protect\small \ on the core }$\gamma${\protect\small .}}%
\end{center}
\end{figure}

We study the problem by using distance function methods from Riemannian
geometry. Throughout the article, we use the squared $\mu-$distance functions
$\left\Vert p-x\right\Vert ^{2}\mu(x)^{-2}$. We define the generalized
exponential function $\exp^{\mu}(q,Rv)=p$ to insure that $q$ is a critical
point of the restriction of $\left\Vert p-x\right\Vert ^{2}\mu(x)^{-2}$ to
$K.$ The image $\exp^{\mu}(NK_{q})$ is going to be a sphere normal to $K$ at
$q$ (with radius depending on $\mu$ where $\mu^{\prime}\neq0$) or a plane
(only where $\mu^{\prime}=0$) normal to $K$ at $q,$ where $NK_{q}$ denotes the
set of vectors normal to $K$ at $q.$

Even though there are many parallel results to the standard case $(\mu
\equiv1),$ we also observed many contrasting cases which never occur in the
standard case. In the standard case, the focal points occur at points
$p=\exp(q,Rv)$ where the first and the second derivatives of the restriction
of $E_{p}(x)=\left\Vert p-x\right\Vert ^{2}$ to $K$ are zero at $q$. The
second derivatives become negative immediately after the focal points as $R$
increases. Therefore, a line normal to $K$ is never minimizing the distance to
$K$ past a focal point, and the exponential map can not be injective past a
focal point. This is not always the case for nonconstant $\mu.$ First of all,
$\exp^{\mu}(q,Rv)$ is not always a line for a fixed point $q$ and a normal
vector $v$. Since there is a quadratic term $\frac{R^{2}}{2}(\mu^{2}%
)^{\prime\prime}$ in the second derivative of the restriction of $\left\Vert
p-x\right\Vert ^{2}\mu(x)^{-2}$ to $K,$ points with zero second derivatives
can be isolated away from the set of points with negative second derivatives.
As a result, there are some cases with an exponential map which is a
homeomorphism within the injectivity radius but not a diffeomorphism. In other
words, the injectivity radius can be larger than the $\mu-$distance to first
focal points. As a consequence, we need to modify the notion of injectivity
radius.%
\begin{figure}
[ptb]
\begin{center}
\includegraphics[
height=1.7634in,
width=2.93in
]%
{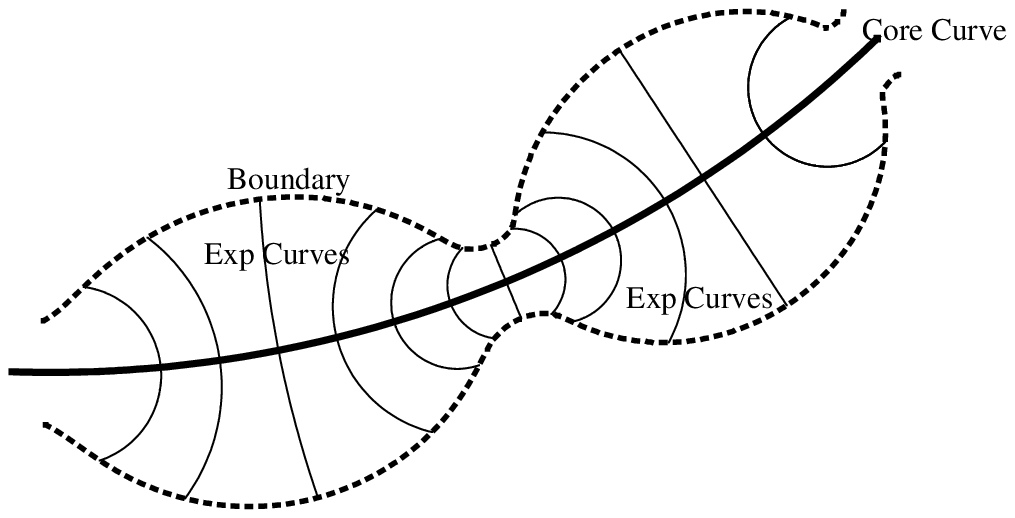}%
\caption{{} \ \ }%
\end{center}
\end{figure}
\begin{figure}
[ptb]
\begin{center}
\includegraphics[
height=1.7634in,
width=2.9291in
]%
{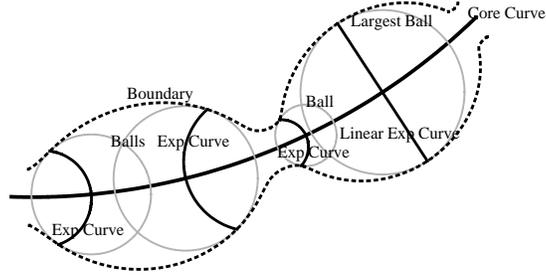}%
\caption{{\protect\small Some curves of type }$\exp^{\mu}(\gamma
(s_{i}),tN(s_{i}))${\protect\small \ for }$-r<t<r${\protect\small \ and for
some choices }$s_{i}${\protect\small \ are shown in the balls of radius }%
$r\mu(s_{i})${\protect\small \ and center }$\gamma(s_{i})${\protect\small ,
where }$N${\protect\small \ is the normal of }$\gamma\subset\mathbf{R}^{2}%
.${\protect\small \ Note the bending direction and the curvature of the
exponential curves in the balls of radius }$\mu r.$}%
\end{center}
\end{figure}

\begin{definition}
Let $K$ be a union of finitely many disjoint smoothly closed curves
in\ $\mathbf{R}^{n},$ $\mu:K\rightarrow(0,\infty)$ be a $C^{2}$ function, and
$\mathit{grad}\mu(q)$ be the gradient of $\mu$. Let $NK$ be the normal bundle
of $K$ in \ $\mathbf{R}^{n}$.
\begin{align*}
\text{Define }\exp^{\mu}  &  :W\rightarrow\mathbf{R}^{n}\text{ by}\\
\exp^{\mu}(q,w)  &  =q-\mu(q)\left\Vert w\right\Vert ^{2}\mathit{grad}%
\mu(q)+\mu(q)\sqrt{1-\left\Vert \mathit{grad}\mu(q)\right\Vert ^{2}\left\Vert
w\right\Vert ^{2}}w
\end{align*}%
\[
\text{where }W=\{w\in NK_{q}:q\in K\text{ and }\left\Vert w\right\Vert
\leq\frac{1}{\left\Vert \mathit{grad}\mu(q)\right\Vert }\text{ when
}\left\Vert \mathit{grad}\mu(q)\right\Vert \neq0\}.
\]

\end{definition}

Let $\gamma$ be a parametrization of $K$ locally with respect to arclength
$s.$ We use a standard abuse of notation $\mu(s)=\mu(\gamma(s)).$ We can take
the (intrinsic) gradient $\mathit{grad}\mu(\gamma(s))=\mu^{\prime}%
(s)\gamma^{\prime}(s)$, since $\mu$ is defined only on $K$ which is one
dimensional, see Definition 6 and Remark 1 for justifications. Hence, we can
rewrite $\exp^{\mu}$ as follows.
\[
\exp^{\mu}(\gamma(s),w)=\gamma(s)-\mu(s)\mu^{\prime}(s)\gamma^{\prime
}(s)\left\Vert w\right\Vert ^{2}+\mu(s)\sqrt{1-\left(  \mu^{\prime
}(s)\left\Vert w\right\Vert \right)  ^{2}}w
\]

\begin{definition}
Let $D(r)=\{(q,w)\in NK:q\in K$ and $\left\Vert w\right\Vert <r\}$.

i. The differentiable injectivity radius $DIR(K,\mu)$ is
\[
\sup\{r:\exp^{\mu}\text{ restricted to }D(r)\text{ is a diffeomorphism onto
its image}\}
\]

ii. The topological injectivity radius $TIR(K,\mu)$ is
\[
\sup\{r:\exp^{\mu}\text{ restricted to }D(r)\text{ is a homeomorphism onto its
image}\}
\]

iiii. The almost injectivity radius $AIR(K,\mu)$ is

$\sup\left\{
\begin{array}
[c]{c}%
r:\exp^{\mu}:U(r)\rightarrow U_{0}(r)\text{ is a homeomorphism where
}U(r)\text{ is an open }\\
\text{and dense subset of }D(r),\text{ and }U_{0}(r)\text{ is an open subset
of }\mathbf{R}^{n}.
\end{array}
\right\}  $
\end{definition}

%

\begin{figure}
[ptb]
\begin{center}
\includegraphics[
height=1.3733in,
width=3.6348in
]%
{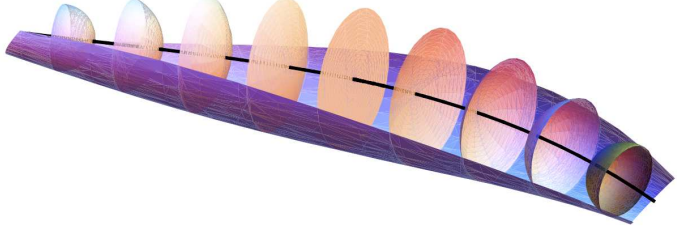}%
\caption{{\protect\small A 3-dimensional version of Figure 2. This shows some
spherical caps of type }$\exp^{\mu}(NK_{q}\cap D(r))$ {\protect\small normal
to }$K${\protect\small , in the }$\mu r${\protect\small -neighborhood, for
some choices of }$q${\protect\small \ on }$K${\protect\small . See Proposition
1.}}%
\end{center}
\end{figure}

Observe that $r<TIR(K,\mu)$ is equivalent to that for all $p\in O(K,\mu r)$
there exists a unique minimum of $\left\Vert p-x\right\Vert ^{2}\mu
(x)^{-2}:K\rightarrow\mathbf{R}$, i. e. there is a unique $\mu-$closest point
of $K$ to $p$. There are examples in $\mathbf{R}^{n}$ showing that
$DIR(K,\mu)<TIR(K,\mu)$ and $TIR(K,\mu)<AIR(K,\mu)$ in every dimension
$n\geq2$, see section 5. In the $\mu=1$ case, the injectivity radius
functional is upper semicontinuous in the $C^{1}$ topology. As a consequence,
thickest/tight/ideal knots and links exist, see [CKS], [D1], [D2], [GL], and
[N]. There are examples in $\mathbf{R}^{n}$ showing that $DIR(K,\mu)$ and
$TIR(K,\mu)$ are not upper semicontinuous, see Section 5. Hence,
thickest/tight/ideal knots and links in $DIR$ (or $TIR)$ sense may not exist.

The generalizations of the notion of double critical self distance, two
separate the notions of focal distance, $FocRad^{0}(K,\mu)$ and $FocRad^{-}%
(K,\mu),$ the upper and lower radii for the nonuniform $(K,\mu)$ will be given
immediately after Theorem 1. $FocRad^{-}$ and $FocRad^{0}$ are not necessarily
equal in general, due to certain "even" multiplicity zeroes of $\mu
^{\prime\prime}+\frac{1}{4}\kappa^{2}\mu=0.$ This difference allows
interesting examples mentioned above, which do not occur in the $\mu=1$
case.\pagebreak

\begin{theorem}
Let $K$ be a union of finitely many disjoint simple smoothly closed (possibly
linked or knotted) curves in $\mathbf{R}^{n}$. Then,

i. $LR(K,\mu)=DIR(K,\mu)\leq TIR(K,\mu)\leq AIR(K,\mu)=UR(K,\mu).$

ii. For a fixed choice of embedding $K\subset\mathbf{R}^{n},$ $LR(K,\mu
)=UR(K,\mu)$ holds for $\mu$ in an open and dense subset of $C^{3}%
(K,(0,\infty))$ in the $C^{3}-$ topology.

iii. Let $\left\{  (K_{i},\mu_{i}):i=1,2,...\right\}  $ be a sequence where
each $K_{i}$ is a disjoint union of finitely many simple smoothly closed
curves in $\mathbf{R}^{n}$ with $C^{2}$ weight functions, and similarly for
$(K_{0},\mu_{0})$. If $(K_{i},\mu_{i})\rightarrow(K_{0},\mu_{0})$ in $C^{2}$
topology, then
\[
\underset{i\rightarrow\infty}{\lim\sup}AIR(K_{i},\mu_{i})\leq AIR(K_{0}%
,\mu_{0}).
\]

\end{theorem}

\begin{definition}
A pair of points $(q_{1},q_{2})\in K\times K$ is called a double critical pair
for $(K,\mu)$, if $q_{1}\neq q_{2}$ and $\mathit{grad}\Sigma(q_{1},q_{2})=0,$
where $\Sigma:K\times K\rightarrow\mathbf{R}$ is defined by $\Sigma
(q_{1},q_{2})=\left\Vert q_{1}-q_{2}\right\Vert ^{2}(\mu(q_{1})+\mu
(q_{2}))^{-2}.$

By taking parametrizations $\gamma_{1},\gamma_{2}$ of $K$ locally with respect
to arclength $s,$ and $\sigma(s,t)=\left\Vert \gamma_{1}(s)-\gamma
_{2}(t)\right\Vert ^{2}(\mu(\gamma_{1}(s))+\mu(\gamma_{2}(t))^{-2}:$ (See
Definition 6.)%
\[
\mathit{grad}\Sigma(q_{1},q_{2})=0\Leftrightarrow\nabla\sigma(s_{1}%
,s_{2})=0,\text{ where }q_{i}=\gamma_{i}(s_{i})\text{ for }i=1,2.
\]
Double critical self $\mu-$distance of $(K,\mu)$ is defined as
\[
\frac{1}{2}DCSD(K,\mu)=\min\left\{  \frac{\left\Vert q_{1}-q_{2}\right\Vert
}{\mu(q_{1})+\mu(q_{2})}:(q_{1},q_{2})\text{ is a double critical pair for
}(K,\mu)\right\}  .
\]

\end{definition}

\begin{definition}
If $K$ is connected, by using a unit speed parametrization $\gamma
(s):\mathbf{R\rightarrow}K,$ such that $\gamma(s+L)=\gamma(s)$ where $L$ is
the length of $K$, $\mu(s)=\mu(\gamma(s)),$ and the curvature $\kappa(s)$ of
$\gamma(s),$ one defines $FocRad^{0}(K,\mu)$ to be%
\[
\left(  \max\left[
\begin{array}
[c]{c}%
\max\left\{
\begin{array}
[c]{c}%
\frac{1}{2}(\mu^{2})^{\prime\prime}+\frac{1}{2}\kappa^{2}\mu^{2}+\kappa
\mu\sqrt{\mu\left(  \mu^{\prime\prime}+\frac{1}{4}\kappa^{2}\mu\right)  }:\\
\text{where }\mu^{\prime\prime}+\frac{1}{4}\kappa^{2}\mu\text{ }\mathbf{\geq0}%
\end{array}
\right\}  ,\\
\max\left\{  \left\vert \mu^{\prime}\right\vert ^{2}:s\in Domain(\gamma
)\right\}
\end{array}
\right]  \right)  ^{-\frac{1}{2}}.
\]
$FocRad^{-}(K,\mu)$ is defined similarly by using the following expression
instead.
\[
\left(  \max\left[
\begin{array}
[c]{c}%
\sup\left\{
\begin{array}
[c]{c}%
\frac{1}{2}(\mu^{2})^{\prime\prime}+\frac{1}{2}\kappa^{2}\mu^{2}+\kappa
\mu\sqrt{\mu\left(  \mu^{\prime\prime}+\frac{1}{4}\kappa^{2}\mu\right)  }:\\
\text{where }\mu^{\prime\prime}+\frac{1}{4}\kappa^{2}\mu\text{ }\mathbf{>0}%
\end{array}
\right\}  ,\\
\max\left\{  \left\vert \mu^{\prime}\right\vert ^{2}:s\in Domain(\gamma
)\right\}
\end{array}
\right]  \right)  ^{-\frac{1}{2}}%
\]
If $K$ has several components $K_{i},$ $i=1,2,...i_{0}$, then $FocRad^{0}%
(K,\mu)$ is the minimum of $FocRad^{0}(K_{i},\mu)$ for $i=1,2,...i_{0},$ and
$FocRad^{-}(K,\mu)$ is the minimum of $FocRad^{-}(K_{i},\mu)$ for
$i=1,2,...i_{0}.$ The lower and upper radii \ are defined as follows:
\begin{align*}
LR(K,\mu)  &  =\min\left(  \frac{1}{2}DCSD(K,\mu),FocRad^{0}(K,\mu)\right) \\
UR(K,\mu)  &  =\min\left(  \frac{1}{2}DCSD(K,\mu),FocRad^{-}(K,\mu)\right)  .
\end{align*}

\end{definition}

If $\mu=1$, then $FocRad^{0}(K,1)=FocRad^{-}(K,1)=\left(  \max\kappa\right)
^{-1}$. Lemma 2 provides us the characterization of $DCSD$ in terms of the
angles that the line segment $\overline{q_{1}q_{2}}$ makes with $K$ at $q_{1}$
and $q_{2},$ generalizing the usual definition of $DCSD$ of the standard case
where $\mu=1$ and line segment $\overline{q_{1}q_{2}}$ is perpendicular to $K$
at both $q_{1}$ and $q_{2}.$%

\begin{figure}
[ptb]
\begin{center}
\includegraphics[
height=3.0675in,
width=2.7129in
]%
{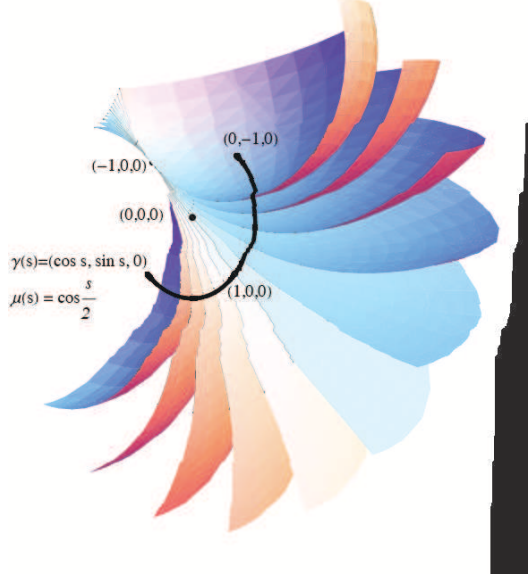}%
\caption{{\protect\small The normal exponential map from a portion of a unit
circle with }$\mu=\cos\frac{s}{2}${\protect\small \ in }$\mathbf{R}^{3}%
${\protect\small , showings some spherical caps of type }$\exp^{\mu}%
(NK_{q}\cap D(r))$ {\protect\small normal to }$K.${\protect\small \ See
Example 1B and Theorem 2. }}%
\end{center}
\end{figure}

We studied the properties of the singular $\exp^{\mu}$ maps within $AIR.$
Theorem 2 classifies all collapsing type singularities. If the injectivity of
$\exp^{\mu}$ fails within $UR(K,\mu)$ radius, that is if two distinct points
of $D(UR(K,\mu))$ are identified by $\exp^{\mu}$, then a curve of constant
height in $D(UR(K,\mu))$ joining the identified points collapses to the same
point under $\exp^{\mu}$. Figure 5 shows the unique way the injectivity of
$\exp^{\mu}$ fails within $UR(K,\mu),$ up to rescaling and isometries of
$\mathbf{R}^{3}$.

\begin{theorem}
\textbf{Horizontal Collapsing Property}

Assume that $\exp^{\mu}(q_{1},r_{1}v_{1})=\exp^{\mu}(q_{2},r_{2}v_{2})=p_{0}$
for $r_{1}$, $r_{2}<UR(K,\mu)$, $v_{i}\in UNK_{q_{i}}$ with $(q_{1},r_{1}%
v_{1})\neq(q_{2},r_{2}v_{2})$. Then,

(i) $q_{1}$ and $q_{2}$ belong to the same component of $K,$ which is denoted
by $K_{1}.$

(ii) Let $\gamma(s):\mathbf{R\rightarrow}K_{1}\subset\mathbf{R}^{n}$ be a unit
speed parametrization of $K_{1}$ such that $\gamma(s+L)=\gamma(s)$ where $L$
is the length of $K_{1},$ $N_{\gamma}(s)$ denotes the principal normal of
$\gamma,$ and $q_{i}=\gamma(s_{i})$ for $i=1,2$ with $0\leq s_{1}<s_{2}<L.$
Then, $r_{1}=r_{2},$ $v_{i}=N_{\gamma}(s_{i})$ for $i=1,2,$ and $\exp^{\mu
}(\gamma(s),r_{1}N_{\gamma}(s))=p_{0}$, $\forall s\in I$ where $I=[s_{1}%
,s_{2}]$ or $[s_{2}-L,s_{1}].$

(iii) On the interval $I$, $\kappa$ is a positive constant and all of the
following hold:
\begin{align*}
\left(  \mu^{\prime}\right)  ^{2}-\mu\mu^{\prime\prime}  &  =\frac{1}%
{r_{1}^{2}}\text{ and }\gamma^{\prime\prime\prime}+\kappa^{2}\gamma^{\prime
}=0,\\
\mu &  =\frac{2}{\kappa r_{1}}\cos\left(  \frac{\kappa s}{2}+a\right)  \text{
for some }a\in\mathbf{R.}%
\end{align*}
Therefore, Horizontal Collapsing occurs in a unique way only above arcs of
circles of curvature $\kappa$ and with a specific $\mu.$ $\gamma(I)\neq K_{1}%
$, even if $I$ is chosen to be a maximal interval satisfying above.
\end{theorem}

As a consequence, we can obtain $TIR(K,\mu)$ in terms of $\mu,\kappa,$ and
$\frac{1}{2}DCSD(K,\mu).$ Theorems 2 and 3 give us a complete understanding of
the differences between $DIR,TIR$ and $AIR.$

\begin{theorem}
Let $K$ be a union of finitely many disjoint simple smoothly closed (possibly
linked or knotted) curves in $\mathbf{R}^{n}$. Let $\gamma:Domain(\gamma
)\rightarrow K$ parametrize $K$ with unit speed and $\mu(s)=\mu(\gamma(s))$.
If $TIR(K,\mu)<UR(K,\mu)$, then $K$ contains a circular arc of curvature
$\kappa$ and positive length, along which $\mu=\frac{2}{\kappa r}\cos\left(
\frac{\kappa s}{2}+a\right)  $ for some $a\in\mathbf{R}$ and $r<UR(K,\mu).$ In
this case, $TIR(K,\mu)$ is equal to the infimum of such $r.$

If $K$ has no such circular arc with a compatible $\mu$, that is, the set%
\[
\left\{
\begin{array}
[c]{c}%
s\in Domain(\gamma):\left(  \mu^{\prime\prime}+\frac{1}{4}\kappa^{2}%
\mu\right)  (s)=0,\text{ and }\kappa^{\prime}(s)=0\text{ with }\kappa
(s)>0,\text{ and }\\
\gamma^{\prime\prime\prime}(s)+\kappa^{2}(s)\gamma^{\prime}(s)=0\text{ and
}\left(  \mu^{\prime}\right)  ^{2}(s)-\mu\mu^{\prime\prime}(s)=\frac{1}{r^{2}%
}\in\mathbf{R}\text{ where }r<UR(K,\mu)\text{.}%
\end{array}
\right\}
\]
has no interior, then $TIR(K,\mu)=AIR(K,\mu)=UR(K,\mu).$
\end{theorem}

The following theorem summarizes the remaining results obtained in the course
of proving the theorems above, the exact structure of the singular set of
$\exp^{\mu}$ within $UR(K,\mu),$ as well as the structure of the set of
regular points.

\begin{theorem}
Let $K_{i}$ denote the components of $K.$ Let $\gamma_{i}:domain(\gamma
_{i})\rightarrow K_{i}$ be an onto parametrization of the component $K_{i}$
with unit speed and $\mu_{i}(s)=\mu(\gamma_{i}(s)).$ Then, the singular set
$Sng^{NK}(K,\mu)$ of $\exp^{\mu}$ within $D(UR(K,\mu))\subset NK$ is a graph
over a portion of $K$:%
\begin{align*}
Sng^{NK}(K,\mu)  &  =%
{\textstyle\bigcup\nolimits_{i}}
Sng_{i}^{NK}(K,\mu)\text{ and}\\
Sng_{i}^{NK}(K,\mu)  &  =\left\{
\begin{array}
[c]{c}%
(\gamma_{i}(s),R_{i}(s)N_{\gamma_{i}}(s))\in NK_{i}\text{ where}\\
s\in domain(\gamma_{i}),\text{ }\kappa_{i}(s)>0,\text{ }\\
\left(  \mu_{i}^{\prime\prime}+\frac{1}{4}\kappa_{i}^{2}\mu_{i}\right)
(s)=0\text{, and }\\
0<R_{i}(s)=\left(  \left(  \mu_{i}^{\prime}\right)  ^{2}-\mu_{i}\mu
_{i}^{\prime\prime}\right)  (s)^{-\frac{1}{2}}<UR(K,\mu)
\end{array}
\right\}
\end{align*}
where $\kappa_{i}$ and $N_{\gamma_{i}}$ are the curvature and the principal
normal of $\gamma_{i}$, respectively$.$ $D(UR(K,\mu))-Sng^{NK}(K,\mu)$ is
connected in each component of $NK,$ when $n\geq2.$ Let
\begin{align*}
Sng(K,\mu)  &  =\exp^{\mu}(Sng^{NK}(K,\mu)),\\
A_{q}  &  =\exp^{\mu}\left(  NK_{q}\cap D(UR(K,\mu))\right)  ,\text{ and}\\
A_{q}^{\ast}  &  =\exp^{\mu}\left(  NK_{q}\cap D(UR(K,\mu))-Sng^{NK}%
(K,\mu)\right)  .
\end{align*}

i. $O(K,\mu UR(K,\mu))-Sng(K,\mu)$ has a codimension 1 foliation by
$A_{q}^{\ast}$, which are (possibly punctured) spherical caps or discs.$\ $

ii. $\exp^{\mu}(D(UR(K,\mu))-Sng^{NK}(K,\mu))=O(K,\mu UR(K,\mu))-Sng(K,\mu).$

iii. If $A_{q_{1}}\cap A_{q_{2}}\neq\varnothing$ for $q_{1}\neq q_{2}$ then
$q_{1}$ and $q_{2}$ must belong to the same component of $K,$ and $A_{q_{1}}$
intersects $A_{q_{2}}$ tangentially at exactly one point $p_{0}=\exp^{\mu
}(q_{1},r_{1}v_{1})=\exp^{\mu}(q_{2},r_{2}v_{2})$ where $(q_{i},r_{i}v_{i})\in
Sng^{NK}(K,\mu),$ for $i=1,2.$
\end{theorem}

The remaining definitions and notation are given in Section 2. The first and
second order analysis of the $\mu-$distance functions, and basic properties of
$\exp^{\mu}$ are studied in Section 3. Section 4 contains the proofs involving
$DIR$ and $TIR.$ Section 5 has several examples shoving the deviation from the
standard $\mu=1$ case. $AIR$ \ and Horizontal Collapsing Property are studied
in Section 6 after the examples which give the motivation for many proofs.

\section{Further Notation and Definitions}

We assume that $K$ is a union of finitely many disjoint simple smoothly closed
(possibly linked or knotted) curves in $\mathbf{R}^{n}$. Hence, $K$ is a
$1-$dimensional compact submanifold of $\mathbf{R}^{n},$ with finitely many
components. All parametrizations $\gamma:I\rightarrow K$ are with respect to
arclength $s$ and $C^{3},$ unless it is indicated otherwise.\ All
$\mu:K\rightarrow(0,\infty)$ are at least $C^{3}$. For some compactness
arguments on a $K$, we may take $Domain(\gamma)$ to be a disjoint union of
$\mathbf{R}$ $/$ $Length(K_{i})\mathbf{Z}$ by considering $\gamma$ as periodic
function of period $length(K_{i})$ on each component $K_{i}.$

\begin{notation}
$TK$ and $NK$ denote the tangent and normal bundles of $K$ in\ $\mathbf{R}%
^{n}$, respectively$.$ $UTK$ and $UNK$ denote the unit vectors, $NK_{q}$
denotes the set normal vectors to $K$ at $q$, and similarly for the others.
For $v\in T\mathbf{R}_{q}^{n}=TK_{q}\oplus NK_{q},$ $v^{T}$ and $v^{N}$ denote
the tangential and normal components of $v$ to $K$, respectively. $D(r)$
denotes $\{(q,w)\in NK:q\in K$ and $\left\Vert w\right\Vert <r\}.$
\end{notation}

\begin{notation}
i. We use the standard distance function $d(p,q)=\left\Vert p-q\right\Vert $
in $\mathbf{R}^{n}$. $B(p,r)$ and $\bar{B}(p,r)$ denote open and closed metric
balls. For $A\subset\mathbf{R}^{n}$, $B(A,r)=\{x\in X:d(x,A)<r\}.$

ii. The unit direction vector from $q$ to $p$ is $u(q,p)=\frac{p-q}{\left\Vert
p-q\right\Vert }$ for $p\neq q.$
\end{notation}

\begin{definition}
Let $K\subset\mathbf{R}^{n}$ and $\mu:K\rightarrow(0,\infty)$ be given. We define:

i. The $\mu R$ neighborhood of $K$, $O(K,\mu R)=%
{\displaystyle\bigcup\nolimits_{q\in K}}
B(q,\mu(q)R),$

ii. For $p\in\mathbf{R}^{n},$

$E_{p}:K\rightarrow\mathbf{R}$ by $E_{p}(x)=\left\Vert p-x\right\Vert ^{2},$

$F_{p}:K\rightarrow\mathbf{R}$ by $F_{p}(x)=\left\Vert p-x\right\Vert ^{2}%
\mu(x)^{-2},$ the square of the $\mu-$distance function from $p,$

$F_{p}^{c}:K\rightarrow\mathbf{R}$ by $F_{p}^{c}(x)=\left\Vert p-x\right\Vert
^{2}(\mu(x)+c)^{-2},$

$G:\mathbf{R}^{n}\rightarrow\mathbf{R}$ by $G(p)=\min_{x\in K}F_{p}(x)$ so
that $O(K,\mu R)=G^{-1}([0,R^{2})),$ and

$\Sigma:K\times K\rightarrow\mathbf{R}$ by $\Sigma(x,y)=\left\Vert
x-y\right\Vert ^{2}(\mu(x)+\mu(y))^{-2},$
\end{definition}

\begin{notation}
For a local parametrization $\gamma:I\rightarrow K$ with respect to arclength
$s,$ we will identify $\mu(s)=\mu(\gamma(s))$, $F_{p}(s)=F_{p}(\gamma
(s))=\left\Vert p-\gamma(s)\right\Vert ^{2}\mu(\gamma(s))^{-2}$, and similarly
for all functions above. We use $s\in\mathbf{R,}$ and $x$ or $q\in K$ to avoid ambiguity.
\end{notation}

\begin{definition}
For a $C^{1}$ function $\mu:K\rightarrow(0,\infty)$, $\operatorname{grad}\mu$
denotes the intrinsic gradient field of $\mu$, that is the unique vector field
tangential to $K$ such that for every tangent vector $v\in TK_{q},$ the
directional derivative of $\mu$ at $q$ in the direction $v$ along $K$ is
$v\cdot\left(  \operatorname{grad}\mu\right)  (q).$ For every $C^{1}$
extension\thinspace$\widetilde{\mu}$ of $\mu$ to an open subset of
$\mathbf{R}^{n}$, containing $q,$ one has $\left(  \operatorname{grad}%
\mu\right)  (q)=\left(  \nabla\widetilde{\mu}(q)\right)  ^{T}$ where $\nabla$
denotes the usual gradient in $\mathbf{R}^{n}$ defined by using the partial
derivatives in $\mathbf{R}^{n}.$ See [T], p. 96. Since $K$ is one dimensional,
one has
\[
\left(  \operatorname{grad}\mu\right)  (\gamma(s))=\mu^{\prime}(\gamma
(s))\gamma^{\prime}(s)=\mu^{\prime}(s)\gamma^{\prime}(s)
\]
for a parametrization $\gamma$ with respect to arclength.
\end{definition}

\begin{remark}
The last line above is justified by the Chain Rule:
\begin{align*}
\mu^{\prime}(s)  &  =\frac{d}{ds}\mu(\gamma(s))=\frac{d}{ds}\widetilde{\mu
}(\gamma(s))=\nabla\widetilde{\mu}(\gamma(s))\cdot\gamma^{\prime}(s)=\left(
\nabla\widetilde{\mu}(\gamma(s))\right)  ^{T}\cdot\gamma^{\prime}(s)\\
&  =\left(  \operatorname{grad}\mu\right)  (\gamma(s))\cdot\gamma^{\prime}(s).
\end{align*}

\end{remark}

\begin{remark}
For a given parametrization $\gamma$ of $K$ with respect to arclength,
$\mu^{\prime\prime}(s_{0})$, $\gamma^{\prime\prime}(s_{0}),$ $\left(
\mu^{\prime}(s_{0})\right)  ^{2}$, $\left\Vert \mathit{grad}\mu(q)\right\Vert
$ and $F_{p}^{\prime\prime}(s_{0})$ are calculated at $q=\gamma(s_{0})$ by
using the given parametrization. However, all of these quantities depend only
on $K,\mu$ and $q,$ but not on the choice of the parametrization with respect
to arclength. Observe that when one reverses the orientation of a
parametrization, both $\mu^{\prime}$ and $\gamma^{\prime}$ change signs at
$q$. $\operatorname{grad}\mu(q)$ and $\left\Vert \mathit{grad}\mu
(q)\right\Vert $ are both well-defined. Although the sign of $\mu^{\prime}(q)$
is ambiguous, depending on the orientation of $\gamma,$ we can use $\left\vert
\mu^{\prime}(q)\right\vert =\left\Vert \mathit{grad}\mu(q)\right\Vert $. If
$\mathit{grad}\mu(q)=0,$ then $\left\Vert \mathit{grad}\mu(q)\right\Vert
^{-1}$ is taken to be $+\infty$. The definitions given in Section 1,
exponential map, focal radii, double critical self distance by using by using
a parametrization, are independent of the choice of the parametrization.
\end{remark}

\begin{notation}
For any function $f:X\rightarrow Y$ and $Z\subset X,$ $f\mid Z$ is the
restriction of $f$ to $Z.$
\end{notation}

\begin{definition}
Let $\gamma:I\rightarrow K\subset\mathbf{R}^{n}$, $\mu:K\rightarrow(0,\infty
)$, $p\in\mathbf{R}^{n}$ and $q=\gamma(s_{0})\in K$ be given.

$q\in CP(p)$, if $q$ is a critical point of $F_{p}(x),$ that is $F_{p}%
^{\prime}(s_{0})=0$,

$q\in CP(p,+),$ if $F_{p}^{\prime}(s_{0})=0$ and $F_{p}^{\prime\prime}%
(s_{0})>0,$

$q\in CP(p,0),$ if $F_{p}^{\prime}(s_{0})=0$ and $F_{p}^{\prime\prime}%
(s_{0})=0,$

$q\in CP(p,-),$ if $F_{p}^{\prime}(s_{0})=0$ and $F_{p}^{\prime\prime}%
(s_{0})<0.$
\end{definition}

\begin{definition}
The radius of regularity is
\[
\text{\textit{RegRad}}(K,\mu)=\sup\{r:\exp^{\mu}\text{ restricted to
}D(r)\text{ is a non-singular }C^{1}\text{ map}\}.
\]

\end{definition}

\section{Basic Properties of $exp^{\mu}$}

\begin{remark}
If $f(s)=\frac{E(s)}{g(s)},$ then by logarithmic differentiation
$\frac{f^{\prime}}{f}=\frac{E^{\prime}}{E}-\frac{g^{\prime}}{g}.$

If $f^{\prime}(s_{0})=0,$ then\ $\frac{E^{\prime}}{E}(s_{0})=\frac{g^{\prime}%
}{g}(s_{0})$ and $\frac{f^{\prime\prime}}{f}(s_{0})=\left(  \frac
{E^{\prime\prime}}{E}-\frac{g^{\prime\prime}}{g}\right)  (s_{0}).$
\end{remark}

\begin{notation}
For $q\in K$ and $p\in\mathbf{R}^{n}-\{q\}:$

$\alpha(q,p)=\measuredangle(\mathit{grad}\mu(q),u(q,p))$ when $\mathit{grad}%
\mu(q)\neq0,$ and

$\alpha(q,p)=\frac{\pi}{2}$ when $\mathit{grad}\mu(q)=0$.
\end{notation}

\begin{lemma}
For $q\in K$ and $p\in\mathbf{R}^{n}-\{q\}$, and $c\in\lbrack0,\infty),$%
\[
q\text{ is a critical point of }F_{p}^{c}(x)\iff u(q,p)^{T}=-\frac{\left\Vert
p-q\right\Vert \mathit{grad}\mu(q)}{\mu(q)+c}.
\]
If $q$ is a critical point of $F_{p}^{c}(x),$ then
\[
\cos\alpha(q,p)=-\frac{\left\Vert p-q\right\Vert \left\Vert \mathit{grad}%
\mu(q)\right\Vert }{\mu(q)+c}\text{ and hence }\frac{\pi}{2}\leq
\alpha(q,p)\leq\pi.
\]

\end{lemma}

\begin{proof}
For a given $\gamma:I\rightarrow K$ with $q=\gamma(s_{0})$, $v=\gamma^{\prime
}(s_{0}),$ and $E(s)=\left\Vert p-\gamma(s)\right\Vert ^{2},$ one has
$E^{\prime}(s_{0})=2\left(  p-\gamma(s_{0})\right)  \cdot(-\gamma^{\prime
}(s_{0}))=2\left(  p-q\right)  \cdot(-v).$ If $q$ is a critical point of
$F_{p}^{c}(x),$ then $s_{0}$ is a critical point of
\[
F_{p}^{c}(\gamma(s))=\left\Vert p-\gamma(s)\right\Vert ^{2}(\mu(s))+c)^{-2}%
=E(s)(\mu(s))+c)^{-2}.
\]
By Remark 3:%
\begin{align*}
\frac{2\left(  p-q\right)  \cdot(-v)}{\left\Vert p-q\right\Vert ^{2}}  &
=\frac{E^{\prime}}{E}(s_{0})=\frac{\left(  (\mu(s))+c)^{2}\right)  ^{\prime}%
}{(\mu(s))+c)^{2}}(s_{0})=\frac{2\mu^{\prime}(s_{0})}{\mu(s_{0})+c}\\
-2u(q,p)\cdot v  &  =\left\Vert p-q\right\Vert \frac{2\mu^{\prime}(s_{0})}%
{\mu(s_{0})+c}=\left\Vert p-q\right\Vert \frac{2\mu^{\prime}(s_{0})v}%
{\mu(s_{0})+c}\cdot v\\
u(q,p)\cdot v  &  =-\left\Vert p-q\right\Vert \frac{\mathit{grad}\mu(q)}%
{\mu(q)+c}\cdot v\\
u(q,p)^{T}  &  =-\frac{\left\Vert p-q\right\Vert \mathit{grad}\mu(q)}%
{\mu(q)+c}%
\end{align*}

This argument is reversible for the converse. The statement for $\cos\alpha$
is obvious when $\mathit{grad}\mu(q)=0=u(q,p)^{T}.$ In the other case, we have
the following.%
\begin{align*}
\left\Vert \mathit{grad}\mu(q)\right\Vert \cos\alpha(q,p)  &  =u(q,p)\cdot
\mathit{grad}\mu(q)\\
&  =-\left\Vert p-q\right\Vert \frac{\mathit{grad}\mu(q)}{\mu(q)+c}%
\cdot\mathit{grad}\mu(q)\\
&  =-\frac{\left\Vert p-q\right\Vert \left\Vert \mathit{grad}\mu(q)\right\Vert
^{2}}{\mu(q)+c}%
\end{align*}

\end{proof}

\begin{proposition}
i. $p=\exp^{\mu}(q,w)$ if and only if

$\left\{
\begin{tabular}
[c]{ll}%
$q\in CP(p)$, $\left\Vert p-q\right\Vert =\left\Vert w\right\Vert \mu(q),$ and
$w=R\frac{u(q,p)^{N}}{\left\Vert u(q,p)^{N}\right\Vert }$ & when
$u(q,p)^{N}\neq0$\\
$q\in CP(p),$ and ($R=0$ or $R=\left\Vert \mathit{grad}\mu(q)\right\Vert
^{-1}$) & when $u(q,p)^{N}=0.$%
\end{tabular}
\ \ \right.  $

ii. If $p=\exp^{\mu}(q,Rv)$ for a unit vector $v$ and $R>0,$ then
\[
F_{p}(q)=R^{2}\text{ and }\cos\alpha(q,p)=-R\left\Vert \mathit{grad}%
\mu(q)\right\Vert =-\left\Vert u(q,p)^{T}\right\Vert \text{ and}%
\]%
\[
\exp^{\mu}(q,Rv)=\left\{
\begin{array}
[c]{cc}%
q+\mu(q)R\left(  \cos\alpha(q,p)\frac{\mathit{grad}\mu(q)}{\left\Vert
\mathit{grad}\mu(q)\right\Vert }+\sin\alpha(q,p)v\right)  & \text{if
}\mathit{grad}\mu(q)\neq0\\
q+\mu(q)Rv & \text{if }\mathit{grad}\mu(q)=0
\end{array}
\right.
\]

iii. $\exp^{\mu}:W\rightarrow\mathbf{R}^{n}$ is an onto map, where

$W=\{w\in NK_{q}:q\in K$ and $\left\Vert w\right\Vert \leq\left\Vert
\mathit{grad}\mu(q)\right\Vert ^{-1}$ when $\left\Vert \mathit{grad}%
\mu(q)\right\Vert \neq0\}.$

iv. $\exp^{\mu}$ is $C^{1}$ on the interior of $W$ and the differential
$d(\exp^{\mu})(q,\mathbf{0})=\mu(q)Id.$ Consequently, there exists
$\varepsilon>0,$ such that $\exp^{\mu}$ is a diffeomorphism on $\{w\in
NK_{q}:q\in K$ and $\left\Vert w\right\Vert <\varepsilon\}$ by the Inverse
Function Theorem.

v. If $\mathit{grad}\mu(q)=0,$ then $exp^{\mu}(NK_{q})$ is a $(n-1)-$%
dimensional plane normal to $K$ at $q.$ If $\mathit{grad}\mu(q)\neq0,$ then
$exp^{\mu}(NK_{q}\cap W)$ is a $(n-1)-$dimensional sphere normal to $K$ at
$q,$ with the radius $\frac{1}{2}\frac{\mu(q)}{\left\Vert \mathit{grad}%
\mu(q)\right\Vert }$ and the center at $q-\frac{1}{2}\frac{\mu(q)\mathit{grad}%
\mu(q)}{\left\Vert \mathit{grad}\mu(q)\right\Vert ^{2}}.$

vi. If $\mathit{grad}\mu(q)\neq0,$ then $exp^{\mu}(NK_{q}\cap W)\cap K$ has a
least two distinct points. Consequently, $TIR(K,\mu)<\frac{1}{\max_{q\in
K}\left\Vert \mathit{grad}\mu(q)\right\Vert }$.
\end{proposition}

\begin{proof}
i. Assume that $p=\exp^{\mu}(q,w)$ for some $w\in NK_{q}$. $\mathit{grad}%
\mu(q)\in TK_{q}$ and $w\in NK_{q}.$
\begin{align*}
p-q  &  =-\mu(q)\left\Vert w\right\Vert ^{2}\mathit{grad}\mu(q)+\mu
(q)\sqrt{1-\left\Vert \mathit{grad}\mu(q)\right\Vert ^{2}\left\Vert
w\right\Vert ^{2}}w\\
\left\Vert p-q\right\Vert  &  =\mu(q)\left\Vert w\right\Vert \\
u(q,p)^{T}  &  =\left(  \frac{p-q}{\left\Vert p-q\right\Vert }\right)
^{T}=-\left\Vert w\right\Vert \mathit{grad}\mu(q)=-\frac{\left\Vert
p-q\right\Vert \mathit{grad}\mu(q)}{\mu(q)}%
\end{align*}
By Lemma 1, we conclude that $q\in CP(p).$

For the converse, assume that $q$ is a critical point of $F_{p}(x)$ for some
$p\in\mathbf{R}^{n}$ and $\left\Vert p-q\right\Vert =R\mu(q)$ for some $R.$

If $R=0,$ then $p=q=\exp^{\mu}(q,0).$

Suppose that $R>0$. By Lemma 1 for $c=0$, one obtains that
\begin{align*}
u(q,p)^{T}  &  =-\frac{\left\Vert p-q\right\Vert \mathit{grad}\mu(q)}{\mu
(q)}=-R\mathit{grad}\mu(q)\\
\cos\alpha(q,p)  &  =-R\left\Vert \mathit{grad}\mu(q)\right\Vert =-\left\Vert
u(q,p)^{T}\right\Vert \geq-1\\
\sin\alpha(q,p)  &  =\sqrt{1-\left\Vert \mathit{grad}\mu(q)\right\Vert
^{2}R^{2}}=\left\Vert u(q,p)^{N}\right\Vert .
\end{align*}

If $\sin\alpha(q,p)>0,$ then one takes $w=R\frac{u(q,p)^{N}}{\left\Vert
u(q,p)^{N}\right\Vert }$ so that $R=\left\Vert w\right\Vert $ and
\begin{align*}
p-q  &  =R\mu(q)u(q,p)=R\mu(q)\left(  u(q,p)^{T}+u(q,p)^{N}\right) \\
&  =-R^{2}\mu(q)\mathit{grad}\mu(q)+\mu(q)\left\Vert u(q,p)^{N}\right\Vert w\\
&  =\exp^{\mu}(q,w)-q.
\end{align*}

If $\sin\alpha(q,p)=0,$ then $\cos\alpha(q,p)=-1=-R\left\Vert \mathit{grad}%
\mu(q)\right\Vert .$%
\begin{align*}
u(q,p)  &  =u(q,p)^{T}=-\frac{\mathit{grad}\mu(q)}{\left\Vert \mathit{grad}%
\mu(q)\right\Vert }\\
p  &  =q+\left\Vert p-q\right\Vert u(p,q)=q-R\mu(q)\frac{\mathit{grad}\mu
(q)}{\left\Vert \mathit{grad}\mu(q)\right\Vert }=q-R^{2}\mu(q)\mathit{grad}%
\mu(q)\\
&  =\exp^{\mu}(q,Rv),\forall v\in UNK_{q}%
\end{align*}

ii. This follows the proof of (i).

iii. For every $p\in\mathbf{R}^{n},~$the continuous map $F_{p}:K\rightarrow
\mathbf{R}$ must have a minimum on compact $K$, and hence it has a critical
point $q\in K.$ By the construction in (i), $p=\exp^{\mu}(q,w)$ for some $w\in
NK_{q},$ and $\left\Vert w\right\Vert =R\leq\left\Vert \mathit{grad}%
\mu(q)\right\Vert ^{-1}.$

iv. $\exp^{\mu}(q,w)=q-\mu(q)\left\Vert w\right\Vert ^{2}\mathit{grad}%
\mu(q)+\mu(q)\sqrt{1-\left\Vert \mathit{grad}\mu(q)\right\Vert ^{2}\left\Vert
w\right\Vert ^{2}}w$ is $C^{1}$ except when $\left\Vert \mathit{grad}%
\mu(q)\right\Vert \left\Vert w\right\Vert =1.$ For a fixed $q\in K$, $v\in
UNK_{q}$ and taking $w=Rv,$%
\begin{align*}
&  \frac{d}{dR}\left.  \exp^{\mu}(q,Rv)\right\vert _{R=0}\\
&  =\frac{d}{dR}\left.  \left(  q-\mu(q)R^{2}\mathit{grad}\mu(q)+\mu
(q)\sqrt{1-\left\Vert \operatorname{grad}\mu(q)\right\Vert ^{2}R^{2}%
}vR\right)  \right\vert _{R=0}\\
&  =\mu(q)v
\end{align*}

v. $exp^{\mu}(NK_{q})$ is a $(n-1)-$dimensional is a plane normal to $K$ at
$q$ when $\mathit{grad}\mu(q)=0$ by the definition of $\exp^{\mu}$.

Assume that $\mathit{grad}\mu(q)\neq0,$\thinspace and choose an arbitrary
$v\in UNK_{q}$. For every $p=\exp^{\mu}(q,Rv),$ where $0\leq R\leq\left\Vert
\mathit{grad}\mu(q)\right\Vert ^{-1},$%
\begin{align*}
\cos(\pi-\alpha(q,p))  &  =R\left\Vert \mathit{grad}\mu(q)\right\Vert
=\frac{\left\Vert p-q\right\Vert }{\mu(q)}\left\Vert \mathit{grad}%
\mu(q)\right\Vert \\
\left\Vert p-q\right\Vert  &  =\frac{\mu(q)}{\left\Vert \mathit{grad}%
\mu(q)\right\Vert }\cos(\pi-\alpha(q,p))
\end{align*}
where $\mu(q)\left\Vert \mathit{grad}\mu(q)\right\Vert ^{-1}$ does not depend
on $p.$ This is an equation of a semi-circle in the polar coordinates of the
2-plane passing through $q$ and parallel to $\mathit{grad}\mu(q)$ and $v,$
where $q$ is the origin, $\theta$ is angle from $-\mathit{grad}\mu
(q)\left\Vert \mathit{grad}\mu(q)\right\Vert ^{-1}$ turning towards $v$, and
$r=\left\Vert p-q\right\Vert .$ The radius of the circle is $\frac{1}{2}%
\mu(q)\left\Vert \mathit{grad}\mu(q)\right\Vert ^{-1}$, the center is at
$q-\frac{1}{2}\mu(q)\mathit{grad}\mu(q)\left\Vert \mathit{grad}\mu
(q)\right\Vert ^{-2}$, and the circle is tangent to $v$ at $q.$ Since the
center and the radius depend only on $q$ and not on $v,$ one concludes that
$exp^{\mu}(NK_{q}\cap W)$ is a $(n-1)-$dimensional sphere normal to $K$ at
$q.$

vi. Intuitively, since $K$ goes into $\exp^{\mu}(NK_{q}\cap W)$ (an $(n-1)-
$dimensional plane sphere in $\mathbf{R}^{n}$) transversally at $q,$ it has to
come out of it somewhere else. By using the mod-2 intersection theory [G],
page 77, the mod 2 intersection number of $K$ and $\exp^{\mu}(NK_{q}\cap W)$
must be zero mod 2, since one can isotope two compact submanifolds away from
each other in $\mathbf{R}^{n}.$ Since $q\in\exp^{\mu}(NK_{q}\cap W)$, and the
intersection of $K$ and $\exp^{\mu}(NK_{q}\cap W)$ is transversal at $q$, the
number of points in $K\cap\exp^{\mu}(NK_{q}\cap W)$ is more than 1. For
another point $q^{\prime}\in K\cap\exp^{\mu}(NK_{q}\cap W),$ and for every
open neighborhood $U$ of $q^{\prime}$ in $K$ with $q\notin U,$ $\exp^{\mu
}(\{(y,w)\in NK:y\in U$ and $\left\Vert w\right\Vert <\varepsilon\})$
intersects $\exp^{\mu}(NK_{q}\cap W)$ along an open subset. The injectivity of
$\exp^{\mu}$ must fail strictly before reaching $q^{\prime}$ and the antipodal
point of $q$ in $\exp^{\mu}(NK_{q}\cap W)$, that is when $R=\left\Vert
\mathit{grad}\mu(q)\right\Vert ^{-1}$.
\end{proof}

\begin{corollary}
By the proof of Proposition 1 (iii), for every $p\in O(K,\mu R),$ there exists
$q\in K$ and $v\in UNK_{q}$ such that $p=\exp^{\mu}(q,rv)$ for some
$r=\sqrt{G(p)}<R.$ Consequently, $\exp^{\mu}(D(R))=O(K,\mu R)=G^{-1}%
([0,R^{2})),$ for all $R>0.$
\end{corollary}

\begin{lemma}
i. $(q_{1},q_{2})$ is a double critical pair for $(K,\mu)$ if and only if
there exists $R>0$ and $p$ on the line segment joining $q_{1}$ and $q_{2}$
such that $\left\Vert p-q_{i}\right\Vert =R\mu(q_{i})$ and $p=\exp^{\mu}%
(q_{i},Rv_{i})$ with $v_{i}\in UNK_{q_{i}}$ for $i=1$ and $2.$ Consequently,
$(q_{1},q_{2})$ is a double critical pair for $(K,\mu)$ if and only if
$q_{1},q_{2}\in CP(p)$ and $F_{p}(q_{1})=F_{p}(q_{2})>0.$

ii. If $(q_{1},q_{2})$ is a double critical pair for $(K,\mu),$ then for $i=1$
and $2,$
\[
\cos\alpha(q_{i},p)=-\frac{\left\Vert q_{1}-q_{2}\right\Vert \left\Vert
\mathit{grad}\mu(q_{i})\right\Vert }{\mu(q_{1})+\mu(q_{2})}=\frac{\left\Vert
p-q_{i}\right\Vert \left\Vert \mathit{grad}\mu(q_{i})\right\Vert }{\mu(q_{i}%
)}=-R\left\Vert \mathit{grad}\mu(q_{i})\right\Vert .
\]

\end{lemma}

\begin{proof}
Assume that $(q_{1},q_{2})$ is a double critical pair for $(K,\mu)$ and take
$R=\frac{\left\Vert q_{1}-q_{2}\right\Vert }{\mu(q_{1})+\mu(q_{2})}.$ There
exists a unique $p$ on the line segment joining $q_{1}$ and $q_{2}$ such that
$\left\Vert p-q_{i}\right\Vert =R\mu(q_{i})$ for $i=1$ and $2.$ Let $q_{2}$ be
fixed. $\mathit{grad}\Sigma(x,q_{2})\mid_{x=q_{1}}=0,$ that is $q_{1}$ is a
critical point of $\left(  \frac{\left\Vert x-q_{2}\right\Vert }{\mu
(x)+\mu(q_{2})}\right)  ^{2}=F_{q_{2}}^{\mu(q_{2})}(x).$ By Lemma 1,
\begin{align*}
u(q_{1},p)^{T}  &  =u(q_{1},q_{2})^{T}=-\frac{\left\Vert q_{1}-q_{2}%
\right\Vert \mathit{grad}\mu(q_{1})}{\mu(q_{1})+\mu(q_{2})}\\
&  =-R\mathit{grad}\mu(q_{1})=-\frac{\left\Vert q_{1}-p\right\Vert
\mathit{grad}\mu(q_{1})}{\mu(q_{1})}%
\end{align*}
and consequently $q_{1}\in CP(p).$ By Proposition 1, $p=\exp^{\mu}%
(q_{1},Rv_{1})$ for some $v_{1}\in UNK_{q_{1}}$. The $q_{2}$ case is similar.
This argument is reversible for the converse. The second statement of (i) and
(ii) are straightforward by using Lemma 1.
\end{proof}

\begin{lemma}
Let $A,B,C\in\mathbf{R}$ \ with $A,B\geq0$, $f(t)=1-\frac{1}{2}Ct^{2}%
-At\sqrt{1-B^{2}t^{2}}$ for $t\in I,$ where $I=[0,\frac{1}{B}]$ if $B>0,$ and
$I=[0,\infty)$ if $B=0.$

i. The equation (3.1) has no solution when $\frac{C}{2}+\frac{A^{2}}{4}%
-B^{2}<0$ or $A=C=0:$
\begin{equation}
1-\frac{1}{2}Ct^{2}-At\sqrt{1-B^{2}t^{2}}=0\text{ for }t\in I.
\end{equation}
Assume $A^{2}+C^{2}\neq0$ and $\frac{C}{2}+\frac{A^{2}}{4}-B^{2}\geq0$ for the
rest of the lemma.

ii. $\frac{C}{2}+\frac{A^{2}}{2}>0,$ and $\frac{C}{2}+\frac{A^{2}}{2}\geq
A\sqrt{\frac{C}{2}+\frac{A^{2}}{4}-B^{2}},$ where the equality occurs if and
only if $B=C=0<A.$

iii. The equation (3.1), $f(t)=0$ has at most 2 solutions on $I$, and they are
in the form $t_{0}^{+}$ or $t_{0}^{-}$ when they exist:
\[
t_{0}^{\pm}=\left(  \frac{C}{2}+\frac{A^{2}}{2}\pm A\sqrt{\frac{C}{2}%
+\frac{A^{2}}{4}-B^{2}}\right)  ^{-\frac{1}{2}}.
\]

Both $t_{0}^{+}$ and $t_{0}^{-}$ are the solutions of (3.1) unless $B=C=0$
($t_{0}^{-}=\infty\notin\mathbf{R}$). $t_{0}^{-}=\frac{1}{B}$ if and only if
$2B^{2}=C\neq0.$ Also, $t_{0}^{\pm}=\frac{1}{B}$ if and only if $2B^{2}%
=C\neq0=A.$

iv. $f^{\prime}(t)=0$ has at most one solution on $(0,\frac{1}{B}).$

v. If $B=C=0<A,$ then $t_{0}^{+}=\frac{1}{A}$ is the only solution of (3.1),

\qquad and $f(t)<0\Longleftrightarrow t_{0}^{+}<t.$

vi. If $\frac{C}{2}+\frac{A^{2}}{4}-B^{2}=0,$ then $t_{0}^{+}=t_{0}^{-}$ is
the only solution of (3.1),

\qquad and $f(t)>0$, for all $t\neq t_{0}^{+}.$

vii. If $\frac{C}{2}+\frac{A^{2}}{4}-B^{2}>0$ and $B^{2}+C^{2}\neq0$ then both
$t_{0}^{+}<t_{0}^{-}$ are the solutions of (3.1), and
$f(t)<0\Longleftrightarrow t_{0}^{+}<t<t_{0}^{-}$.
\end{lemma}

\begin{proof}
Squaring both sides of $1-\frac{1}{2}Ct^{2}=At\sqrt{1-B^{2}t^{2}}$ gives a
quadratic equation in $t^{2},$ and then solve for $u=1/t^{2}.$ For (iv),
substitute $t=\frac{1}{B}\sin\theta.$ The rest is elementary and long.
\end{proof}

\begin{figure}
[ptb]
\begin{center}
\includegraphics[
height=3.0191in,
width=4.3414in
]%
{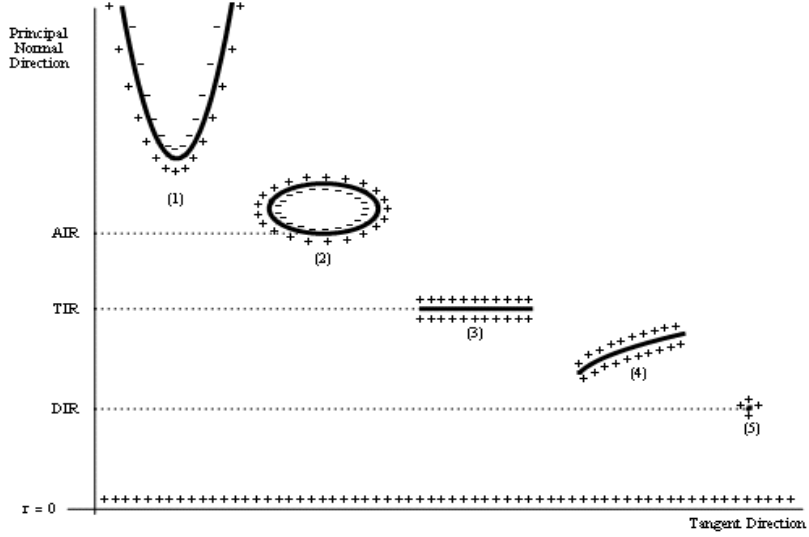}%
\caption{{\protect\small An example of the graph of the singular set in the
domain of }$\exp^{\mu}${\protect\small \ along the principal normal direction
}$N${\protect\small \ of a curve }$\gamma${\protect\small \ of positive
curvature is shown, as indicated in Proposition 2 and 5(ii). It is assumed
that }$DCSD${\protect\small \ is larger than }$2FocRad^{-}$%
{\protect\small \ in this example in order to indicate exact values of }%
$AIR${\protect\small , }$TIR${\protect\small , and }$DIR${\protect\small . The
second derivative of the squared weighted distance function }$\left\Vert
p-x\right\Vert ^{2}/\mu^{2}(x)${\protect\small \ is }$0$%
{\protect\small \ along the singular set, and its signs at nearby points are
indicated. Type (1) is the most common behavior, it is the only possibility
when }$\mu${\protect\small \ is sufficiently close to a constant, and it is
the graph of }$1/\kappa${\protect\small \ when }$\mu=1${\protect\small . The
"positive to negative and then to back to positive" behavior shown in (2)
occurs in Figure 8 (see Example 3), and Figure 11 (see Example 6). (3) depicts
the Horizontal Collapsing Property, as in Figure 7 (see Example 1A) and Figure
5 (Example 1B). (5) is a "Fake" focal point around which the }$\mu
${\protect\small -exponential map is a local homeomorphism but not a local
diffeomorphism, as in Figure 10, (see Example 4).}}%
\end{center}
\end{figure}

\begin{proposition}
Let a local parametrization $\gamma:I\rightarrow K$ with respect to arclength
$s$ be given, $\kappa(s)$ denote the curvature of $K$ at $\gamma(s),$
$\mu(s)=\mu(\gamma(s)):I\rightarrow\mathbf{R}^{+}\mathbf{,}$ and
$q=\gamma(s_{0}).$

i. If $p=\exp^{\mu}(q,Rv)$ for some $R\in(0,\left\Vert \mathit{grad}%
\mu(q)\right\Vert ^{-1})$ and $v\in UNK_{q},$ then
\[
F_{p}^{\prime\prime}(s_{0})=\frac{2}{\mu^{2}(s_{0})}\left(  1-\kappa
(s_{0})R\mu(s_{0})\sqrt{1-\left\Vert \mathit{grad}\mu(s_{0})\right\Vert
^{2}R^{2}}\cos\beta-\frac{R^{2}}{2}(\mu^{2})^{\prime\prime}(s_{0})\right)
\]
where $\beta=\measuredangle(\gamma^{\prime\prime}(s_{0}),u(q,p)^{N})$ when
both vectors are non-zero, and $\beta=0$ otherwise.

ii. Let $q$ and $v\in UNK_{q}$ be fixed, and $R$ vary. For $p(R)=\exp^{\mu
}(q,Rv),$ the sign of $\frac{d^{2}}{ds^{2}}\left.  F_{p(R)}(s)\right\vert
_{s=s_{0}}$ behaves in only one of the following manners, and in all cases
$q\in CP(q,+)$ at $R=0$:

\qquad a. $\forall R,$ $q\in CP(p(R),+)$

\qquad b. $\exists R_{1}>0,$ such that
\[
q\in\left\{
\begin{array}
[c]{cc}%
CP(p(R),+) & \text{if }R\in(0,R_{1})\text{ \ \ \ \ \ \ \ \ \ \ \ \ }\\
CP(p(R),0) & \text{if }R=R_{1}\text{ \ \ \ \ \ \ \ \ \ \ \ \ \ \ \ \ \ }\\
CP(p(R),-) & \text{if }R\in(R_{1},\left\Vert \mathit{grad}\mu(q)\right\Vert
^{-1})
\end{array}
\right.
\]

\qquad c. $\exists R_{2}>R_{1}>0$ such that
\[
q\in\left\{
\begin{array}
[c]{cc}%
CP(p(R),+) & \text{if }R\in(0,R_{1})\cup(R_{2},\left\Vert \mathit{grad}%
\mu(q)\right\Vert ^{-1})\\
CP(p(R),0) & \text{if }R=R_{1}\text{ or }R_{2}\text{
\ \ \ \ \ \ \ \ \ \ \ \ \ \ \ \ \ \ \ \ \ \ }\\
CP(p(R),-) & \text{if }R\in(R_{1},R_{2})\text{
\ \ \ \ \ \ \ \ \ \ \ \ \ \ \ \ \ \ \ \ \ \ \ }%
\end{array}
\right.
\]

\qquad d. $\exists R_{1}>0$ such that
\[
q\in\left\{
\begin{array}
[c]{cc}%
CP(p(R),+) & \text{if }R\neq R_{1}\\
CP(p(R),0) & \text{if }R=R_{1}%
\end{array}
\right.  .
\]

\end{proposition}

\begin{proof}
i. To simplify the calculations, set $E(s)=\left\Vert p-\gamma(s)\right\Vert
^{2}$ so that $F_{p}(s)=E(s)\mu(s)^{-2}.$ Since $p=\exp^{\mu}(q,Rv),$ we
already know that $F_{p}^{\prime}(s_{0})=0$ and $\left\Vert p-q\right\Vert
=R\mu(q)$ by Proposition 1(i). $\gamma^{\prime\prime}(s_{0})=\kappa
(s_{0})N_{\gamma}(s_{0})$ where $\kappa(s)$ is the curvature of $\gamma(s)$ in
the ambient space $\mathbf{R}^{n}$, and $N_{\gamma}(s)$ is the principal
normal of $\gamma(s)$ when $\kappa(s)>0.$ When $\kappa(s)=0,$ we will write
$\gamma^{\prime\prime}(s)=\kappa(s)N_{\gamma}(s)=0$ although $N_{\gamma}(s)$
is not defined. Since $s$ is the arclength$,$ $\gamma^{\prime\prime}(s_{0})\in
NK_{q}$. Let $\beta=\measuredangle(\gamma^{\prime\prime}(s_{0}),u(q,p)^{N})$
when both vectors are non-zero, otherwise take $\beta=0.$
\begin{align*}
\gamma^{\prime\prime}(s_{0})\cdot(p-q)  &  =\gamma^{\prime\prime}(s_{0})\cdot
u(q,p)\left\Vert p-q\right\Vert =\gamma^{\prime\prime}(s_{0})\cdot
u(q,p)^{N}\left\Vert p-q\right\Vert \\
&  =\kappa(s_{0})\cos\beta\left\Vert u(q,p)^{N}\right\Vert \left\Vert
p-q\right\Vert \\
&  =\kappa(s_{0})\cos\beta\sqrt{1-\left\Vert \mathit{grad}\mu(q)\right\Vert
^{2}R^{2}}R\mu(s_{0})
\end{align*}%
\begin{align*}
E^{\prime}(s)  &  =2\left(  p-\gamma(s)\right)  \cdot(-\gamma^{\prime}(s))\\
E^{\prime\prime}(s)  &  =2\gamma^{\prime}(s)\cdot\gamma^{\prime}(s)+2\left(
p-\gamma(s)\right)  \cdot(-\gamma^{\prime\prime}(s))\\
E^{\prime\prime}(s_{0})  &  =2\left[  1-\left(  p-q\right)  \cdot
\gamma^{\prime\prime}(s_{0})\right]
\end{align*}%
\begin{align*}
F_{p}^{\prime\prime}(s_{0})  &  =F_{p}(s_{0})\left(  \frac{E^{\prime\prime}%
}{E}-\frac{(\mu^{2})^{\prime\prime}}{\mu^{2}}\right)  (s_{0})\\
&  =\frac{\left\Vert p-q\right\Vert ^{2}}{\mu^{2}(s_{0})}\left(
\frac{2\left[  1-\left(  p-q\right)  \cdot\gamma^{\prime\prime}(s_{0})\right]
}{\left\Vert p-q\right\Vert ^{2}}-\frac{(\mu^{2})^{\prime\prime}}{\mu^{2}%
}(s_{0})\right) \\
&  =\frac{2}{\mu^{2}(s_{0})}\left(  1-\gamma^{\prime\prime}(s_{0}%
)\cdot(p-q)-\frac{\left\Vert p-q\right\Vert ^{2}}{2\mu^{2}(s_{0})}(\mu
^{2})^{\prime\prime}(s_{0})\right) \\
&  =\frac{2}{\mu^{2}(s_{0})}\left(  1-\gamma^{\prime\prime}(s_{0}%
)\cdot(p-q)-\frac{R^{2}}{2}(\mu^{2})^{\prime\prime}(s_{0})\right) \\
&  =\frac{2}{\mu^{2}(s_{0})}\left(  1-\kappa(s_{0})R\mu(s_{0})\sqrt
{1-\left\Vert \mathit{grad}\mu(s_{0})\right\Vert ^{2}R^{2}}\cos\beta
-\frac{R^{2}}{2}(\mu^{2})^{\prime\prime}(s_{0})\right)
\end{align*}

ii. Observe that $F_{p}^{\prime\prime}(s_{0})>0$ for small $R>0,$ and the
expression for $F_{p}^{\prime\prime}(s_{0})$ is continuous in $R,$ and it has
at most two roots by Lemma 3.
\end{proof}

\begin{definition}
For one variable functions $\mu\in C^{2},$ and $\kappa\in C^{0}\mathbf{,}$
define:%
\begin{align*}
\Delta(\kappa,\mu)  &  =\frac{1}{2}(\mu^{2})^{\prime\prime}+\frac{1}{4}%
\kappa^{2}\mu^{2}-(\mu^{\prime})^{2}=\mu\left(  \mu^{\prime\prime}%
+\frac{\kappa^{2}}{4}\mu\right) \\
\Lambda(\kappa,\mu)  &  =\frac{1}{2}(\mu^{2})^{\prime\prime}+\frac{1}{2}%
\kappa^{2}\mu^{2}+\kappa\mu\sqrt{\Delta(\kappa,\mu)}%
\end{align*}
\ \ Observe that $\Delta(\kappa,\mu)=\frac{C}{2}+\frac{A^{2}}{4}-B^{2}$ and
$\Lambda(\kappa,\mu)=\frac{C}{2}+\frac{A^{2}}{2}+A\sqrt{\frac{C}{2}%
+\frac{A^{2}}{4}-B^{2}}$, if $A=\kappa\mu,$ $B=\left\vert \mu^{\prime
}\right\vert $ and $C=(\mu^{2})^{\prime\prime},$ see Lemma 3.
\end{definition}

\begin{proposition}
i. Let $K$ be connected, with a given (onto) parametrization $\gamma
:Domain(\gamma)\rightarrow K,$ with respect to arclength $s$, $\kappa(s)$
denote the curvature of $K$ at $\gamma(s),$ $\mu(s)=\mu(\gamma
(s)):Domain(\gamma)\rightarrow\mathbf{R}^{+}\mathbf{,}$ and $q=\gamma(s_{0}).$
If the set%
\[
\left\{  R\in\left[  0,\left\Vert \mathit{grad}\mu(q)\right\Vert ^{-1}\right)
:\exists v\in UNK_{q}\text{, }p=\exp^{\mu}(q,Rv)\text{ and }F_{p}%
^{\prime\prime}(s_{0})=0\right\}
\]
is not empty, then its infimum is $\Lambda(\kappa,\mu)(s_{0})^{-\frac{1}{2}}.$

ii. $\left\{  s\in Domain(\gamma):\mu^{\prime\prime}+\frac{\kappa^{2}}{4}%
\mu>0\right\}  \neq\varnothing.$

iii. Both $FocRad^{0}(K,\mu)$ and $FocRad^{-}(K,\mu)\in\mathbf{R}^{+}$ are
positive (finite) real numbers.

iv. If $K$ has more than one component, then all of the above hold for each
component, and the zero-focal radius of the union is the minimum zero-focal
radii of all components.
\end{proposition}

\begin{proof}
i. For fixed $q\in K$ and $R,$ and varying $v\in UNK_{q}$, the expression for
$F_{p}^{\prime\prime}(s_{0})$ in Proposition 2 is minimal for $\beta=0.$ If
$\kappa(s_{0})>0$, then the minimum occurs when $v_{0}=N_{\gamma}(s_{0}),$ and
$p_{0}=\exp^{\mu}(q,Rv_{0})$. If $\kappa(s_{0})=0,$ then $F_{p}^{\prime\prime
}(s_{0})$ does not depend on $\cos\beta.$ Hence, for all $v\in UNK_{q}$, and
$p=\exp^{\mu}(q,Rv)$:
\[
F_{p}^{\prime\prime}(s_{0})\geq F_{p_{0}}^{\prime\prime}(s_{0})=\frac{2}%
{\mu^{2}(s_{0})}\left(  1-\kappa(s_{0})R\mu(s_{0})\sqrt{1-\left\Vert
\mathit{grad}\mu(s_{0})\right\Vert ^{2}R^{2}}-\frac{R^{2}}{2}(\mu^{2}%
)^{\prime\prime}(s_{0})\right)
\]
Assume that there is a solution of $F_{p}^{\prime\prime}(s_{0})=0$ with
$R\in\left[  0,\left\Vert \mathit{grad}\mu(q)\right\Vert ^{-1}\right)  .$ In
Lemma 3, if the smaller positive solution $t_{0}^{+}$ exists, then $t_{0}^{+}
$ decreases as $A=\kappa(s_{0})\mu(s_{0})\cos\beta$ increases to $\kappa
(s_{0})\mu(s_{0}).$ The smallest solution of $R$ for $F_{p_{0}}^{\prime\prime
}(s_{0})=0$ is $\Lambda(\kappa,\mu)(s_{0})^{-\frac{1}{2}},$ by Definition 9
and Lemma 3.

ii-iii. Since $K$ is compact, there exists $s_{1}\in Domain(\gamma)$ so that
$\mu^{\prime\prime}(s_{1})>0$ unless $\mu$ is constant$.$ Also, there exists
$s_{2}\in Domain(\gamma)$ so that $\kappa_{\gamma}(s_{2})>0$, in the case of
constant $\mu$. Hence, there exists $s_{i}$ (for either $i=1$ or $2)$ such
that $\Delta(\kappa,\mu)(s_{i})=\mu\left(  \mu^{\prime\prime}+\frac{\kappa
^{2}}{4}\mu\right)  (s_{i})>0$. Hence $\left\{  s\in Domain(\gamma
):\Delta(\kappa,\mu)(s)\geq0\right\}  $ is a non-empty compact subset of
$Domain(\gamma)$, and the maximum of $\Lambda(\kappa,\mu)$ is attained. This
maximum must be positive by Lemma 3(ii). Although $\left\vert \mu^{\prime
}(s)\right\vert ^{-1}\geq\Lambda(\kappa,\mu)(s)$ where $\Delta(s)\geq0$, it is
possible that maximum of $\left\vert \mu^{\prime}(s)\right\vert $ to occur
where $\Delta(s)<0.$ The proof for $FocRad^{-}(K,\mu)$ is similar, since
$\Lambda(\kappa,\mu)$ is bounded.

iv. This follows Definition 4.
\end{proof}

\section{$DIR$ and $TIR$}

Lemma 4.i is a well known result for $\mu=1,$ see [DC] or [CE] for example.

\begin{lemma}
(Recall that $F_{p}(x)=\left\Vert p-x\right\Vert ^{2}\mu(x)^{-2}$ and
$G(p)=\min_{x\in K}F_{p}(x).)$

i. Given $\,p\in\mathbf{R}^{n}$and $q\in K$ such that $G(p)=F_{p}(q)=R^{2}>0 $
so that $p=\exp^{\mu}(q,Rv)$ where $v\in UN_{q}$. $\forall w\in UT\mathbf{R}%
_{p}^{n}$ such that $u(p,q)\cdot w>0,$ there exists $\eta>0$ such that
$\forall t\in(0,\eta),$ $G(p+tw)<R^{2}.$

ii. If $G$ is differentiable at $p,$ then $\nabla G(p)=c_{1}u(q,p)$ for some
$c_{1}\geq\frac{2\left\Vert p-q\right\Vert }{\mu^{2}(q)}>0$

and $\nabla\sqrt{G}(p)=c_{2}u(q,p)$ for some $c_{2}\geq\frac{1}{\mu(q)}>0. $
\end{lemma}

\begin{proof}
Let $\measuredangle(u(p,q),w)=\theta<\frac{\pi}{2}.$

i. By a simple acute triangle argument in $\mathbf{R}^{n},$ for all small
$t>0:$
\[
R^{2}=G(p)=\frac{\left\Vert p-q\right\Vert ^{2}}{\mu^{2}(q)}>\frac{\left\Vert
p+tw-q\right\Vert ^{2}}{\mu^{2}(q)}\geq\min_{x\in K}F_{p+tw}(x)=G(p+tw)
\]

ii. $\forall w\in UT\mathbf{R}_{p}^{n}$ such that $u(p,q)\cdot w=\cos
\theta>0,$ and for all small $t>0,$ (by the Law of Cosines)
\begin{align*}
G(p)-G(p+tw)  &  \geq\frac{\left\Vert p-q\right\Vert ^{2}}{\mu^{2}(q)}%
-\frac{\left\Vert p+tw-q\right\Vert ^{2}}{\mu^{2}(q)}=\frac{2t\left\Vert
p-q\right\Vert \cos\theta-t^{2}}{\mu^{2}(q)}\\
\mu^{2}(q)\left(  -\nabla G(p)\right)  \cdot w  &  \geq2\left\Vert
p-q\right\Vert \cos\theta>0
\end{align*}
Therefore, $\nabla G(p)$ points in the direction of $u(q,p)=-u(p,q).$%
\begin{align*}
\left\Vert \nabla G(p)\right\Vert  &  \geq\frac{2\left\Vert p-q\right\Vert
}{\mu^{2}(q)}\\
\nabla\sqrt{G}  &  =\frac{1}{2\sqrt{G}}\nabla G\\
\left\Vert \nabla\sqrt{G}\right\Vert  &  \geq\frac{1}{\mu(q)}%
\end{align*}

\end{proof}

$DIR(K,\mu)=\min\left(  \frac{1}{2}DCSD(K,\mu),\mathit{RegRad}(K,\mu)\right)
$ in Proposition 5, generalizes a proposition in [CE, p. 95] or [DC, p. 274],
about the injectivity radius of the ($\mu=1$) exponential map from a point
which use the local invertibility of $\exp_{p}$ where it is non-singular$.$
However, our proofs must follow an altered course. Geodesics are not
minimizing past focal points in the $\mu=1$ case where $DIR(K,1)=TIR(K,1)$.
Hence, $\exp^{1}$ fails to be injective past first focal point(s). For general
$\mu,$ we have examples with $\mathit{RegRad}(K,\mu)<TIR(K,\mu),$ that is
$\exp^{\mu}$ is injective past some focal points, (Example 4) and it is
possible to have $DIR(K,\mu)=LR(K,\mu)<TIR(K,\mu)<UR(K,\mu)$, (Examples 2, 4
and 5). The approach of the proof of Proposition 4 about $TIR$ is in essence
similar to the proofs in [CE, p. 95], or [DC, p. 274]. However, we will use
the positivity of the second derivatives instead of regularity of the
exponential map. We will discuss the relation of singular points and zeroes of
the second derivatives to understand the relation of $DIR$ with $TIR$.

\begin{proposition}
i. If $R=TIR(K,\mu),$ then either $R=\frac{1}{2}DCSD(M,\mu)$ or there exists
$q\in K$ and $p\in\mathbf{R}^{n}$ such that $\left\Vert p-q\right\Vert
=R\mu(q)$ and $q\in CP(p,0).$

ii. $LR(K,\mu)\leq TIR(K,\mu)\leq UR(K,\mu).$
\end{proposition}

\begin{proof}
First, we will prove the second inequality of (ii):

\textbf{Claim 1. }$TIR(K,\mu)\leq FocRad^{-}(K,\mu).$

Suppose that $FocRad^{-}(K,\mu)<TIR(K,\mu)$. Then, there exists $p=\exp^{\mu
}(q_{1},v_{1})$ such that $FocRad^{-}(K,\mu)<\left\Vert v_{1}\right\Vert
<TIR(K,\mu)$ and $q_{1}\in CP(p,-).$ $F_{p}^{\prime\prime}(s_{1})<0$ for
$\gamma:I\rightarrow K\subset\mathbf{R}^{n}$ with $q_{1}=\gamma(s_{1})\in K.$
$F_{p}$ can not attain its minimum at $q_{1}.$ Consequently, $\exists q_{2}\in
K-\{q_{1}\}$ such that $F_{p}(q_{2})=G(p)=\min_{x\in K}F_{p}(x)<F_{p}%
(q_{1})=\left\Vert v_{1}\right\Vert ^{2}$ and $q_{2}\in CP(p).$ By Proposition
1, $p=\exp^{\mu}(q_{2},v_{2})$ for some $v_{2}\in NK_{q_{2}}$ such that
$\left\Vert v_{2}\right\Vert ^{2}=F_{p}(q_{2})<\left\Vert v_{1}\right\Vert
^{2}<TIR(K,\mu)^{2}.$ This implies that $\exp^{\mu}$ restricted to $D(r)$ is
not injective for all $r$ with $\left\Vert v_{1}\right\Vert <r<TIR(K,\mu)$
which contradicts with the definition of $TIR.$ This proves Claim 1.

By Lemma 2, if $\{q_{1},q_{2}\}$ is a critical pair, then there exists $p$ on
the line segment joining $q_{1}$ and $q_{2}$ such that $\left\Vert
p-q_{i}\right\Vert =R\mu(q_{i})$ and $p=\exp^{\mu}(q_{i},Rv_{i})$ for and
$v_{i}\in UNK_{q_{i}}$ for $i=1$ and $2,$ and injectivity of $\exp^{\mu}$
fails on $D(R+\varepsilon),\forall\varepsilon>0.$ Hence,
\begin{equation}
TIR(K,\mu)\leq\min\left(  \frac{1}{2}DCSD(K,\mu),FocRad^{-}(K,\mu)\right)
=UR(K,\mu).
\end{equation}
The rest of (ii) will be proved after (i).

(i) Since, $d(\exp^{\mu}(q,v))_{v=0}=\mu(q)Id,$ and $K$ is compact, there
exists $r_{0}>0,$ such that $exp^{\mu}$ restricted to $D(r_{0})$ is a
diffeomorphism. Let $R=\sup\{r:exp^{\mu}$ restricted to $D(r)$ is
injective$\}.$ $\exp^{\mu}:D(R)\rightarrow O(K,\mu R)$ is injective, since
$\exp^{\mu}(q_{1},w_{1})=\exp^{\mu}(q_{2},w_{2})$ with $\max(\left\Vert
w_{1}\right\Vert ,\left\Vert w_{2}\right\Vert )<R$ would imply that
$\max(\left\Vert w_{1}\right\Vert ,\left\Vert w_{2}\right\Vert )<r$ for some
$r<R. $ $\exp^{\mu}:\overline{D(r)}\rightarrow\overline{O(K,\mu r)}$ is a
homeomorphism onto its image $\forall r<R$, since it is continuous and
injective on a compact domain. The map $\exp^{\mu}:D(r)\rightarrow O(K,\mu r)$
is onto by Corollary 1, and an open map into $\mathbf{R}^{n},$ since $O(K,\mu
r)$ is an open subset of $\mathbf{R}^{n},$ $\forall r<R.$ Hence, $\exp^{\mu
}:D(R)\rightarrow O(K,\mu R)$ is continuous, open and injective, and therefore
a homeomorphism. It follows that $R=TIR(K,\mu).$ $\forall m\in\mathbf{N}^{+},$
injectivity of $\exp^{\mu}$ fails on $D(R+\frac{1}{m}),~$and there exist
distinct $(y_{m},v_{m}),(z_{m},w_{m})\in D(R+\frac{1}{m}) $ such that
$\exp^{\mu}(y_{m},v_{m})=\exp^{\mu}(z_{m},w_{m})=x_{m}\in\mathbf{R}^{n}$,
$\left\Vert v_{m}\right\Vert <R+\frac{1}{m}$ and $\left\Vert w_{m}\right\Vert
<R+\frac{1}{m}$. If both $\left\Vert v_{m}\right\Vert <R$ and $\left\Vert
w_{m}\right\Vert <R$ were true simultaneously, $exp^{\mu}$ restricted to
$D(r)$ would not be injective for some $r<R.$ So, we can assume that
$\left\Vert v_{m}\right\Vert \geq R,\forall m.$ By compactness, there exist
convergent subsequences (use index $j$ instead of $m_{j}$) $y_{j}\rightarrow
y_{0},$ $v_{j}\rightarrow v_{0}\in NK_{y_{0}}\cap W,$ $z_{j}\rightarrow z_{0}$
and $w_{j}\rightarrow w_{0}\in NK_{z_{0}}\cap W$ as $j\rightarrow\infty,$ such
that $\exp^{\mu}(y_{0},v_{0})=\exp^{\mu}(z_{0},w_{0})=p.$
\[
\left\Vert v_{0}\right\Vert =\lim\left\Vert v_{j}\right\Vert =R\text{ and
}\left\Vert w_{0}\right\Vert =\lim\left\Vert w_{j}\right\Vert \leq R
\]
Suppose that $\left\Vert w_{0}\right\Vert <R$. We showed that $exp^{\mu
}:D(R)\rightarrow O(K,\mu R)$ is a homeomorphism onto an open subset
of\textbf{\ }$\mathbf{R}^{n}$. Observe that $\exp^{\mu}(y_{0},tv_{0})$ is a
curve starting at $y_{0}$, going to $p$ at the boundary of $\exp^{\mu}(D(R))$,
and $p=\exp^{\mu}(z_{0},w_{0})$ which is an interior point of $\exp^{\mu
}(D(R)).$ This leads to a contradiction. Hence,
\[
\left\Vert w_{0}\right\Vert =\left\Vert v_{0}\right\Vert =R.
\]

Let $\gamma:Domain(\gamma)\rightarrow K$ be a parametrization with respect to
arclength such that $y_{0}=\gamma(s_{0})$ and $z_{0}=\gamma(t_{0}).$

\textbf{Case 1. }If $y_{0}\in CP(p,0)$ or $z_{0}\in CP(p,0),$ then the proof
of (i) is finished. We also have $FocRad^{0}(K,\mu)\leq TIR(K,\mu)$ in this
case$.$

\textbf{Case 2. }If $y_{0}\in CP(p,-),$ that is $F_{p}^{\prime\prime}%
(s_{0})<0,$ then it would still be true that $F_{p^{\prime}}^{\prime\prime
}(s_{0})<0$ for $p^{\prime}=\exp^{\mu}(y_{0},(1-\varepsilon)v_{0}) $ for some
$\varepsilon>0.$ This would imply that $FocRad^{-}(K,\mu)\leq(1-\varepsilon
)R<R$ which contradicts Claim 1. Hence, $y_{0}\notin CP(p,-)$ and $z_{0}\notin
CP(p,-).$

\textbf{Case 3. }$y_{0}=z_{0}\in CP(p,+)$ and $v_{0}=w_{0}.$
\begin{align*}
\exists\varepsilon_{1}  &  >0\text{ with }I_{1}=[s_{0}-\varepsilon_{1}%
,s_{0}+\varepsilon_{1}]\text{ such that}\\
\forall x  &  \in B(p,\varepsilon_{1})\text{, }\forall s\in I_{1},\text{
}F_{x}^{\prime\prime}(s)>0.
\end{align*}%
\begin{align*}
\exists\varepsilon_{2}  &  \in(0,\varepsilon_{1})\text{ with }I_{2}%
=[s_{0}-\varepsilon_{2},s_{0}+\varepsilon_{2}]\subset I_{1}\text{ and }%
\exists\delta>0\text{ such that}\\
\text{i. }\forall s  &  \in I_{2}-\{s_{0}\},\text{ }F_{p}(s)>F_{p}%
(s_{0})=R^{2}\text{ and }\\
\text{ii. }\forall s  &  \in\partial I_{2},\text{ }F_{p}(s)\geq\left(
R+\delta\right)  ^{2}.
\end{align*}%
\[
\exists j_{0},\forall j\geq j_{0},\text{ }\left\Vert x_{j}-p\right\Vert
<\min\left(  \frac{\delta\min\mu}{3},\varepsilon_{1}\right)  \text{, }y_{j}%
\in\gamma(I_{2})\text{ and }z_{j}\in\gamma(I_{2})
\]%
\begin{align*}
\forall s  &  \in\partial I_{2}\text{ and }\forall j\geq j_{0}:\\
\left\Vert \gamma(s)-x_{j}\right\Vert  &  \geq\left\Vert \gamma
(s)-p\right\Vert -\left\Vert p-x_{j}\right\Vert \geq\mu(s)(R+\delta
)-\frac{\delta\min\mu}{3}\geq\mu(s)(R+\frac{2\delta}{3})\\
\text{ hence, }F_{x_{j}}(s)  &  \geq\left(  R+\frac{2\delta}{3}\right)  ^{2}%
\end{align*}%
\begin{align*}
\forall j  &  \geq j_{0},\\
\left\Vert y_{0}-x_{j}\right\Vert  &  \leq\left\Vert y_{0}-p\right\Vert
+\left\Vert p-x_{j}\right\Vert \leq\mu(s_{0})R+\frac{\delta\min\mu}{3}\leq
\mu(s_{0})\left(  R+\frac{\delta}{3}\right) \\
F_{x_{j}}(s_{0})  &  \leq\left(  R+\frac{\delta}{3}\right)  ^{2}%
\end{align*}
The minima of $F_{x_{j}}$ restricted to $I_{2}$ are attained in the interior
of $I_{2},\forall j\geq j_{0}.$ The function $F_{x_{j}}(s)$ has interior
strict local minima at both $y_{j}$ and $z_{j}$ by the choice of
$\varepsilon_{2}.$ We chose $(y_{j},v_{j})\neq(z_{j},w_{j})$ initially$.$ The
case of $y_{j}=z_{j}$ with $v_{j}\neq w_{j}$ and $\exp^{\mu}(y_{j},v_{j}%
)=\exp^{\mu}(z_{j},w_{j})$ implies that $\left\Vert v_{j}\right\Vert
=\left\Vert w_{j}\right\Vert =\left\Vert \operatorname{grad}\mu(y_{j}%
)\right\Vert ^{-1}>TIR(K,\mu)$ by Proposition 1(ii, vi). There exist
$j_{1}\geq j_{0}$ such that $\forall j\geq j_{1},$ $y_{j}\neq z_{j}.$ For
otherwise, one would obtain $R=\left\Vert v_{0}\right\Vert =\left\Vert
w_{0}\right\Vert =\left\Vert \operatorname{grad}\mu(y_{0})\right\Vert
^{-1}>TIR(K,\mu)$ which is not the case. There must be a local maximum of
$F_{x_{j}}(s)$ between $y_{j}$ and $z_{j}$ at an interior point of
$\gamma(I_{2}),$ which contradicts with the choice of $\varepsilon_{1}.$ Case
3 can not occur.

\textbf{Case 4. }$y_{0}=z_{0}$ and $v_{0}\neq w_{0}.$ The injectivity of
$exp^{\mu}\mid(NK_{y_{0}}\cap W)$ can only fail at $\left\Vert v_{0}%
\right\Vert =\left\Vert w_{0}\right\Vert =\left\Vert \mathit{grad}\mu
(y_{0})\right\Vert ^{-1}$, Proposition 1(ii). However, $\left\Vert
\mathit{grad}\mu(y_{0})\right\Vert ^{-1}>R=TIR(K,\mu)$ by Proposition 1(vi).
Case 4 can not occur.

\textbf{Case 5. }$y_{0}\neq z_{0}$ with $y_{0}\in CP(p,+)$ and $z_{0}\in
CP(p,+).$ Recall $y_{0}=\gamma(s_{0})$ and $z_{0}=\gamma(t_{0}).$

\textbf{Claim 2.} $u(p,y_{0})=-u(p,z_{0}).$

There exists $\varepsilon_{1}>\varepsilon_{2}>0$ and $\delta>0$ (as in Case 3)
with $I_{i}=[s_{0}-\varepsilon_{i},s_{0}+\varepsilon_{i}]$ and $J_{i}%
=[t_{0}-\varepsilon_{i},t_{0}+\varepsilon_{i}]$ for $i=1,$ $2$ such that

i. $\gamma(I_{1})\cap\gamma(J_{1})=\varnothing,$

ii. $\forall x\in B(p,\varepsilon_{1})$ and $\forall s\in I_{1}\cup J_{1},$
$F_{x}^{\prime\prime}(s)>0,$

iii. $\forall s\in I_{2}-\{s_{0}\},$ $F_{p}(s)>F_{p}(s_{0})=R^{2}$ and
$\forall s\in J_{2}-\{t_{0}\},$ $F_{p}(s)>F_{p}(t_{0})=R^{2}$, and

iv. $\forall s\in\partial I_{2},$ $F_{p}(s)\geq\left(  R+\delta\right)  ^{2} $
and $\forall s\in\partial J_{2},$ $F_{p}(s)\geq\left(  R+\delta\right)  ^{2}.$

Suppose that $u(p,y_{0})\neq-u(p,z_{0}).$ There exists $w\in UT\mathbf{R}%
_{p}^{n}$ with $u(p,y_{0})\cdot w>0$ and $u(p,z_{0})\cdot w>0.$ As in the
proof of Lemma 4, there exists $\eta\in(0,\delta\min\mu)$ such that the point
$p_{1}=p+\eta w$ satisfies that
\begin{align*}
0  &  <\left\Vert y_{0}-p_{1}\right\Vert <\left\Vert y_{0}-p\right\Vert
=R\mu(y_{0})\\
0  &  <\left\Vert z_{0}-p_{1}\right\Vert <\left\Vert z_{0}-p\right\Vert
=R\mu(z_{0})
\end{align*}%
\begin{align*}
\forall s  &  \in\partial I_{2},\\
\left\Vert \gamma(s)-p\right\Vert  &  \geq(R+\delta)\mu(s)\\
\left\Vert \gamma(s)-p_{1}\right\Vert  &  \geq\left\Vert \gamma
(s)-p\right\Vert -\left\Vert p-p_{1}\right\Vert \\
&  \geq(R+\delta)\mu(s)-\delta\min\mu\\
&  \geq R\mu(s)\\
F_{p_{1}}(s)  &  \geq R^{2}\\
F_{p_{1}}(s_{0})  &  =\left\Vert y_{0}-p_{1}\right\Vert ^{2}\mu(y_{0}%
)^{-2}<R^{2}%
\end{align*}
The minimum of $F_{p_{1}}$ restricted to $I_{2}$ is attained at $q_{1}%
=\gamma(s_{0}^{\prime})$ with $s_{0}^{\prime}\in interior(I_{2})$ and
$F_{p_{1}}(q_{1})<R^{2}.$ In fact, $q_{1}$ is unique (see the very end of Case
3). Similarly, there exists $q_{2}=\gamma(t_{0}^{\prime})$ with $t_{0}%
^{\prime}\in interior(J_{2})$ such that $F_{p_{1}}(q_{2})=\min\left(
F_{p_{1}}\mid J_{2}\right)  <R^{2}.$ Clearly, $q_{1}\neq q_{2}.$ $p_{1}%
=\exp^{\mu}(q_{1},R_{1}u_{1})=\exp^{\mu}(q_{2},R_{2}u_{2}),$ for some
$u_{i}\in UNK_{q_{i}}$ and $R_{i}<R,$ for $i=1,2$. This would imply that
$exp^{\mu}$ is not injective on $D(r)$ for some $r<R=TIR(K,\mu),$ which
contradicts the definition of $TIR.$ This concludes the proof of Claim 2,
$u(p,y_{0})=-u(p,z_{0}).$

We have three colinear points $y_{0},p,z_{0},$ where $y_{0}$ and $z_{0}$ are
both in $CP(p)$ and $R=\frac{\left\Vert p-y_{0}\right\Vert }{\mu(y_{0})}%
=\frac{\left\Vert p-z_{0}\right\Vert }{\mu(z_{0})}.$ By Lemma 2,
$\{y_{0},z_{0}\}$~is a critical pair for $(K,\mu)$ and $R\geq\frac{1}%
{2}DCSD(K,\mu).$ By (4.1), $R=TIR(K,\mu)=\frac{1}{2}DCSD(K,\mu).$ This
finishes all cases for (i).

ii. Summarizing all the cases, we have either $FocRad^{0}(K,\mu)\leq
TIR(K,\mu)$ in Case 1, or $TIR(K,\mu)=\frac{1}{2}DCSD(K,\mu)$ in Case 5.
\[
LR(K,\mu)=\min\left(  \frac{1}{2}DCSD(K,\mu),FocRad^{0}(K,\mu)\right)  \leq
TIR(K,\mu).
\]
\pagebreak
\end{proof}

\begin{lemma}
Let $\gamma(s):I\rightarrow K$ be a parametrization of $K$ with respect to
arclength, $v(s):I\rightarrow UNK$ be $C^{1}$ with $v(s)\in UNK_{\gamma(s)}$
and $R\in\mathbf{R}^{+}$ be such that $(\gamma(s),Rv(s))\in interior(W)$ for
$\left|  s-s_{0}\right|  <\varepsilon$, $\eta(s)=\exp^{\mu}(\gamma(s),Rv(s)),$
$q=\gamma(s_{0})$ and $p=\eta(s_{0})$. Then,
\[
\eta^{\prime}(s_{0})\cdot\gamma^{\prime}(s_{0})=\frac{\mu^{2}(s_{0})}{2}%
\frac{d^{2}}{ds^{2}}\left.  F_{p}(\gamma(s))\right|  _{s=s_{0}}=\frac{\mu
^{2}(s_{0})}{2}F_{p}^{\prime\prime}(s_{0})
\]
\[
\eta^{\prime}(s_{0})\cdot\left(  \eta(s_{0})-c(s_{0})\right)  =\frac{\mu
^{3}(s_{0})}{4\mu^{\prime}(s_{0})}\frac{d^{2}}{ds^{2}}\left.  F_{p}%
(\gamma(s))\right|  _{s=s_{0}}=\frac{\mu^{3}(s_{0})}{4\mu^{\prime}(s_{0}%
)}F_{p}^{\prime\prime}(s_{0})
\]
provided that in the second equality one has $\mu^{\prime}(s)\neq0$ and
$c(s)=\gamma(s)-\frac{\mu(s)}{2\mu^{\prime}(s)}\gamma^{\prime}(s)$ to be the
center of the $n-1$ dimensional sphere containing $\exp^{\mu}(NK_{\gamma
(s)}\cap W).$
\end{lemma}

\begin{proof}
By the definition of $\exp^{\mu}$ and $grad$ $\mu$, and proof of Proposition
2(i)$:$%
\[
\eta=\gamma-\mu\mu^{\prime}R^{2}\gamma^{\prime}+\mu R\sqrt{1-\left(
\mu^{\prime}R\right)  ^{2}}v
\]%
\begin{align}
\eta\cdot\gamma^{\prime}  &  =\gamma\cdot\gamma^{\prime}-\mu\mu^{\prime}%
R^{2}=\gamma\cdot\gamma^{\prime}-\frac{1}{2}R^{2}\left(  \mu^{2}\right)
^{\prime}\nonumber\\
\eta^{\prime}\cdot\gamma^{\prime}  &  =\left(  \eta\cdot\gamma^{\prime
}\right)  ^{\prime}-\eta\cdot\gamma^{\prime\prime}\nonumber\\
\eta^{\prime}\cdot\gamma^{\prime}  &  =1+\left(  \gamma-\eta\right)
\cdot\gamma^{\prime\prime}-\frac{1}{2}R^{2}\left(  \mu^{2}\right)
^{\prime\prime}\\
\eta^{\prime}(s_{0})\cdot\gamma^{\prime}(s_{0})  &  =1-(p-q)\cdot
\gamma^{\prime\prime}(s_{0})-\frac{1}{2}R^{2}\left(  \mu^{2}\right)
^{\prime\prime}(s_{0})\nonumber\\
&  =\frac{\mu^{2}(s_{0})}{2}F_{p}^{\prime\prime}(s_{0})=\frac{\mu^{2}(s_{0}%
)}{2}\frac{d^{2}}{ds^{2}}\left.  F_{p}(\gamma(s))\right\vert _{s=s_{0}}%
\end{align}

For the second part, assume that $\mu^{\prime}(s)\neq0$ locally.
\begin{align}
\eta &  =\gamma-\mu\mu^{\prime}R^{2}\gamma^{\prime}+\mu R\sqrt{1-\left(
\mu^{\prime}R\right)  ^{2}}v\nonumber\\
c  &  =\gamma-\frac{\mu}{2\mu^{\prime}}\gamma^{\prime}\nonumber\\
\eta^{\prime}\cdot\left(  \eta-c\right)   &  =\eta^{\prime}\cdot\gamma
^{\prime}\left(  -\mu\mu^{\prime}R^{2}+\frac{\mu}{2\mu^{\prime}}\right)
+\eta^{\prime}\cdot v\left(  \mu R\sqrt{1-\left(  \mu^{\prime}R\right)  ^{2}%
}\right)
\end{align}
By $v\cdot\gamma^{\prime}=$ $v\cdot v^{\prime}=0$, $\gamma^{\prime}\cdot
\gamma^{\prime}=v\cdot v=1,$ and the proof of Proposition 2(i):
\begin{align}
\eta^{\prime}\cdot v  &  =\left(  \gamma-\mu\mu^{\prime}R^{2}\gamma^{\prime
}+\mu R\sqrt{1-\left(  \mu^{\prime}R\right)  ^{2}}v\right)  ^{\prime}\cdot
v\nonumber\\
\eta^{\prime}\cdot v  &  =-\mu\mu^{\prime}R^{2}\gamma^{\prime\prime}\cdot
v+\left(  \mu R\sqrt{1-\left(  \mu^{\prime}R\right)  ^{2}}\right)  ^{\prime}%
\end{align}%
\[
\left(  \mu R\sqrt{1-\left(  \mu^{\prime}R\right)  ^{2}}\right)  \left(  \mu
R\sqrt{1-\left(  \mu^{\prime}R\right)  ^{2}}\right)  ^{\prime}=\frac{1}%
{2}\left(  \mu^{2}R^{2}\left(  1-\left(  \mu^{\prime}R\right)  ^{2}\right)
\right)  ^{\prime}%
\]%
\begin{equation}
=\mu\mu^{\prime}R^{2}-\left(  \mu\left(  \mu^{\prime}\right)  ^{3}+\mu^{2}%
\mu^{\prime}\mu^{\prime\prime}\right)  R^{4}%
\end{equation}
By the proof of Proposition 1(i) and $\gamma^{\prime\prime}(s)\in
NK_{\gamma(s)}$:
\begin{align}
\gamma^{\prime\prime}\cdot(\eta-\gamma)  &  =\gamma^{\prime\prime}\cdot
u(\gamma,\eta)R\mu=\gamma^{\prime\prime}\cdot u(\gamma,\eta)^{N}%
R\mu\nonumber\\
\gamma^{\prime\prime}\cdot(\eta-\gamma)  &  =\gamma^{\prime\prime}\cdot
v\left\Vert u(\gamma,\eta)^{N}\right\Vert R\mu=\gamma^{\prime\prime}\cdot
vR\mu\sqrt{1-\left(  \mu^{\prime}R\right)  ^{2}}%
\end{align}
By combining (4.5), (4.6), (4.7) and using (4.2) in the last step:%
\[
\eta^{\prime}\cdot v\left(  \mu R\sqrt{1-\left(  \mu^{\prime}R\right)  ^{2}%
}\right)  =\text{
\ \ \ \ \ \ \ \ \ \ \ \ \ \ \ \ \ \ \ \ \ \ \ \ \ \ \ \ \ \ \ \ \ \ \ \ \ \ \ \ \ \ \ \ \ \ \ \ \ \ \ \ \ \ \ \ \ \ \ \ \ \ \ \ \ \ \ \ \ \ \ \ \ \ \ \ \ \ \ \ \ }%
\]%
\begin{align}
&  =\mu R\sqrt{1-\left(  \mu^{\prime}R\right)  ^{2}}\left(  -\mu\mu^{\prime
}R^{2}\gamma^{\prime\prime}\cdot v+\left(  \mu R\sqrt{1-\left(  \mu^{\prime
}R\right)  ^{2}}\right)  ^{\prime}\right) \nonumber\\
&  =-\mu\mu^{\prime}R^{2}\left(  \mu R\sqrt{1-\left(  \mu^{\prime}R\right)
^{2}}\right)  \gamma^{\prime\prime}\cdot v+\mu\mu^{\prime}R^{2}-\left(
\mu\left(  \mu^{\prime}\right)  ^{3}+\mu^{2}\mu^{\prime}\mu^{\prime\prime
}\right)  R^{4}\nonumber\\
&  =-\mu\mu^{\prime}R^{2}\gamma^{\prime\prime}\cdot(\eta-\gamma)+\mu
\mu^{\prime}R^{2}-\mu\mu^{\prime}\left(  \left(  \mu^{\prime}\right)  ^{2}%
+\mu\mu^{\prime\prime}\right)  R^{4}\nonumber\\
&  =\mu\mu^{\prime}R^{2}\left(  1-\gamma^{\prime\prime}\cdot(\eta
-\gamma)-\frac{1}{2}R^{2}\left(  \mu^{2}\right)  ^{\prime\prime}\right)
\nonumber\\
&  =\mu\mu^{\prime}R^{2}\left(  \eta^{\prime}\cdot\gamma^{\prime}\right)
\end{align}
By combining (4.4), (4.8) and using (4.3) in the last step:%
\begin{align*}
\eta^{\prime}\cdot\left(  \eta-c\right)   &  =\left(  -\mu\mu^{\prime}%
R^{2}+\frac{\mu}{2\mu^{\prime}}\right)  \left(  \eta^{\prime}\cdot
\gamma^{\prime}\right)  +\mu\mu^{\prime}R^{2}\left(  \eta^{\prime}\cdot
\gamma^{\prime}\right) \\
&  =\frac{\mu}{2\mu^{\prime}}\left(  \eta^{\prime}\cdot\gamma^{\prime}\right)
\\
\eta^{\prime}(s_{0})\cdot\left(  \eta(s_{0})-c(s_{0})\right)   &  =\frac
{\mu(s_{0})}{2\mu^{\prime}(s_{0})}\eta^{\prime}(s_{0})\cdot\gamma^{\prime
}(s_{0})=\frac{\mu(s_{0})}{2\mu^{\prime}(s_{0})}\frac{\mu^{2}(s_{0})}{2}%
F_{p}^{\prime\prime}(s_{0})\\
&  =\frac{\mu^{3}(s_{0})}{4\mu^{\prime}(s_{0})}F_{p}^{\prime\prime}(s_{0})
\end{align*}

\end{proof}

\begin{proposition}
Let $K$ be a union of finitely many disjoint simple smoothly closed possibly
linked or knotted curves in $\mathbf{R}^{n}$ and $\mu:K\rightarrow(0,\infty)$
be given.

i. $\exp^{\mu}$ restricted to the normal plane $NK_{q}\cap int(W)$ is
non-singular, for each $q\in K$. $\exp^{\mu}$ is singular at the boundary of
$W$ where the spheres $\exp^{\mu}(NK_{q}\cap W)$ close up at the antipodal of
$q.$

ii. Let $(q,w)$ be an interior point of $W$, $\exp^{\mu}(q,w)=p$,
$\gamma:I\rightarrow K$ be a parametrization of $K$ with respect to arclength
and $q=\gamma(s_{0})$.%
\[
\exp^{\mu}\text{ is singular at }(q,w)\text{ if and only if }\frac{d^{2}%
}{ds^{2}}\left.  F_{p}(\gamma(s))\right\vert _{s=s_{0}}=0.
\]

iii$.$%
\[
\mathit{RegRad}(K,\mu)=FocRad^{0}(K,\mu)
\]
\[
DIR(K,\mu)=LR(K,\mu)=\min\left(  \frac{1}{2}DCSD(K,\mu),\mathit{RegRad}%
(K,\mu)\right)
\]

\end{proposition}

\begin{proof}
i. For a fixed $q,$ by Proposition 1(ii):
\[
\exp^{\mu}(q,Rv)=\left\{
\begin{array}
[c]{cc}%
q+\mu(q)R\left(  \cos\alpha(R)\frac{\mathit{grad}\mu(q)}{\left\Vert
\mathit{grad}\mu(q)\right\Vert }+\sin\alpha(R)v\right)  & \text{if
}\mathit{grad}\mu(q)\neq0\\
q+\mu(q)Rv & \text{if }\mathit{grad}\mu(q)=0
\end{array}
\right.
\]%
\[
\text{where }\cos\alpha(R)=-R\left\Vert \mathit{grad}\mu(q)\right\Vert \text{
and }\sin\alpha(R)=\sqrt{1-\left(  R\left\Vert \mathit{grad}\mu(q)\right\Vert
\right)  ^{2}}.
\]

If $\mathit{grad}\mu(q)=0,$ $\exp^{\mu}$ restricted to $NK_{q}$ is a dilation
and translation, and it is non-singular along \thinspace$NK_{q}$. If
$\mathit{grad}\mu(q)\neq0$, for each fixed $v\in UNK_{q},$ $\exp^{\mu}(q,Rv)$
follows the great circles of the sphere $\exp^{\mu}(NK_{q}\cap W)$ starting at
$q$ with non-zero speed until $q^{\prime}=\exp^{\mu}(q,v\left\Vert
\mathit{grad}\mu(q)\right\Vert ^{-1})$ and $\exp^{\mu}$ is non-singular along
\thinspace$NK_{q}\cap int(W)$. However, $q^{\prime}=\exp^{\mu}(q,v\left\Vert
\mathit{grad}\mu(q)\right\Vert ^{-1})$ for all $v\in UNK_{q},$ the sphere
$\exp^{\mu}(NK_{q}\cap W)$ closes up at $q^{\prime}$, the antipodal of $q$.
Hence, $\exp^{\mu}$ is singular along \thinspace$NK_{q}\cap\partial W.$

ii. \textbf{Case 1}. $\mu^{\prime}(s_{0})\neq0.$

Assume that $\exp^{\mu}$ is singular at $(q,w)$ where $\exp^{\mu}(q,w)=p$,
$(q,w)\in int(W).$ There exists a regular curve $\overline{\beta}(t)$ in $NK$,
such that $\overline{\beta}(t_{0})=(q,w)$ and $\exp^{\mu}(\overline{\beta
}(t))$ is singular at $t=t_{0}.$ $\overline{\beta}(t)=(\overline{\gamma
}(t),\overline{R}(t)\overline{v}(t))$ for $\overline{v}(t)\in UNK_{\overline
{\gamma}(t)}.$ By (i), the singular directions can not be tangential to
$NK_{q},$ and $0\neq\frac{d\overline{\gamma}}{dt}(t_{0})=\frac{d\overline
{\gamma}}{ds}\frac{ds}{dt}(t_{0}).$ Hence, one can reparametrize
$\overline{\beta}(t)=\beta(s)=(\gamma(s),R(s)v(s))$, with respect to the
arclength $s$ of $\gamma$ for $\left\vert s-s_{0}\right\vert <\varepsilon,$
and $s(t_{0})=s_{0},$ and still have a regular curve $\beta(s)$ such that
$\exp^{\mu}(\beta(s))=\exp^{\mu}(\gamma(s),R(s)v(s))$ is singular at
$s=s_{0}.$ The curve $\varphi(R)=\exp^{\mu}(\gamma(s_{0}),Rv(s_{0}))$ lies on
the sphere $\exp^{\mu}(NK_{q}\cap W)$ with center $c(s_{0})$ and it is normal
to the radial vectors from the center$.$ The curve $\eta(s)=\exp^{\mu}%
(\gamma(s),R(s_{0})v(s))$ satisfies Lemma 5(ii), and $p=\eta(s_{0}%
)=\varphi(R(s_{0})).$%
\begin{align*}
0  &  =\frac{d}{ds}\left.  \exp^{\mu}(\beta(s))\right\vert _{s=s_{0}}\\
&  =\frac{d}{ds}\left.  \exp^{\mu}(\gamma(s),R(s_{0})v(s))\right\vert
_{s=s_{0}}+\left.  \frac{dR}{ds}\right\vert _{s=s_{0}}\frac{d}{dR}\left.
\exp^{\mu}(\gamma(s_{0}),Rv(s_{0}))\right\vert _{R=R(s_{0})}\\
0  &  =\frac{d}{dR}\left.  \exp^{\mu}(\gamma(s_{0}),Rv(s_{0}))\right\vert
_{R=R(s_{0})}\cdot\left(  \varphi(R(s_{0}))-c(s_{0})\right) \\
0  &  =\frac{d}{ds}\left.  \exp^{\mu}(\gamma(s),R(s_{0})v(s))\right\vert
_{s=s_{0}}\cdot\left(  \eta(s_{0})-c(s_{0})\right) \\
&  =\frac{d\eta}{ds}(s_{0})\cdot\left(  \eta(s_{0})-c(s_{0})\right)
=\frac{\mu^{3}(s_{0})}{4\mu^{\prime}(s_{0})}F_{p}^{\prime\prime}(s_{0})
\end{align*}
This finishes the proof of $\left(  \Rightarrow\right)  $ in Case 1.

Assume that $F_{p}^{\prime\prime}(s_{0})=0$ where $\exp^{\mu}(q,w)=p$, and
$(q,w)\in int(W).$ Consider $\eta(s)=\exp^{\mu}(\gamma(s),Rv(s))$ where
$v(s):I\rightarrow UNK$ be $C^{1}$ with $v(s)\in UNK_{\gamma(s)}$ and
$R\in\mathbf{R}^{+}$ be such that $(\gamma(s),Rv(s))\in interior(W)$ for
$\left\vert s-s_{0}\right\vert <\varepsilon$, and $w=Rv(s_{0})$.
\[
0=\frac{\mu^{3}(s_{0})}{4\mu^{\prime}(s_{0})}F_{p}^{\prime\prime}(s_{0}%
)=\eta^{\prime}(s_{0})\cdot\left(  \eta(s_{0})-c(s_{0})\right)
\]

The non-zero vector $(\gamma^{\prime}(s_{0}),Rv^{\prime}(s_{0}))$ is not
tangential to $NK_{q}\cap int(W).$ $\eta^{\prime}(s_{0})$ is either zero or it
is normal to the radial vector $\eta(s_{0})-c(s_{0}).$ Therefore,
$\eta^{\prime}(s_{0})$ is tangent to the $n-1$ dimensional sphere
$\mathbf{S=}\exp^{\mu}(NK_{q}\cap W)$ at $p.$%
\[
d(\exp^{\mu})(q,w):T(NK)_{(q,w)}=T(NK_{q})_{w}\oplus\mathbf{R}\approx
\mathbf{R}^{n}\rightarrow T\mathbf{R}_{p}^{n}=T\mathbf{S}_{p}\oplus
\mathbf{R}\approx\mathbf{R}^{n}%
\]
\begin{align*}
d(\exp^{\mu})(q,w)|T(NK_{q})_{w}  &  :T(NK_{q})_{w}\rightarrow T\mathbf{S}%
_{p}\text{ is an isomorphism by (i)}\\
(\gamma^{\prime}(s_{0}),Rv^{\prime}(s_{0}))  &  \in T(NK)_{(q,w)}\\
(\gamma^{\prime}(s_{0}),Rv^{\prime}(s_{0}))  &  \notin T(NK_{q})_{w}\\
d(\exp^{\mu})(q,w)((\gamma^{\prime}(s_{0}),Rv^{\prime}(s_{0})))  &
=\eta^{\prime}(s_{0})\in T\mathbf{S}_{p}%
\end{align*}
\[
d(\exp^{\mu})(q,w):T(NK)_{(q,w)}\approx\mathbf{R}^{n}\rightarrow
T\mathbf{R}_{p}^{n}\approx\mathbf{R}^{n}\text{ is not one-to one.}%
\]
Therefore, $\exp^{\mu}$ is singular at $(q,w)$ to conclude the proof of
$\left(  \Leftarrow\right)  $ in Case 1.

\textbf{Case 2. }$\mu^{\prime}(s_{0})=0.$ The proof is essentially the same as
in Case 1 by replacing all \textquotedblleft$\cdot\left(  \eta(s_{0}%
)-c(s_{0})\right)  $\textquotedblright\ with \textquotedblleft$\cdot
\gamma^{\prime}(s_{0})\text{\textquotedblright}\ ,$ since $\exp^{\mu}(NK_{q})$
is an $n-1$ dimensional plane through $q=\gamma(s_{0})$ normal to
$\gamma^{\prime}(s_{0}),$ and one uses the first equation of Lemma 5,
$\eta^{\prime}(s_{0})\cdot\gamma^{\prime}(s_{0})=\frac{1}{2}\mu^{2}%
(s_{0})F_{p}^{\prime\prime}(s_{0})$ instead of the second equation.

iii. $\mathit{RegRad}(K,\mu)=FocRad^{0}(K,\mu)$ immediately follows (ii) and
the definitions. Combining Proposition 4, definitions of $DIR(K,\mu),$
$TIR(K,\mu),$ $LR(K,\mu)$ and $UR(K,\mu):$%
\begin{align*}
LR(K,\mu)  &  \leq TIR(K,\mu)\leq UR(K,\mu)\\
LR(K,\mu)  &  =\min\left(  \frac{1}{2}DCSD(K,\mu),FocRad^{0}(K,\mu)\right) \\
UR(K,\mu)  &  =\min\left(  \frac{1}{2}DCSD(K,\mu),FocRad^{-}(K,\mu)\right) \\
DIR(K,\mu)  &  \leq TIR(K,\mu)\leq\frac{1}{2}DCSD(K,\mu)\\
DIR(K,\mu)  &  \leq\mathit{RegRad}(K,\mu)=FocRad^{0}(K,\mu)\\
DIR(K,\mu)  &  \leq\min\left(  \frac{1}{2}DCSD(K,\mu),\mathit{RegRad}%
(K,\mu)\right)
\end{align*}
For all $0<r<\min\left(  \frac{1}{2}DCSD(K,\mu),\mathit{RegRad}(K,\mu)\right)
\leq TIR(K,\mu),$ $\exp^{\mu}$ restricted to $D(r)$ is a homeomorphism onto an
open subset $O(K,\mu r)$ of $\mathbf{R}^{n}$ by the proof of Proposition 4(i),
it is $C^{1}$ and non-singular, by Proposition 1. $\exp^{\mu}$ restricted to
$D(r)$ is a diffeomorphism, for all $0<r<\min\left(  \frac{1}{2}%
DCSD(K,\mu),\mathit{RegRad}(K,\mu)\right)  $, by the Inverse Function
Theorem.
\begin{align*}
DIR(K,\mu)  &  =\min\left(  \frac{1}{2}DCSD(K,\mu),\mathit{RegRad}%
(K,\mu)\right) \\
&  =\min\left(  \frac{1}{2}DCSD(K,\mu),FocRad^{0}(K,\mu)\right)  =LR(K,\mu)
\end{align*}

\end{proof}

\begin{lemma}
$LR(K,\mu)=UR(K,\mu)$ holds for $\mu$ on an open and dense subset of
$C^{3}(K,(0,\infty))$ in the $C^{3}-$ topology, for a fixed choice of
embedding $K\subset\mathbf{R}^{n}.$
\end{lemma}

\begin{proof}
For simplicity, we will assume that $K$ has one component. For a given onto
parametrization $\gamma:domain(\gamma)=\mathbf{R/}(lengthK)\mathbf{Z}%
\rightarrow K$, that is given $\kappa(s)$, define $X_{\kappa}=\left\{  \mu\in
C^{3}(K,(0,\infty)):0\text{ is a regular value of }\mu^{\prime\prime}%
+\frac{\kappa^{2}}{4}\mu\right\}  .$ This condition is equivalent to "the
graph of $\mu^{\prime\prime}+\frac{\kappa^{2}}{4}\mu$ intersects $s-$axis
transversally at every point of intersection" and it implies that $\left\{
s:\left(  \mu^{\prime\prime}+\frac{1}{4}\kappa^{2}\mu\right)  (s)=0\right\}  $
is a subset of the closure of $\left\{  s:\left(  \mu^{\prime\prime}+\frac
{1}{4}\kappa^{2}\mu\right)  (s)<0\right\}  $ to conclude that $FocRad^{0}%
(K,\mu)=FocRad^{-}(K,\mu).$ $X_{\kappa}$ is an open subset, since it is
defined by an open condition, regularity. $X_{\kappa}$ is dense in
$C^{3}(K,(0,\infty))$, if we prove that for every given $\mu,$ we have
$\mu_{\varepsilon}=\mu-\varepsilon\mu_{0}$ in $X_{\kappa}$ for almost all
small $\left\vert \varepsilon\right\vert $, for a fixed and appropriate choice
of $\mu_{0}.$ $\kappa$ can not be zero everywhere, since $K$ is compact.
Choose $\mu_{1}:domain(\gamma)\rightarrow(0,\infty)$ such that $\mu
_{1}^{\prime\prime}(s)>0$ on a proper open subinterval of $domain(\gamma),$
containing the points where $\kappa(s)=0.$ Choose $c_{1}>0 $ sufficiently
large so that $\mu_{0}=\mu_{1}+c_{1}$ satisfies that $\mu_{0}^{\prime\prime
}+\frac{\kappa^{2}}{4}\mu_{0}=\mu_{1}^{\prime\prime}+\frac{\kappa^{2}}{4}%
\mu_{1}+\frac{\kappa^{2}}{4}c_{1}>0.$ Let $f=\left(  \mu^{\prime\prime}%
+\frac{1}{4}\kappa^{2}\mu\right)  \left(  \mu_{0}^{\prime\prime}+\frac{1}%
{4}\kappa^{2}\mu_{0}\right)  ^{-1}:domain(\gamma)\rightarrow\mathbf{R.}$ By
Proposition 3(ii), $\mu^{\prime\prime}+\frac{\kappa^{2}}{4}\mu\leq0,\forall s$
is not possible. If $\mu^{\prime\prime}+\frac{\kappa^{2}}{4}\mu>0,\forall s $,
then $\mu\in X_{\kappa}$ which is open, and the proof is done. If $\mu
^{\prime\prime}+\frac{\kappa^{2}}{4}\mu>0,\forall s$ is not true, then $f $ is
not constant, and $range(f)=[a,b]$ with $a\leq0<b$. By Sard's Theorem [M], for
almost all $\varepsilon\in range(f),$ $\varepsilon$ is a regular value of $f$
(that is $f(s)=\varepsilon$ and $f^{\prime}(s)=0$ have no common roots).
Consequently, for the same $\varepsilon,$ $0$ is a regular value
of\ $\mu_{\varepsilon}^{\prime\prime}+\frac{1}{4}\kappa^{2}\mu_{\varepsilon
}=\mu^{\prime\prime}+\frac{1}{4}\kappa^{2}\mu-\varepsilon\left(  \mu
_{0}^{\prime\prime}+\frac{1}{4}\kappa^{2}\mu_{0}\right)  $. Hence,
$\mu_{\varepsilon}$ is in $X_{\kappa}$ for almost all small $\varepsilon$.
\end{proof}

\section{Examples}

We will use the pointwise focal radii for $\gamma(s)$ and $\mu(s)$ in the examples:

$FocRad^{0}(\gamma(s),\mu(s))=\Lambda(\kappa,\mu)(s)^{-\frac{1}{2}}$ if
$\Delta(\kappa,\mu)(s)\geq0,$ and $\left\vert \mu^{\prime}(s)\right\vert
^{-1}$ otherwise$.$

$FocRad^{-}(\gamma(s),\mu(s))=\Lambda(\kappa,\mu)(s)^{-\frac{1}{2}}$ if
$\Delta(\kappa,\mu)(s)>0,$ and $\left\vert \mu^{\prime}(s)\right\vert ^{-1}$
otherwise$.$%

\begin{figure}
[ptb]
\begin{center}
\includegraphics[
height=2.8617in,
width=3.2162in
]%
{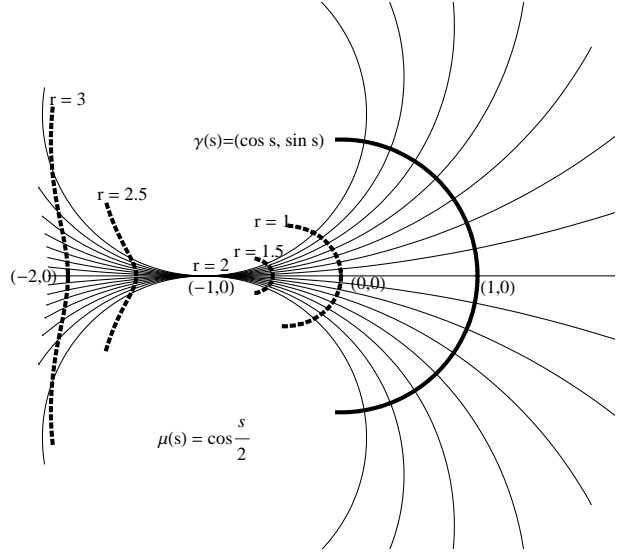}%
\caption{$\gamma(s)=(\cos s,\sin s)${\protect\small \ and }$\mu(s)=\cos
s/2${\protect\small . This figure depicts the Horizontal Collapsing Property
in dimension 2.}}%
\end{center}
\end{figure}

\begin{example}
A. Figure 7. Let $\gamma(s)=(\cos s,\sin s):\left(  -\frac{\pi}{2},\frac{\pi
}{2}\right)  \mathbf{\rightarrow}K\subset\mathbf{S}^{1}\subset\mathbf{R}^{2}$
and $\mu(s)=\cos\frac{s}{2}.$ $K$ is the half of $\mathbf{S}^{1}$ with $x>0.$
For all $s,$
\begin{align*}
\Delta(\kappa,\mu) &  =\mu\left(  \mu^{\prime\prime}+\frac{1}{4}\mu\right)
=0\\
\Lambda(\kappa,\mu) &  =\frac{1}{2}(\mu^{2})^{\prime\prime}+\frac{1}{2}\mu
^{2}=\frac{1}{4}\\
FocRad^{0}(K,\mu) &  =2\\
FocRad^{-}(K,\mu) &  =\inf\left\vert \mu^{\prime}(s)\right\vert ^{-1}%
=\inf2\left\vert \sin\frac{s}{2}\right\vert ^{-1}=2\sqrt{2}\\
FocRad^{0}(K,\mu) &  <FocRad^{-}(K,\mu)
\end{align*}
Since $\mu^{\prime}(0)=0$, $\exp^{\mu}(NK_{(1,0)})$ is the $x-axis.$ For
$s\neq0,$ $exp^{\mu}(NK_{\gamma(s)}\cap W)$ is a circle of radius $\left\vert
\frac{\mu}{2\mu^{\prime}}\right\vert =\left\vert \cot\frac{s}{2}\right\vert $
and with center $\gamma-\frac{\gamma^{\prime}\mu}{2\mu^{\prime}}=(-1,\cot
\frac{s}{2}).$ For $s\neq0$, all $exp^{\mu}$-circles are tangent to $x-axis$
at $(-1,0),$ and all intersecting $\mathbf{S}^{1}$ perpendicularly at both
points of intersection. For all $s$, $\exp^{\mu}(\gamma(s),2(-\cos s,-\sin
s))=(-1,0).$ Hence, $\exp^{\mu}$ is singular and not injective along the $R=2$
curve in $NK.$ However, $\exp^{\mu}$ is still injective for $R>2$. This type
of singularity does not occur for ($\mu=1$)-exponential map in which case
after the first focal point the exponential map is not injective.

B. Figure 5. Let $\gamma(s)=(\cos s,\sin s,0,...,0):[a,b]\mathbf{\rightarrow
}K\subset E_{12}\subset\mathbf{R}^{n}$ and $\mu(s)=\cos\frac{s}{2},$ where
$E_{12}$ is the $2-$ plane with $x_{i}=0$ for $i\geq3$ and $[a,b]\subset
(-\pi/2,\pi/2).$ $\exp^{\mu}(NK_{(1,0,..0)})$ is the $x_{2}=0$ hyperplane, and
all the spheres containing $\exp^{\mu}(NK_{q}\cap W)$ have centers on $E_{12}$
and $\exp^{\mu}(NK_{q}\cap W)\cap E_{12}$ are the circles discussed in part A.
Consequently, all $\exp^{\mu}(NK_{q}\cap W)$ are tangent to the plane
$\exp^{\mu}(NK_{(1,0,..,0)})$ at $(-1,0,0,..,0).$ The horizontal collapsing,
$\exp^{\mu}(\gamma(s),2N_{\gamma}(s))=(-1,0,0,..,0)$ is the only singularity,
since $\gamma^{\prime}$ and $\gamma^{\prime\prime}$ being parallel to $E_{12}$
implies that the singular set $Sng(K,\mu)\subset E_{12}$ by Proposition 8 of
Section 6.
\end{example}

\begin{example}
The open arc of Example 1A can be extended to a simple closed curve with an
appropriate $\mu$ to obtain examples with $TIR<UR.$ Let $C_{1}$ be the unit
circle centered at the origin. Given a small $\varepsilon>0,$ let $q_{1}%
^{+}=(\cos\varepsilon,\sin\varepsilon)\in C_{1}$ and $q_{1}^{-}=(\cos
\varepsilon,-\sin\varepsilon)$. Let $L^{+}$ and $L^{-}$ be the tangent lines
to $C_{1}$ at $q_{1}^{+}$ and $q_{1}^{-}$, respectively. Given a large $\ell,$
take $q_{2}^{+}\in L^{+}$ so that the line segment between $q_{1}^{+}$ and
$q_{2}^{+}$ has length $\ell$ and the $y-$coordinate $q_{2}^{+}$ is larger
than of $q_{1}^{+}$. Take $q_{2}^{-}\in L^{-}$ in a symmetric manner with
respect to the $x-$axis. Let $C_{2}$ be the circle tangent to $L^{+}$ at
$q_{2}^{+}$ and to $L^{-}$ at $q_{2}^{-}$. Consider the continuously
differentiable closed convex curve $\overline{\gamma}$ which is a
concatenation of $C_{1}$ between $q_{1}^{-}$ and $q_{1}^{+},$ $L^{+}$ between
$q_{1}^{+}$ and $q_{2}^{+},$ $C_{2}$ between $q_{2}^{+}$ and $q_{2}^{-}$, and
$L^{-}$ between $q_{2}^{-}$ and $q_{1}^{-}$. Let $\gamma$ be the smooth closed
curve which is the same as $\overline{\gamma}$ outside small $(0<\delta
\ll\varepsilon)$ $\delta-$neighborhoods $U_{i}^{\pm}$ of $q_{i}^{\pm},$ such
that the curvature is strictly monotone on each $U_{i}^{\pm},$ and $\gamma$ is
symmetric with respect to the $x-$axis. Parametrize $\gamma(s)$ with the
domain $[-A,A]$, $\gamma(0)=(1,0)$, arclength $s$, and take $K=\gamma
([-A,A]).$

We will construct $\mu$ so that $\mu(-s)=\mu(s).$ Let $\mu=\cos\frac{s}{2}$
for $\left\vert s\right\vert \leq2\varepsilon.$ For small $\varepsilon>0,$
$\mu(2\varepsilon)\approx1-\frac{\varepsilon^{2}}{2},$ $\mu^{\prime
}(2\varepsilon)\approx-\frac{\varepsilon}{2},$ and $\mu^{\prime\prime
}(2\varepsilon)\approx-\frac{1}{4}\left(  1-\frac{\varepsilon^{2}}{2}\right)
.$ By taking $\ell$ sufficiently large, one can extend $\mu$ smoothly to
$[0,A]$ so that $\frac{-1}{4}\leq\mu^{\prime\prime}\leq\frac{1}{20},$
$-\varepsilon\leq\mu^{\prime}\leq0,$ and $\frac{1}{4}\leq\mu\leq1$ over
$[2\varepsilon,\ell],$ and $\mu\equiv c_{0}\geq\frac{1}{4}$ on $[\ell-1,A].$
Observe that $\gamma(\ell)$ is on $L^{+}$ before $q_{2}^{+},$ and $\left\vert
\mu^{\prime}\right\vert \leq\varepsilon$ on all of $[-A,A]$.

On $[0,\varepsilon-\delta]:\Delta(\kappa,\mu)=0,$ $\Lambda(\kappa,\mu
)=\frac{1}{4},$ $FocRad^{0}(\gamma(s),\mu(s))=2$, and $\frac{4}{\varepsilon
}\leq\left\vert \mu^{\prime}(s)\right\vert ^{-1}=FocRad^{-}(\gamma
(s),\mu(s)).$ Moreover, for all $s\in\lbrack0,\varepsilon-\delta]$, $(-1,0)=$
$\exp^{\mu}(\gamma(s),2(-\cos s,-\sin s)).$ Hence, $\exp^{\mu}$ is singular
and not injective along the $R=2$ curve in $NK$ and $TIR(K,\mu)\leq2.$

On $(\varepsilon-\delta,\varepsilon+\delta):\Delta(\kappa,\mu)=\mu\left(
\mu^{\prime\prime}+\frac{1}{4}\kappa^{2}\mu\right)  <0,$ since $\kappa$ is
decreasing from $1$ to $0,$ and $\mu=\cos\frac{s}{2}$. Hence, $FocRad^{0}%
(\gamma(s),\mu(s))=FocRad^{-}(\gamma(s),\mu(s))\geq\frac{1}{\varepsilon}.$

On $[\varepsilon+\delta,\ell],$ $\kappa\equiv0.$ Hence, $\Lambda(\kappa
,\mu)=\frac{1}{2}(\mu^{2})^{\prime\prime}=\mu\mu^{\prime\prime}+\left(
\mu^{\prime}\right)  ^{2}\leq\frac{1}{20}+\varepsilon^{2}\leq\frac{1}{16},$ to
conclude that $FocRad^{0}(\gamma(s),\mu(s))=FocRad^{-}(\gamma(s),\mu
(s))\geq4.$ Observe that when $\mu\mu^{\prime\prime}+\left(  \mu^{\prime
}\right)  ^{2}<0,$ both pointwise radii are equal to $\left\vert \mu^{\prime
}(s)\right\vert ^{-1}.$

On $[\ell-1,A],$ $\mu\equiv c_{0}$. $\Delta(\kappa,\mu)=\frac{\kappa^{2}%
c_{0}^{2}}{4}$, $\Lambda(\kappa,\mu)=\kappa^{2}c_{0}^{2}$ and $FocRad^{0}%
(\gamma(s),\mu(s))$ $=FocRad^{-}(\gamma(s),\mu(s))\geq\frac{R_{2}}{c_{0}}$
where $R_{2}$ is the radius of $C_{2}.$

Overall, $FocRad^{0}(K,\mu)=2$ controlled by $C_{1}$ part and $FocRad^{-}%
(K,\mu)\geq4.$ For the double critical points $p$ and $q$ on $\gamma,$
$\cos\alpha(p,q)=-R\mu^{\prime}(p)$, and $\left\vert \mu^{\prime
}(p)\right\vert \leq\varepsilon.$ By taking $\varepsilon>0$ sufficiently small
and $\ell$ sufficiently large, one can keep $\alpha(p,q) $ close to $\frac
{\pi}{2}$ and $\frac{1}{2}DCSD\geq5.$ By Proposition 5(ii):

$DIR(K,\mu)=TIR(K,\mu)=2<4\leq UR(K,\mu).$
\end{example}

%

\begin{figure}
[ptb]
\begin{center}
\includegraphics[
height=2.8729in,
width=3.2292in
]%
{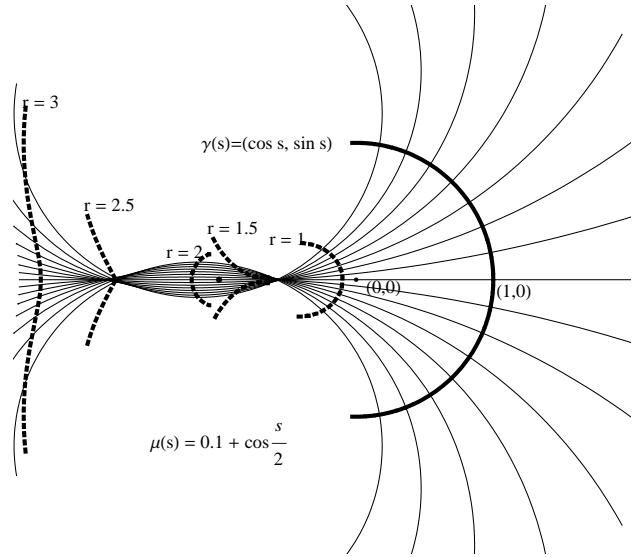}%
\caption{{\protect\small Compare the normal exponential maps from a portion of
the unit circle with }$\mu(s)=t+\cos s/2${\protect\small \ for }%
$t=0.1${\protect\small \ and }$t=-0.1${\protect\small \ with }$t=0$%
{\protect\small \ of Figure 7. The diagrams also show the curves of type
}$\exp^{\mu}(\gamma(s),rN(s)${\protect\small \ for some choices of }%
$r.${\protect\small \ Figures 7-9 together show the instability of DIR under
small perturbations}.}%
\end{center}
\end{figure}
\begin{figure}
[ptb]
\begin{center}
\includegraphics[
height=2.8158in,
width=3.1652in
]%
{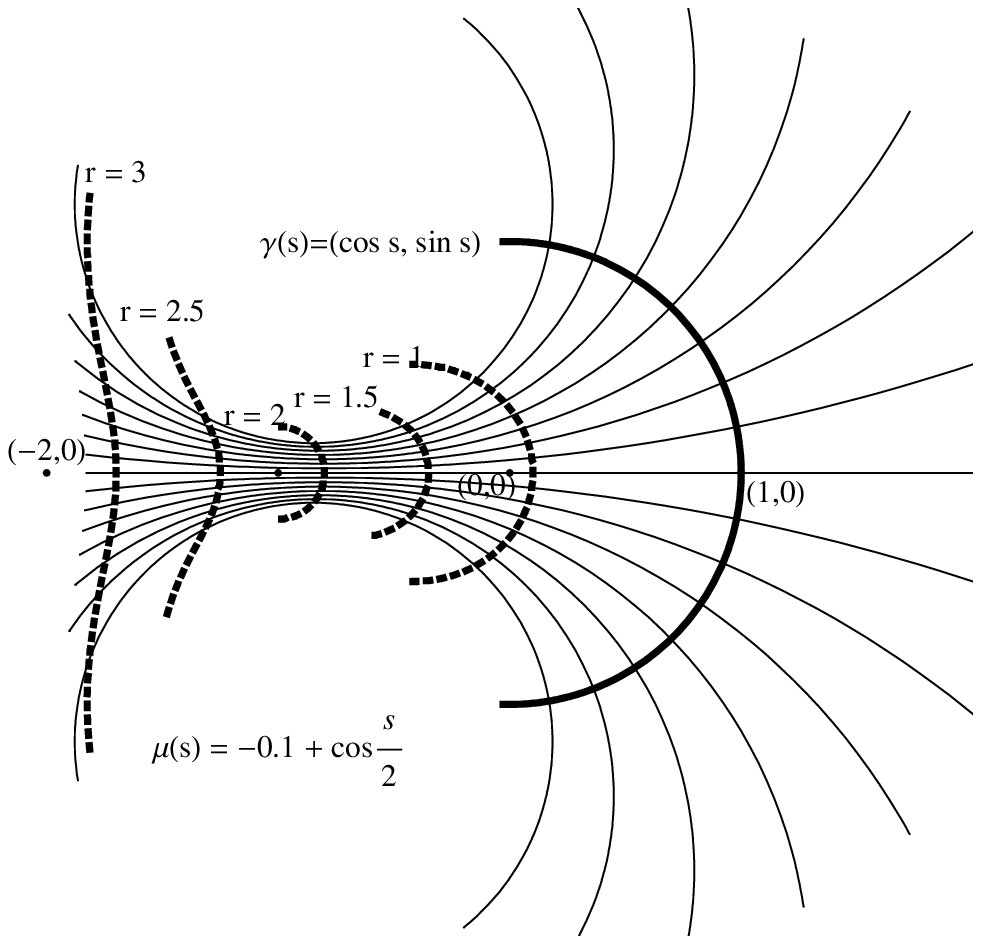}%
\caption{ \ \ }%
\end{center}
\end{figure}

\begin{example}
Figures \ 8 and 9. Let $\varepsilon,\ell,\gamma$ and $\mu$ be as in Example 2,
and $\mu_{t}(s)=t+\mu(s)=t+\cos\frac{s}{2}.$ For small $t>0,$ and $\left\vert
s\right\vert <\varepsilon-\delta,$ and $\kappa=1,$%
\begin{align*}
\Delta(\kappa,\mu_{t}) &  =\mu_{t}\left(  \mu_{t}^{\prime\prime}+\frac{1}%
{4}\mu_{t}\right)  >0\\
\Lambda(\kappa,\mu_{t}) &  =\frac{1}{2}(\mu_{t}^{2})^{\prime\prime}+\frac
{1}{2}\mu_{t}^{2}+\mu_{t}\sqrt{\Delta(\kappa,\mu_{t})}>\frac{1}{4}\\
FocRad^{-}(\gamma(s),\mu_{t}(s)) &  =FocRad^{0}(\gamma(s),\mu_{t}(s))<2
\end{align*}
On the interval $(\varepsilon-\delta,\varepsilon+\delta),$ $\mu=\cos\frac
{s}{2},$ but $\kappa$ starts to decrease to $0$ and $\Delta$ becomes
negative$.$ $\mu_{t}^{\prime\prime}+\frac{1}{4}\kappa^{2}\mu_{t}=\mu
^{\prime\prime}+\frac{1}{4}\kappa^{2}\left(  \mu+t\right)  =\frac{1}{4}\left(
\mu(\kappa^{2}-1)+t\kappa^{2}\right)  $ should have $0$ as a regular value for
almost all small $t$ to secure that $FocRad^{-}=FocRad^{0}$, see the proof of
Lemma 6. The effects of $t$ on the remainder of $\gamma$ and $DCSD$ are small.
Hence, for almost all small $t>0,$ $DIR(K,\mu_{t})=TIR(K,\mu_{t})=UR(K,\mu
_{t})<2.$

For small $t<0$ and $\left\vert s\right\vert <2\varepsilon:$%
\begin{align*}
\Delta(\kappa,\mu_{t})  &  =\mu_{t}\left(  \mu_{t}^{\prime\prime}+\frac{1}%
{4}\kappa^{2}\mu_{t}\right)  <0\\
FocRad^{0}(\gamma(s),\mu_{t}(s))  &  =FocRad^{-}(\gamma(s),\mu_{t}%
(s))\geq\frac{1}{\varepsilon}%
\end{align*}
The effects of $t$ on the remainder of $\gamma$ and $DCSD$ are small. For all
small $t<0:$
\begin{align*}
FocRad^{0}(K,\mu_{t})  &  =FocRad^{-}(K,\mu_{t})\geq3\\
DIR(K,\mu_{t})  &  =TIR(K,\mu_{t})=UR(K,\mu_{t})\geq3
\end{align*}
We see that $TIR$ and $DIR$ are not upper semicontinuous:
\begin{align*}
\underset{t\rightarrow0^{-}}{\lim\inf}DIR(K,\mu_{t})  &  =\underset
{t\rightarrow0^{-}}{\lim\inf}TIR(K,\mu_{t})\geq3>2=TIR(K,\mu)=DIR(K,\mu)\\
\underset{n\rightarrow\infty}{\lim}\text{ }UR(K,\mu_{t_{n}})  &  \leq2<4\leq
UR(K,\mu)\text{ for some sequence }0<t_{n}\rightarrow0.
\end{align*}

\end{example}

\begin{example}
Figure \ 10. Let $\gamma(s)=(\cos s,\sin s):\mathbf{R\rightarrow}%
K\subset\mathbf{S}^{1}\subset\mathbf{R}^{2}$ and $\mu(s)=1-\frac{s^{2}}{8}$
for $\left\vert s\right\vert <1.$ Observe that $0<\left(  \cos\frac{s}%
{2}\right)  -\left(  1-\frac{s^{2}}{8}\right)  =o(s^{3})$ for $s\neq0.$%
\begin{align*}
\forall s,\text{ }\Delta(\kappa,\mu) &  =\mu\left(  \mu^{\prime\prime}%
+\frac{1}{4}\mu\right)  =-\frac{1}{256}s^{2}(s^{2}-8)\leq0\\
\forall s,\text{ }\Lambda(\kappa,\mu) &  =\left\{
\begin{array}
[c]{cc}%
\frac{1}{4} & \text{if }s=0\\
\text{not a real number} & \text{if }s\neq0
\end{array}
\right.  \\
\forall s,\text{ }FocRad^{0}(\gamma(s),\mu(s)) &  =\left\{
\begin{array}
[c]{cc}%
2 & \text{if }s=0\\
\frac{4}{\left\vert s\right\vert } & \text{if }s\neq0
\end{array}
\right.  \\
\forall s,\text{ }FocRad^{-}(\gamma(s),\mu(s)) &  =\frac{1}{\left\vert
\mu^{\prime}(s)\right\vert }=\frac{4}{\left\vert s\right\vert }\\
FocRad^{0}(K,\mu) &  =2<4=FocRad^{-}(K,\mu)
\end{align*}
Since $\mu^{\prime}(0)=0$, $\exp^{\mu}(NK_{(1,0)})$ is the $x-axis.$ For
$s\neq0,$ $exp^{\mu}(NK_{\gamma(s)}\cap W)$ is a circle of radius $\left\vert
\frac{\mu}{2\mu^{\prime}}\right\vert =\frac{8-s^{2}}{4s}$ and with center
$(\cos s,\sin s)+\frac{8-s^{2}}{4s}(-\sin s,\cos s).$ $exp^{\mu}%
(NK_{\gamma(s)}\cap W)$ intersects $\mathbf{S}^{1}$ perpendicularly at both
$(\cos s,\sin s)\in K$ and $(\cos\theta(s),\sin\theta(s))\notin K$ where
$\theta(s):(-1,1)\rightarrow(\frac{\pi}{2},\frac{3\pi}{2})$ is a smooth
function, and
\[
\theta(s)=s+2\arctan\frac{8-s^{2}}{4s}\text{ and }\theta^{\prime}%
(s)=\frac{s^{2}(s^{2}-8)}{s^{4}+64},\text{ for }s>0.
\]
This shows that $\theta(s)$ is an injective function, but $\theta^{\prime
}(0)=0.$ All of the circles $exp^{\mu}(NK_{\gamma(s)}\cap W)$ are disjoint
from each other and the $x-axis.$ As $s\rightarrow0,$ the pointwise focal
radii tend to $\infty,$ and the circles converge to the $x-axis.$
Consequently, for all $\varepsilon$ with $0<\varepsilon<1,$ $exp^{\mu}((\cos
s,\sin s),R(-\cos s,-\sin s))$ is injective and a homeomorphism onto its image
for $\left\vert s\right\vert <\varepsilon$ and $\left\vert R\right\vert
<\frac{4}{\varepsilon}=\inf\frac{1}{\left\vert \mu^{\prime}\right\vert }.$
However, $\exp^{\mu}$ is singular at one isolated point
$(q,Rv)=((1,0),2(-1,0))$, $p=\exp^{\mu}((1,0),2(-1,0))=(-1,0).$ Hence, there
exists a non-closed curve with:

$2=DIR(K,\mu)<TIR(K,\mu)=\frac{4}{\varepsilon}$ and $0<\varepsilon<1.$
\end{example}

%

\begin{figure}
[ptb]
\begin{center}
\includegraphics[
height=2.738in,
width=3.077in
]%
{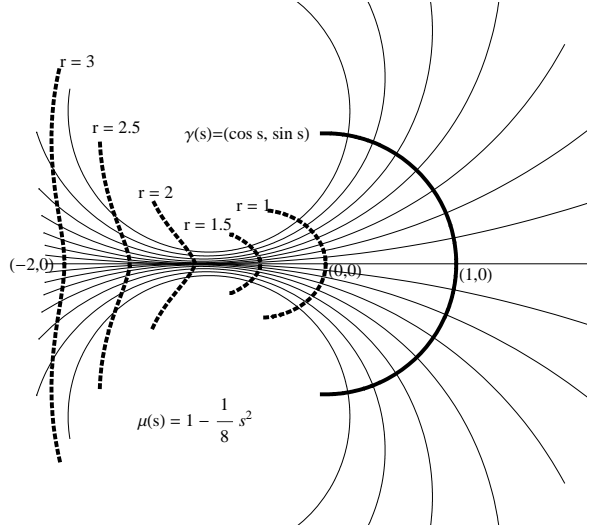}%
\caption{$\gamma(s)=(\cos s,\sin s)${\protect\small \ \ and }$\mu
(s)=1-\frac{s^{2}}{8}${\protect\small . This figure shows an exponential map
which is a local homeomorphism but not a local diffeomorphism near }%
$(-1,0)${\protect\small , See Example 4.}}%
\end{center}
\end{figure}

\begin{example}
Construct $\gamma$ and $\mu$ exactly in the same fashion as in Example 2, with
$\mu(s)=1-\frac{s^{2}}{8}$ instead of $\cos\frac{s}{2}$ on $(-2\varepsilon
,2\varepsilon).$ On $[\delta-\varepsilon,\varepsilon-\delta]$ one has
$\Delta(\kappa,\mu)=-\frac{1}{256}s^{2}(s^{2}-8)\leq0, $ $\Lambda(\kappa
,\mu)(0)=\frac{1}{4}.$ For $s=0,$ $FocRad^{0}(\gamma(0),\mu(s))=2$, and
$FocRad^{-}(\gamma(0),\mu(s))=\infty.$ For $s\neq0,$ $FocRad^{0}(\gamma
(s),\mu(s))=FocRad^{-}(\gamma(s),\mu(s))$ $=\frac{1}{\left\vert \mu^{\prime
}(s)\right\vert }\geq\frac{2}{\varepsilon}$. The remaining estimates are the
same as in Example 2. Overall, $FocRad^{0}(K,\mu)=2$ controlled only by one
point, $\gamma(0),$ and $FocRad^{-}(K,\mu)\geq4.$ Observe that there is only
one point $(q,Rv)$ where $p=\exp^{\mu}(q,Rv) $, $F_{p}^{\prime\prime}(s)=0,$
and $R<3,$ namely $((1,0),2(-1,0)). $ Suppose that $3>TIR(K,\mu)$ and repeat
the proof of Proposition 4. Since, $\frac{1}{2}DCSD\geq5,$ the only
possibilities left are the Cases 1 and 5. If both $y_{0}=z_{0}=\gamma
(0),~$then this would contradict the $\exp^{\mu}$ being a local homeomorphism
as discussed in Example 4. If $z_{0}\neq\gamma(0),$ then one still can repeat
the argument of Case 5, by finding $\mu-$closest point $q_{1}$ to $p_{1}$ by
using the fact that $\exp^{\mu}$ is a local homeomorphism again, to obtain a
double critical point, which is not the case. This shows that $DIR(K,\mu
)=2<3\leq TIR(K,\mu).$\pagebreak
\end{example}

%

\begin{figure}
[ptb]
\begin{center}
\includegraphics[
height=2.4933in,
width=3.109in
]%
{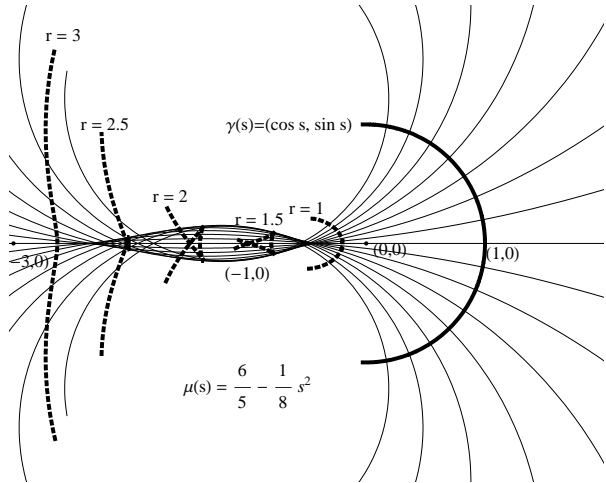}%
\caption{{\protect\small Compare the normal exponential maps from a portion of
the unit circle with }$\mu(s)=t+1-s^{2}/8$ {\protect\small for }$t=0.2$
{\protect\small and }$t=-0.05${\protect\small \ with }$t=0$%
{\protect\small \ of Figure 10. The diagrams also show the curves of type
}$\exp^{\mu}(\gamma(s),rN(s)${\protect\small \ for some choices of }$r.$
{\protect\small The example below is a local diffeomorphism. Figures 10-12
together show the instability of TIR under small perturbations.}}%
\end{center}
\end{figure}
\begin{figure}
[ptb]
\begin{center}
\includegraphics[
height=2.4267in,
width=2.7294in
]%
{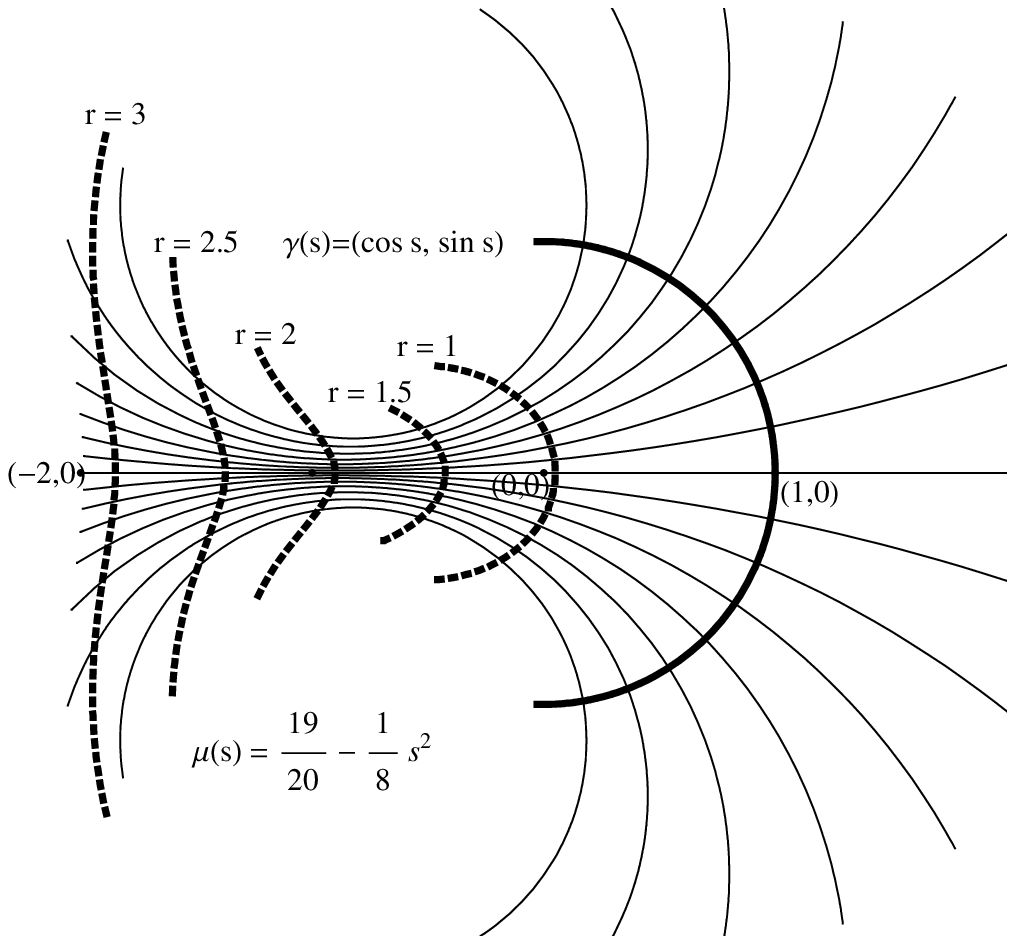}%
\caption{ \ }%
\end{center}
\end{figure}

\begin{example}
Figures\ 11 and 12. Let $\gamma(s)=(\cos s,\sin s):\mathbf{R\rightarrow
}K\subset\mathbf{S}^{1}\subset\mathbf{R}^{2}$ and $\mu_{t}(s)=t+1-\frac{s^{2}%
}{8}$ for $\left\vert s\right\vert <1=\varepsilon.$ For small $t>0,$
\begin{align*}
\Delta(\kappa,\mu_{t}) &  =\mu_{t}\left(  \mu_{t}^{\prime\prime}+\frac{1}%
{4}\mu_{t}\right)  >0\text{ for }\left\vert s\right\vert <\sqrt{8t}\\
\Lambda(\kappa,\mu_{t}) &  >\frac{1}{4}\text{ for }\left\vert s\right\vert
<\sqrt{8t}\text{ }\\
\Delta(\kappa,\mu_{t}) &  <0\text{ for }\sqrt{8t}<\left\vert s\right\vert
<1\text{ }%
\end{align*}%
\[
FocRad^{-}(\gamma(s),\mu_{t}(s))=FocRad^{0}(\gamma(s),\mu_{t}(s))<2\text{ for
}\left\vert s\right\vert <\sqrt{8t}%
\]%
\[
DIR(K,\mu_{t})=TIR(K,\mu_{t})<2
\]
For small $t<0$ and $\left\vert s\right\vert <1:$%
\begin{align*}
\Delta(\kappa,\mu_{t}) &  =\mu_{t}\left(  \mu_{t}^{\prime\prime}+\frac{1}%
{4}\mu_{t}\right)  <0\\
FocRad^{0}(\gamma(s),\mu_{t}(s)) &  =FocRad^{-}(\gamma(s),\mu_{t}(s))=\frac
{4}{\left\vert s\right\vert }\geq4
\end{align*}
Suppose that there is a double critical pair $(p,q)$ for $(K,\mu)$. Then, both
$\alpha(p,q)$ and $\alpha(q,p)$ must be larger than or equal to $\frac{\pi}%
{2},$ by Lemma 1. On $\gamma(s),$ $\mu(s)$ is increasing as $\left\vert
s\right\vert \rightarrow0.$ Hence, $\mathit{grad}\mu$ points in the direction
of $\gamma(0)=(1,0),$ and $\mathit{grad}\mu(0)=0.$ For any two points $p$ and
$q$ on $\gamma(s),$ $\left\vert s\right\vert <1$, the line segment joining
them can not make angle larger than or equal to $\frac{\pi}{2}$ with
$\mathit{grad}\mu$ at both end points, at least one of them is acute. Hence,
there is no double critical pair on $\gamma.$ For $t<0,$%
\[
DIR(K,\mu_{t})=TIR(K,\mu_{t})=4.
\]
Combining with Example 4, we see that $TIR$ and $DIR$ have different
semicontinuity properties:
\begin{align*}
\underset{t\rightarrow0^{-}}{\lim}DIR(K,\mu_{t}) &  =4>2=DIR(K,\mu
)\geq\underset{t\rightarrow0^{+}}{\text{ }\lim\sup}\text{ }DIR(K,\mu_{t})\\
\underset{t\rightarrow0^{-}}{\lim}TIR(K,\mu_{t}) &  =4=TIR(K,\mu
)>2\geq\underset{t\rightarrow0^{+}}{\text{ }\lim\sup}\text{ }TIR(K,\mu_{t})
\end{align*}

\end{example}

\section{$AIR$ and $TIR$}

The almost injectivity radius $AIR(K,\mu,\mathbf{R}^{n})$ is%
\[
\sup\left\{
\begin{array}
[c]{c}%
r:\exp^{\mu}:U(r)\rightarrow U_{0}(r)\text{ is a homeomorphism where
}U(r)\text{ is an open }\\
\text{and dense subset of }D(r),\text{ and }U_{0}(r)\text{ is an open subset
of }\mathbf{R}^{n}.
\end{array}
\right\}  .
\]
We observe that $\exp^{\mu}:D(r)\rightarrow O(K,\mu r)$ is a smooth onto map,
where both $D(r)$ and $O(K,\mu r)$ are open subsets $($for $r>0)$ of
$n-$dimensional manifolds. For $0<r<AIR(K,\mu)$ and all nonempty open subsets
$V$ of $D(r),$ $\exp^{\mu}(V\cap U(r))$ is a nonempty open subset of $O(K,\mu
r),$ and $\exp^{\mu}(V\cap U(r))$ is dense in $\exp^{\mu}(V).$ $\exp^{\mu}(V)$
is not necessarily open in $O(K,\mu r)$ when $V$ contains singular points of
$\exp^{\mu},$ see Figure 7 around $(-1,0)$.

\begin{proposition}
If $p_{0}=\exp^{\mu}(q_{1},R_{1}v_{1})=\exp^{\mu}(q_{2},R_{2}v_{2})\,$with
$v_{i}\in UNK_{q_{i}}$ for $i=1,2,$ and $0\leq\sqrt{G(p_{0})}=R_{2}<R_{1}$,
then $AIR(K,\mu)<R_{1}.$
\end{proposition}

\begin{proof}
Let $R_{0}=AIR(K,\mu).$ For $q\in K$ and $r>0$, let $A(q,r)$ denote the
connected component of $B(q,r;\mathbf{R}^{n})\cap K$ containing $q$ and
$A^{c}(q,r)=K-\overline{A(q,r)}.$ $A(q,r)$ is an open arc for small $r$.
First, we will show that $R_{1}\geq R_{0}.$

Suppose that $R_{1}<R_{0}.$ Let $\varepsilon=\frac{1}{3}\min(R_{0}-R_{1}%
,R_{1}-R_{2})>0$. Choose $\sigma>0$ such that
\begin{align*}
0  &  <\sigma<\mu(q_{1})\varepsilon\text{ and }\\
\max\left\{  \mu(q):q\in\overline{A(q_{1},\sigma)}\right\}   &  \leq\left(
1+\frac{\varepsilon}{R_{1}}\right)  \min\left\{  \mu(q):q\in\overline
{A(q_{1},\sigma)}\right\}  .
\end{align*}
We assert that $q_{2}\in A^{c}(q_{1},\sigma)$, since the assumption of
$q_{2}\in\overline{A(q_{1},\sigma)}$ leads to a contradiction as follows:
\begin{align*}
\sigma &  \geq\left\Vert q_{1}-q_{2}\right\Vert \\
&  \geq\left\Vert q_{1}-p_{0}\right\Vert -\left\Vert q_{2}-p_{0}\right\Vert \\
&  \geq R_{1}\mu(q_{1})-R_{2}\mu(q_{2})\\
&  \geq R_{1}\mu(q_{1})-R_{2}\left(  1+\frac{\varepsilon}{R_{1}}\right)
\mu(q_{1})\\
&  \geq\mu(q_{1})\left(  R_{1}-R_{2}-\frac{\varepsilon R_{2}}{R_{1}}\right) \\
&  \geq\mu(q_{1})\left(  3-\frac{R_{2}}{R_{1}}\right)  \varepsilon\\
&  >2\mu(q_{1})\varepsilon.
\end{align*}
We are given that $G(p_{0})=\min_{q\in K}F_{p_{0}}(q),$ and
\[
\sqrt{G(p_{0})}=R_{2}<R_{1}=\frac{\left\Vert p_{0}-q_{1}\right\Vert }%
{\mu(q_{1})}=\sqrt{F_{p_{0}}(q_{1})}.
\]
There exists a small open neighborhood $V_{0}$ of $p_{0}$ in $\mathbf{R}^{n}$,
such that $\overline{V_{0}}$ is compact with
\begin{align*}
\overline{V_{0}}  &  \subset B(q_{1},(R_{1}+\varepsilon)\mu(q_{1}%
);\mathbf{R}^{n})\cap B(q_{2},(R_{2}+\varepsilon)\mu(q_{2});\mathbf{R}%
^{n})\text{ and}\\
\forall p  &  \in\overline{V_{0}},\text{ }\sqrt{G(p)}\leq R_{2}+\varepsilon
<R_{1}-\varepsilon\leq\frac{\left\Vert p-q_{1}\right\Vert }{\mu(q_{1})}%
=\sqrt{F_{p}(q_{1})}.
\end{align*}
Therefore, there exists $0<\sigma_{0}<\sigma$ such that for every
$p\in\overline{V_{0}}$, each $\mu-$closest point $q_{2}(p)$ of $K$ to $p$
satisfies that $q_{2}(p)\in A^{c}(q_{1},\sigma_{0}),$ by an argument similar
to above for $q_{2}$ with $\varepsilon/3$ replacing $\varepsilon$ in the
choice of $\sigma_{0}.$ We choose $r$ such that $R_{1}+2\varepsilon<r<R_{0}$
and take:
\begin{align*}
D_{1}  &  =\{(q,w)\in NK:q\in A(q_{1},\sigma_{0})\text{ and }\left\Vert
w\right\Vert <r\},\\
D_{2}  &  =\{(q,w)\in NK:q\in A^{c}(q_{1},\sigma_{0})\text{ and }\left\Vert
w\right\Vert <r\},\text{ and }\\
V_{i}  &  =\left(  \exp^{\mu}\mid D_{i}\right)  ^{-1}(V_{0})\text{ for }i=1,2.
\end{align*}
Both $V_{1}$ and $V_{2}$ are open in $NK,$ $V_{1}\cap V_{2}\subset D_{1}\cap
D_{2}=\varnothing,$ but $(q_{i},R_{i}v_{i})\in V_{i}\neq\varnothing$ for
$i=1,2.$ The way $\sigma_{0}$ and $r$ were chosen above implies that
$\overline{V_{0}}\subset\exp^{\mu}(D_{2})$ and $\exp^{\mu}(V_{2})=V_{0}.$
Consequently, $\exp^{\mu}\left(  V_{2}\cap U(r)\right)  $ is a nonempty, open
and dense subset of $V_{0}.$ However, $\exp^{\mu}\left(  V_{1}\cap
U(r)\right)  $ is a nonempty, open (but not necessarily dense) subset of
$V_{0}$. Hence,
\begin{align*}
\exp^{\mu}\left(  V_{1}\cap U(r)\right)  \cap\exp^{\mu}\left(  V_{2}\cap
U(r)\right)   &  \neq\varnothing,\\
\text{but }V_{1}\cap V_{2}  &  =\varnothing.
\end{align*}
This contradicts the definition of $AIR.$ Hence, $AIR(K,\mu)=R_{0}\leq R_{1}.$

For sufficiently small $\delta>0,$ there is $\delta^{\prime}$ such that
$\exp^{\mu}(q_{1},(R_{1}-\delta)v_{1})=p_{1}$ satisfies that $\sqrt{G(p_{1}%
)}=R_{2}+\delta^{\prime}<R_{1}-\delta.$ There exists $q_{3}\in K$ and
$v_{3}\in UNK_{q_{3}}$ such that $p_{1}=\exp^{\mu}(q_{3},(R_{2}+\delta
^{\prime})v_{3})\,$. By the preceding part of the proof, $AIR(K,\mu)\leq
R_{1}-\delta<R_{1}.$\pagebreak
\end{proof}

\begin{corollary}
i. If $R<AIR(K,\mu),$ then $\exp^{\mu}(\partial D(R))=\partial O(K,\mu R). $

ii. If $\exp^{\mu}(q_{1},R_{1}v_{1})=\exp^{\mu}(q_{2},R_{2}v_{2})$ and
$R_{i}<AIR(K,\mu)$ for $i=1$ and $2,$ then $R_{1}=R_{2}.$

iii. If $R_{1}<R_{2}<AIR(K,\mu),$ then $\exp^{\mu}(\partial D(R_{1}))\cap
\exp^{\mu}(\partial D(R_{2}))=\varnothing.$
\end{corollary}

\begin{proof}
$\exp^{\mu}(D(R))=O(K,\mu R)=G^{-1}([0,R^{2}))$ and all are open subsets of
$\mathbf{R}^{n},$ for all $R>0,$ by Corollary 1 of Proposition 1.

i. If $p\in\partial O(K,\mu R)$ then $G(p)=R^{2}.$ Hence, $\partial O(K,\mu
R)\subset\exp^{\mu}(\partial D(R)).$ If there is $p\in\exp^{\mu}(\partial
D(R))$ which is an interior point of $O(K,\mu R),~$then by Proposition 6, one
would have $R>AIR(K,\mu).$

ii and iii immediately follow Proposition 6, and the fact that for every $p$
in $O(K,\mu R),$ there exists $q\in K$ and $v\in UNK_{q}$ such that
$p=\exp^{\mu}(q,rv)$ for some $r=\sqrt{G(p)}<R.$
\end{proof}

\begin{proposition}
i. $AIR(K,\mu)<\left(  \max_{q\in K}\left\Vert \mathit{grad}\mu(q)\right\Vert
\right)  ^{-1}<\infty,$ if $\mu$ is not constant.

ii. $AIR(K,\mu)\leq\left(  c_{0}\cdot\max_{q\in K}\kappa(q)\right)
^{-1}<\infty,$ if $\mu=c_{0}$ is constant.

iii. $TIR(K,\mu)\leq AIR(K,\mu)\leq UR(K,\mu).$
\end{proposition}

\begin{proof}
i. By Proposition 1(vi), $exp^{\mu}(NK_{q}\cap W)\cap K$ has a least two
distinct points, if $\mathit{grad}\mu(q)\neq0.$ Let $q^{\prime}$ $(\neq q) $
be another point of this set. Then, $q^{\prime}=\exp^{\mu}(q,Rv_{1})=\exp
^{\mu}(q^{\prime},0)$ for some $R\leq\left\Vert \mathit{grad}\mu(q)\right\Vert
^{-1}.$ By Proposition 6, $AIR(K,\mu)<R.$ Since $K$ is compact, $\max_{q\in
K}\left\Vert \mathit{grad}\mu(q)\right\Vert $ is attained on $K.$

ii. This is a part of the proof of (iii).

iii. First inequality follows the definitions.

Suppose there exists $R$ such that $FocRad^{-}(K,\mu)<R<AIR(K,\mu).$ Then,
there exists $p_{1}=\exp^{\mu}(q_{1},Rv_{1})$, for some $v_{1}\in UNK_{q_{1}}$
and $q_{1}\in CP(p_{1},-)$. As in the Claim 1 in the proof Proposition 4,
$G(p_{1})<R^{2},$ and $p_{1}=\exp^{\mu}(q_{2},R_{2}v_{2})$ for some
$(q_{2},R_{2}v_{2})\neq(q_{1},Rv_{1})$ with $R_{2}<R.$ This contradicts
Corollary 2(ii). Consequently, $AIR(K,\mu)\leq FocRad^{-}(K,\mu).$

We prove (ii) at this stage. If $\mu=c_{0},$ a positive constant, then
$\Delta(\kappa,c_{0})=\frac{1}{4}\kappa^{2}c_{0}^{2}\geq0,$ $\Lambda
(\kappa,c_{0})=\kappa^{2}c_{0}^{2}.$ Since $K$ is compact, there exists a
point $q_{0}$ of $K$ with maximal $\kappa(q_{0})>0.$ $AIR(K,\mu)\leq
FocRad^{-}(K,\mu)\leq\left(  \kappa(q_{0})c_{0}\right)  ^{-1}<\infty.$ If
$\mu$ is not constant, then $AIR(K,\mu)<\infty$ by (i).

Suppose that $\frac{1}{2}DCSD(K,\mu)=R_{0}<AIR(K,\mu).$ Let $AIR(K,\mu
)-R_{0}=\varepsilon>0.$ Since $K$ is compact, the set of critical points of
$\Sigma$ is a compact subset of $K\times K.$ Let $(q_{3},q_{4})$ be a minimal
double critical pair for $(K,\mu)$, with $p$ on the line segment
$\overline{q_{3}q_{4}}$ joining $q_{3}$ and $q_{4}$ such that $\left\Vert
p-q_{i}\right\Vert =R_{0}\mu(q_{i})$ and $p=\exp^{\mu}(q_{i},R_{0}v_{i})$ for
$i=3,4$. By Lemma 1 with $c=0$, $\alpha(q_{3},p)\in\left[  \frac{\pi}{2}%
,\pi\right]  $. First, we consider the case $\alpha(q_{3},p)>\frac{\pi}{2} $
where $\operatorname{grad}\mu(q_{3})\neq0.$ By part (i) and Proposition 1(ii),
$\alpha(q_{3},p)\neq\pi.$ The circular arc $\beta(s)=\exp^{\mu}(q_{3},sv_{3})$
is contained in the 2-plane containing $q_{3},$ $p$ and $q_{4}$ and parallel
to $v_{3}$. $\measuredangle(\beta^{\prime}(0),u(q_{3},p))=\measuredangle
(\beta^{\prime}(R_{0}),u(p,q_{4}))$ $=\alpha(q_{3},p)-\frac{\pi}{2}<\frac{\pi
}{2}.$ Since $\left\Vert q_{i}-p\right\Vert =\mu(q_{i})R_{0}$ for $i=3,4,$ one
has $\left\Vert q_{4}-\beta(R_{0}+s)\right\Vert \leq\left(  R_{0}-\lambda
s\right)  \mu(q_{4})<R_{0}\mu(q_{4})$ for some $\lambda>0$ and small enough
$\delta>s>0.$ In the case of $\alpha(q_{3},p)=\frac{\pi}{2},$ the last
statement still holds since $\beta(s)$ traces the line segment $\overline
{q_{3}q_{4}}. $ In all cases, choose $p_{0}=\beta(R_{0}+s_{0})$ such that
$0<s_{0}<\min(\varepsilon,\delta).$
\[
F_{p_{0}}(q_{3})=\left(  R_{0}+s_{0}\right)  ^{2}>\left(  R_{0}-\lambda
s_{0}\right)  ^{2}\geq F_{p_{0}}(q_{4})\geq G(p_{0})=F_{p_{0}}(q_{5})
\]
for some $q_{5}\in K.$ By Proposition 6, $AIR(K,\mu)<R_{0}+s_{0}%
<R_{0}+\varepsilon$ which contradicts the initial assumptions. Hence,
$AIR(K,\mu)=R_{0}\leq\frac{1}{2}DCSD(K,\mu).$
\end{proof}

\begin{proposition}
Let $K_{i}$ denote the components of $K.$ Let $\gamma_{i}:domain(\gamma
_{i})\rightarrow K_{i}$ be an onto parametrization of the component $K_{i}$
with unit speed and $\mu_{i}(s)=\mu(\gamma_{i}(s)).$ Then, the singular set
$Sng^{NK}(K,\mu)$ of $\exp^{\mu}$ within $D(UR(K,\mu))\subset NK$ is a graph
over a portion of $K$: $Sng^{NK}(K,\mu)=%
{\textstyle\bigcup\nolimits_{i}}
Sng_{i}^{NK}(K,\mu)$ and%
\[
Sng_{i}^{NK}(K,\mu)=\left\{
\begin{array}
[c]{c}%
(\gamma_{i}(s),R_{i}(s)N_{\gamma_{i}}(s))\in NK_{i}\text{ where}\\
s\in domain(\gamma_{i}),\text{ }\kappa_{i}(s)>0,\text{ }\\
\left(  \mu_{i}^{\prime\prime}+\frac{1}{4}\kappa_{i}^{2}\mu_{i}\right)
(s)=0\text{, and }\\
0<R_{i}(s)=\left(  \left(  \mu_{i}^{\prime}\right)  ^{2}-\mu_{i}\mu
_{i}^{\prime\prime}\right)  (s)^{-\frac{1}{2}}<UR(K,\mu)
\end{array}
\right\}
\]
where $\kappa_{i}$ and $N_{\gamma_{i}}$ are the curvature and the principal
normal of $\gamma_{i}$, respectively$.$ $D(UR(K,\mu))-Sng^{NK}(K,\mu)$ is
connected in each component of $NK,$ when $n\geq2.$
\end{proposition}

\begin{proof}
We will prove it for connected $K,$ and omit \textquotedblleft$i$%
\textquotedblright, since this is a local result. $R<UR(K,\mu)\leq\frac
{1}{\left\vert \mu^{\prime}(s)\right\vert },\forall s.$

$Sng^{NK}(K,\mu)=\left\{
\begin{array}
[c]{c}%
(q,Rv):v\in UNK_{q}\text{, }R<UR(K,\mu)\text{ }\\
\text{and the differential }d(\exp^{\mu})(q,Rv)\text{ is singular}%
\end{array}
\right\}  \subset int(W).$

For $q=\gamma(t),$ $v\in UN_{q},$ $p=\exp^{\mu}(q,Rv)$ and $R<FocRad^{-}%
(K,\mu):$%
\begin{equation}
0\leq\frac{d^{2}}{ds^{2}}\left.  F_{p}(\gamma(s))\right\vert _{s=t}=\frac
{2}{\mu^{2}(t)}\left(  1-\kappa R\mu\sqrt{1-\left(  \mu^{\prime}R\right)
^{2}}\cos\beta-\frac{R^{2}}{2}(\mu^{2})^{\prime\prime}\right)  (t)
\end{equation}
by Proposition 2, where $\beta=\measuredangle(\gamma^{\prime\prime
}(t),u(q,p)^{N})$ when both vectors are non-zero, and $\beta=0$ otherwise. By
proposition 5(ii),

$\exp^{\mu}$ is singular at $(q,Rv)$ if and only if $F_{p}^{\prime\prime
}(t)=0$, when the equality holds in (6.1). For fixed $q$ and $v,$ there is
only one possibility, a repeated root as Lemma 3(vi), to have a zero of (6.1)
and keeping (6.1) non-negative for all $0<R<UR(K,\mu)$.

\textbf{Case 1:} $\kappa(t)=0$. The quadratic in (6.1) can not have a repeated
root when $(\mu^{2})^{\prime\prime}(t)>0$ and it has no roots when $(\mu
^{2})^{\prime\prime}(t)\leq0$. Hence, it has no solution with $R<UR(K,\mu),$
and $Sng^{NK}(K,\mu)$ has no part over zero curvature points of $\gamma.$

\textbf{Case 2.} $\kappa(t)\neq0,$ with $N_{\gamma}(t)$ denoting the principal
normal of $\gamma$. If the expression in (6.1) were zero for $q=\gamma(t)$,
$R>0$ and a unit vector $v\neq N_{\gamma}(t)$ (that is $\cos\beta<1)$, then it
would be negative for the same $q$ and $R$ but $v_{1}=N_{\gamma}(t)$ (with
$\cos\beta_{1}=1$), which would imply that $R\geq UR(K,\mu).$ This proves that
$Sng^{NK}$ must be in the direction of the normal $N_{\gamma}$. In order have
a singular point at $(\gamma(t),Rv)$ and to satisfy (6.1), one must have
$v=N_{\gamma}(t)$ ($\cos\beta=1$) and there must be repeated roots as in Lemma
3(vi), which occur only when $\Delta(\kappa,\mu)=0:$
\begin{align*}
\Delta(\kappa,\mu)  &  =\frac{1}{2}(\mu^{2})^{\prime\prime}+\frac{1}{4}%
\kappa^{2}\mu^{2}-\left(  \mu^{\prime}\right)  ^{2}=\mu\mu^{\prime\prime
}+\frac{1}{4}\kappa^{2}\mu^{2}=0\\
\Lambda(\kappa,\mu)  &  =\frac{1}{2}(\mu^{2})^{\prime\prime}+\frac{1}{2}%
\kappa^{2}\mu^{2}=\left(  \mu^{\prime}\right)  ^{2}-\mu\mu^{\prime\prime}\\
\frac{1}{R^{2}}  &  =\Lambda(\kappa,\mu)(t)>0\text{ when }\kappa(t)>0.
\end{align*}
It is straightforward to show that points satisfying these conditions are the
singular points of $\exp^{\mu}$ within $D(UR(K,\mu))$. If $\mu=c_{0}$ is
constant and $\kappa>0$, then $\Delta(\kappa,\mu)>0$, and as $R$ increases,
the first zero of $F_{p}^{\prime\prime}(t)$ occurs at $R=c_{0}/\kappa(t)$ and
becomes negative for $R>c_{0}/\kappa(t).$ Consequently, $Sng^{NK}%
(K,\mu)=\varnothing$ when $\mu$ is constant. Since $K$ is compact, if $\mu$ is
not constant then there are points where $\mu^{\prime\prime}>0$ and
$\Delta>0.$ Hence, the domain of the graph $Sng^{NK}$ is not all of $K.$
Including the dimension \thinspace$n=2,$ the complement $D(UR)-Sng^{NK}$ is
connected in each component of $NK.$
\end{proof}

\begin{proposition}
$\exp^{\mu}$ restricted to $D(UR(K,\mu))-Sng^{NK}(K,\mu)$ is a diffeomorphism
onto its image in $\mathbf{R}^{n}$ and $AIR(K,\mu)=UR(K,\mu).$
\end{proposition}

\begin{proof}
Let $0<R_{1}<UR(K,\mu)$ be chosen arbitrarily. $\exp^{\mu}$ is a non-singular
map (local diffeomorphism) on $D(R_{1})-Sng^{NK}(K,\mu)$ which is an open
subset of $NK.$ Let $\mu_{\varepsilon}(s)=\mu(s)-\varepsilon$ for small
$\varepsilon>0.$

$\exists\varepsilon_{0}>0$ such that $\forall\varepsilon\in(0,\varepsilon
_{0}),$ $\exp^{\mu_{\varepsilon}}:D(R_{1})\rightarrow\mathbf{R}^{n}$ is a
non-singular map by the following. $\Delta(\kappa,\mu_{\varepsilon}%
)=\mu_{\varepsilon}\left(  \mu_{\varepsilon}^{\prime\prime}+\frac{1}{4}%
\kappa^{2}\mu_{\varepsilon}\right)  =\left(  \mu-\varepsilon\right)  \left(
\mu^{\prime\prime}+\frac{1}{4}\kappa^{2}\mu-\frac{1}{4}\kappa^{2}%
\varepsilon\right)  .$ On the parts of $K$ where $\mu^{\prime\prime}+\frac
{1}{4}\kappa^{2}\mu\leq0,$ and $\kappa>0$, one has $\Delta(\kappa
,\mu_{\varepsilon})<0$ and hence $\exp^{\mu_{\varepsilon}}$ is non-singular
for all small $\varepsilon>0,$ by Propositions 3 and 5. On the parts of $K$
where $\mu^{\prime\prime}+\frac{1}{4}\kappa^{2}\mu\leq0$ and $\kappa=0,$
$\exp^{\mu_{\varepsilon}}$ is non-singular within radius of $UR(K,\mu
_{\varepsilon})\leq FocRad^{-}(K,\mu_{\varepsilon})$, see the Case 1 in the
proof of Proposition 8. On the parts of $K$ where $\mu^{\prime\prime}+\frac
{1}{4}\kappa^{2}\mu>0,$ one has $\Lambda(\kappa,\mu)^{-\frac{1}{2}}\geq
UR(K,\mu).$ Observe that $\Delta(\kappa,\mu_{\varepsilon})(s_{0})>0$ implies
that $\Delta(\kappa,\mu)(s_{0})>0$, and by Proposition 3(ii) both inequalities
must be valid at some common points on $K.$ By continuity, $\exists
\varepsilon_{0}>0,\forall\varepsilon\in(0,\varepsilon_{0}),$ $\Lambda
(\kappa,\mu_{\varepsilon})^{-\frac{1}{2}}\geq R_{1}$ and $Sng^{NK}%
(\mu_{\varepsilon})\cap D(R_{1})=\varnothing,$ by Propositions 3, 8, and
Definitions 4, 9. Consequently, $\exp^{\mu_{\varepsilon}}:D(R_{1}%
)\rightarrow\mathbf{R}^{n}$ is a non-singular map.

Suppose that $\exp^{\mu}$ is not one-to-one on $D(R_{1})-Sng^{NK}(K,\mu)$, and
there exist $(q_{i},w_{i})\in D(R_{1})-Sng^{NK}(K,\mu)$ for $i=1,2$ such that
$(q_{1},w_{1})\neq(q_{2},w_{2})$ but $\exp^{\mu}(q_{1},w_{1})=\exp^{\mu}%
(q_{2},w_{2}).$ By the regularity of $\exp^{\mu} $ on $D(R_{1})-Sng^{NK}%
(K,\mu)$, there exists open sets $U_{i}$ such that $(q_{i},w_{i})\in
U_{i}\subset D(R_{1})-Sng^{NK}(K,\mu)$ for $i=1,2$, $U_{1}\cap U_{2}%
=\varnothing$, $\exp^{\mu}(U_{1})=\exp^{\mu}(U_{2})$ and $\exp^{\mu}\mid
U_{i}$ are diffeomorphisms. $\left\{  \exp^{\mu_{\varepsilon}}:\varepsilon
>0\right\}  $ converge uniformly to $\exp^{\mu} $ on $\overline{D(R_{1})}$ as
$\varepsilon\rightarrow0^{+},$ by the definition of $\exp^{\mu}.$ Since
$\exp^{\mu_{\varepsilon}}(U_{1})$ and $\exp^{\mu_{\varepsilon}}(U_{2})$ are
open subsets of $\mathbf{R}^{n}$ and $\exp^{\mu}(U_{1})=\exp^{\mu}(U_{2}),$
$\exists\varepsilon_{1}>0,\forall\varepsilon\in(0,\varepsilon_{1})$,
$\exp^{\mu_{\varepsilon}}(U_{1})\cap\exp^{\mu_{\varepsilon}}(U_{2}%
)\neq\varnothing.$ Consequently, $\exp^{\mu_{\varepsilon}}:D(R_{1}%
)\rightarrow\mathbf{R}^{n}$ is not injective. By Proposition 5(iii),
$DIR(K,\mu_{\varepsilon})=\frac{1}{2}DCSD(K,\mu_{\varepsilon})\leq
R_{1},\forall\varepsilon\in(0,\min(\varepsilon_{0},\varepsilon_{1})).$ There
exist pairs of points $(x_{\varepsilon},y_{\varepsilon})\in K\times K$ with
$x_{\varepsilon}\neq y_{\varepsilon}$, $\mathit{grad}\Sigma_{\varepsilon
}(x_{\varepsilon},y_{\varepsilon})=0,$ and $\frac{\left\Vert x_{\varepsilon
}-y_{\varepsilon}\right\Vert }{\mu(x_{\varepsilon})+\mu(y_{\varepsilon}%
)}=\frac{1}{2}DCSD(K,\mu_{\varepsilon})$ where $\Sigma_{\varepsilon}:K\times
K\rightarrow\mathbf{R}$ defined by $\Sigma_{\varepsilon}(x,y)=\left\Vert
x-y\right\Vert ^{2}(\mu_{\varepsilon}(x)+\mu_{\varepsilon}(y))^{-2}.$ By
compactness and taking convergent subsequences (and using $x_{j}$, $y_{j}$ and
$\mu_{j}$ for simplifying the subindices), there exists $(x_{j},y_{j}%
)\rightarrow(x_{0},y_{0})\in K\times K$ with $\mathit{grad}\Sigma(x_{0}%
,y_{0})=0.$ Suppose that $x_{0}=y_{0}$. As $R_{j}=\left\Vert x_{j}%
-y_{j}\right\Vert \left(  \mu(x_{j})+\mu(y_{j})\right)  ^{-1}\rightarrow0,$
one has $\cos\alpha(x_{j},y_{j})=-R_{j}\left\vert \mu_{j}^{\prime}%
(x_{j})\right\vert =-R_{j}\left\vert \mu^{\prime}(x_{j})\right\vert
\rightarrow0$, which means that the line through $x_{j}$ and $y_{j}$ is making
an angle close to $\pi/2$ with $K$ at $x_{j}$ and $y_{j}$. On the other hand,
$(x_{j},y_{j})\rightarrow(x_{0},x_{0})$ implies that the same lines are
converging to a line tangent to $K.$ Both can not happen simultaneously.
Hence, $x_{0}\neq y_{0},$ and $(x_{0},y_{0})$ is a critical pair for
$(K,\mu).$ By the definition of $DCSD$ and continuity, $\frac{1}{2}%
DCSD(K,\mu)\leq\frac{\left\Vert x_{0}-y_{0}\right\Vert }{\mu(x_{0})+\mu
(y_{0})}\leq R_{1}. $ However, this contradicts our initial assumption of
$R_{1}<UR(K,\mu)\leq\frac{1}{2}DCSD(K,\mu).$ Finally, $\forall R_{1}%
<UR(K,\mu),$ $\exp^{\mu}$ is one-to-one on $D(R_{1})-Sng^{NK}(K,\mu),$ and it
is a non-singular map onto an open subset of $\mathbf{R}^{n}$. This proves
that $\exp^{\mu}\mid D(UR(K,\mu))-Sng^{NK}(K,\mu)$ is a diffeomorphism onto
its image. $Sng^{NK}(K,\mu)$ has an empty interior, since it is a subset of a
one-dimensional graph over a subset of $K.$ By the definitions and Proposition
7, $AIR(K,\mu)=UR(K,\mu).$\pagebreak
\end{proof}

\begin{corollary}
Let $(K,\mu)$ be given and $\mu_{\varepsilon}(s)=\mu(s)-\varepsilon$. For a
given $0<R_{1}<UR(K,\mu),$ $\exists\varepsilon^{\prime}>0$ such that
$\forall\varepsilon\in(0,\varepsilon^{\prime}),$ $\exp^{\mu_{\varepsilon}%
}:D(R_{1})\rightarrow O(K,\mu_{\varepsilon}R_{1})$ is a diffeomorphism. The
diffeomorphisms $\exp^{\mu_{\varepsilon}}$ converge uniformly to the (possibly
singular) map $\exp^{\mu}$ as $\varepsilon\rightarrow0^{+}$, on $\overline
{D(R_{1})}.$
\end{corollary}

\begin{proof}
This follows the proof of Proposition 9. First, the regularity part is done in
the same way. Then, one supposes that such $\varepsilon^{\prime}$ does not
exist, and for all $j\in\mathbf{N}^{+}\mathbf{,}$ there exist $0<\varepsilon
_{j}\leq\frac{1}{j}$ with a non-singular and non-injective map $\exp
^{\mu_{\varepsilon_{j}}}:D(R_{1})\rightarrow\mathbf{R}^{n}$. One follows the
proof above again, by using the limits of subsequences of double critical
pairs of $(K,\mu_{\varepsilon_{j}}),$ to obtain a double critical pair for
$(K,\mu)$ to contradict $R_{1}<UR(K,\mu)\leq\frac{1}{2}DCSD(K,\mu).$
\end{proof}

\begin{proposition}
For a given $(K,\mu)$ and $q\in K,$ let
\begin{align*}
Sng  &  =\exp^{\mu}(Sng^{NK}),\\
A_{q}  &  =\exp^{\mu}\left(  NK_{q}\cap D(UR)\right)  ,\text{ and}\\
A_{q}^{\ast}  &  =\exp^{\mu}\left(  NK_{q}\cap D(UR)-Sng^{NK}\right)  .
\end{align*}
Then, i. $O(K,\mu UR)-Sng$ has a codimension 1 foliation by $A_{q}^{\ast}$,
which are (possibly punctured) spherical caps or discs.$\ $

ii. $\exp^{\mu}(D(UR)-Sng^{NK})=O(K,\mu UR)-Sng.$

iii. If $A_{q_{1}}\cap A_{q_{2}}\neq\varnothing$ for $q_{1}\neq q_{2}$ then
$q_{1}$ and $q_{2}$ must belong to the same component of $K,$ and $A_{q_{1}}$
intersects $A_{q_{2}}$ tangentially at exactly one point $p_{0}=\exp^{\mu
}(q_{1},r_{1}v_{1})=\exp^{\mu}(q_{2},r_{2}v_{2})$ where $(q_{i},r_{i}v_{i})\in
Sng^{NK},$ for $i=1,2.$

iv. \textbf{Horizontal Collapsing Property}:

Assume that $\exp^{\mu}(q_{1},r_{1}v_{1})=\exp^{\mu}(q_{2},r_{2}v_{2})=p_{0}$
for $r_{1}$, $r_{2}<UR(K,\mu)$, $v_{i}\in UNK_{q_{i}}$ with $(q_{1},r_{1}%
v_{1})\neq(q_{2},r_{2}v_{2})$. Then, $q_{1}$ and $q_{2}$ belong to the same
component of $K,$ which is denoted by $K_{1}.$ Let $\gamma
(s):\mathbf{R\rightarrow}K_{1}\subset\mathbf{R}^{n}$ be a unit speed
parametrization of $K_{1}$ such that $\gamma(s+L)=\gamma(s)$ where $L $ is the
length of $K_{1},$ $N_{\gamma}(s)$ denotes the principal normal of $\gamma,$
and $q_{i}=\gamma(s_{i})$ for $i=1,2$ with $0\leq s_{1}<s_{2}<L.$ Then,
$r_{1}=r_{2},$ $v_{i}=N_{\gamma}(s_{i})$ for $i=1,2,$ and $\exp^{\mu}%
(\gamma(s),r_{1}N_{\gamma}(s))=p_{0}$, $\forall s\in I$ where $I=[s_{1}%
,s_{2}]$ or $[s_{2}-L,s_{1}].$
\end{proposition}

\begin{proof}
The logical order of the proof is different from the presentation order of the results.

For different components $K_{1}$ and $K_{2}$ of $K$, the open sets
$O(K_{1},\mu R)$ and $O(K_{2},\mu R)$ are disjoint for $R<UR(K,\mu),$
otherwise one can obtain a contradiction with Propositions 8 and 9$.$
$\exp^{\mu}\mid D(UR)-Sng^{NK}$ is a diffeomorphism onto its image. $\exp
^{\mu}\mid NK_{q}\cap D(UR)$ is also a diffeomorphism where the image $A_{q}$
is an open (metric) disc of an $n-1$ dimensional plane or sphere. By
Proposition 8, $\exp^{\mu}\left(  Sng^{NK}\cap NK_{q}\right)  $ contains at
most one point denoted by $q^{\ast}$, if it exists. If such $q^{\ast}$ does
not exist, we use $\{q^{\ast}\}=\varnothing.$ Let $A_{q}^{\ast}=A_{q}%
-\{q^{\ast}\}$. The diffeomorphism $\exp^{\mu}\mid D(UR)-Sng^{NK}$ carries the
codimension $1$ foliation of $D(UR)-Sng^{NK}$ by $NK_{q}-Sng^{NK} $ to a
codimension $1$ foliation of $\exp^{\mu}(D(UR)-Sng^{NK})$ by $A_{q}^{\ast}$.

As in Corollary 3, let $\mu_{\varepsilon}(s)=\mu(s)-\varepsilon$ for small
$\varepsilon>0$ and choose large $R_{1}<UR(K,\mu).$ By Proposition 9,
$A_{q_{1}}^{\ast}\cap A_{q_{2}}^{\ast}=\varnothing$ for $q_{1}\neq q_{2}.$
Therefore, $A_{q_{1}}\cap A_{q_{2}}\subset\left\{  q_{1}^{\ast},q_{2}^{\ast
}\right\}  $ for $q_{1}\neq q_{2}.$ Suppose that $A_{q_{1}}$ and $A_{q_{2}}$
intersect transversally. For $n\geq3,$ $A_{q_{1}}\cap A_{q_{2}}$ would have
infinitely many points, which is not the case. In all dimensions including
$n=2,$ take $R_{1}<UR(K,\mu)$ sufficiently large with $\left\{  q_{1}^{\ast
},q_{2}^{\ast}\right\}  \subset O(K,\mu R_{1}).$ By Corollary 3, $A_{q_{1}%
}(\mu_{\varepsilon})\cap A_{q_{2}}(\mu_{\varepsilon})=\varnothing,$ for
sufficiently small $\varepsilon>0.$ In the limit as $\varepsilon
\rightarrow0^{+},$ $A_{q_{1}}$ and $A_{q_{2}}$ can not intersect
transversally, since transversality is an open condition. Hence, $A_{q_{1}}$
and $A_{q_{2}}$ are tangential to each other at $q_{1}^{\ast}$ or $q_{2}%
^{\ast}$ and there is only one point of intersection for $q_{1}\neq q_{2}$, if
the intersection is not empty. If both $A_{q_{1}}$ and $A_{q_{2}}$ are subsets
of hyperplanes, then $A_{q_{1}}\cap A_{q_{2}}=\varnothing$ for $q_{1}\neq
q_{2}.$

From this point on, assume that $p_{0}=\exp^{\mu}(q_{1},r_{1}v_{1})=\exp^{\mu
}(q_{2},r_{2}v_{2})$, for $q_{1}\neq q_{2}.$ $A_{q_{1}}$ and $A_{q_{2}} $ must
intersect tangentially at $p_{0}\in\left\{  q_{1}^{\ast},q_{2}^{\ast}\right\}
,$ and $q_{1}$ and $q_{2}$ must belong to the same component of $K$, denoted
by $K_{1}$. At least one of $A_{q_{i}}$ is spherical. Choose $A_{q_{1}}$ to be
the subset of the sphere with center $c_{1}$ and the smaller radius
$\sigma_{1}$ so that $\operatorname{grad}\mu(q_{1})\neq0.$ Then, $\forall p\in
A_{q_{2}}$, $\left\Vert c_{1}-p\right\Vert \geq\sigma_{1}$. Let $\gamma
(s):\mathbf{R\rightarrow}K_{1}\subset\mathbf{R}^{n}$ be a unit speed
parametrization such that $\gamma(s+L)=\gamma(s)$ where $L$ is the length of
$K_{1},$ and $q_{i}=\gamma(s_{i})$ for $i=1,2$ with $0\leq s_{1}<s_{2}<L.$ Let
$\eta(s)=\exp^{\mu}(\gamma(s),Rv(s))$ be as in Lemma 5:
\begin{align*}
\eta^{\prime}(s_{1})\cdot\left(  \eta(s_{1})-c(s_{1})\right)   &  =\frac
{\mu^{3}(s_{1})}{4\mu^{\prime}(s_{1})}\frac{d^{2}}{ds^{2}}\left.
F_{\eta(s_{1})}(\gamma(s))\right\vert _{s=s_{1}}\text{ since }\mu^{\prime
}(s_{1})\neq0\\
\text{where }c(s_{1})  &  =c_{1}=\gamma(s_{1})-\frac{\mu(s_{1})}{2\mu^{\prime
}(s_{1})}\gamma^{\prime}(s_{1})
\end{align*}
We will assume that $\mu^{\prime}(s_{1})>0,$ and work on the interval
$[s_{1},s_{2}].$ Otherwise, if $\mu^{\prime}(s_{1})<0,$ then one
reparametrizes $K_{1}$ to traverse $\gamma\left(  \lbrack s_{2}-L,s_{1}%
]\right)  $ with opposite orientation starting at $q_{1}$. Choose
$R_{1}<UR(K,\mu)$ sufficiently large with $\left\{  q_{1}^{\ast},q_{2}^{\ast
}\right\}  \subset O(K_{1},\mu R_{1}).$

\textbf{Claim 1. }There exists $\delta>0$ such that

$\forall s\in(s_{1},s_{1}+\delta),$ $\forall p\in A_{\gamma(s)}\cap
O(K_{1},R_{1}\mu)$, $d(c_{1},p)\geq\sigma_{1}.$

For a given curve $(\gamma(s),Rv(s))$ in $NK_{1}$ as in Lemma 5, define
\begin{align*}
\eta_{Rv}(s)  &  =\exp^{\mu}(\gamma(s),Rv(s))\text{ and }\\
f_{Rv}(s)  &  =\left\Vert \eta_{Rv}(s)-c_{1}\right\Vert ^{2}\text{ so that}\\
f_{Rv}(s_{1})  &  =\sigma_{1}^{2}>0\text{ and }f_{Rv}^{\prime}(s_{1}%
)=2\eta_{Rv}^{\prime}(s_{1})\cdot\left(  \eta_{Rv}(s_{1})-c_{1}\right)
\text{.}%
\end{align*}%
\begin{align*}
f_{Rv}^{\prime}(s_{1})  &  >0\text{ if }\eta_{Rv}(s_{1})\in A_{q_{1}}^{\ast}\\
f_{Rv}^{\prime}(s_{1})  &  =0\text{ if }\eta_{Rv}(s_{1})=q_{1}^{\ast}%
\end{align*}
(In the next two statements, the compactness of $\left(  A_{q_{1}}%
-B(q_{1}^{\ast},\delta_{1})\right)  \cap\overline{O(K_{1},R_{1}\mu)}$ is
essential.)%
\begin{align*}
\forall\delta_{1}  &  >0,\exists\delta_{2}>0\text{ such that }\\
\text{if }\eta_{Rv}(s_{1})  &  \in\left(  A_{q_{1}}-B(q_{1}^{\ast},\delta
_{1})\right)  \cap\overline{O(K_{1},R_{1}\mu)}\text{ then }f_{Rv}^{\prime
}(s_{1})\geq\delta_{2}>0\text{.}%
\end{align*}%
\begin{align*}
\exists\delta &  >0\text{ such that }\delta\ll\min(R_{1},r_{1},R_{1}%
-r_{1})\text{ and}\\
\text{if }\eta_{Rv}(s_{1})  &  \in\left(  A_{q_{1}}-B(q_{1}^{\ast},\delta
_{1})\right)  \cap\overline{O(K_{1},R_{1}\mu)}\text{ and }s\in(s_{1}%
,s_{1}+\delta),\\
\text{then }f_{Rv}(s)  &  >\sigma_{1}^{2}.
\end{align*}
Suppose there exists $Rv(s)$ with $\eta_{Rv}(s_{1})\in A_{q_{1}}\cap
B(q_{1}^{\ast},\delta_{1})\cap O(K_{1},R_{1}\mu)$, $s^{\prime}\in(s_{1}%
,s_{1}+\delta)$ and $f_{Rv}(s^{\prime})<\sigma_{1}^{2}.$ Then, $A_{\gamma
(s^{\prime})}$ must intersect $A_{q_{1}}$ near $q_{1}^{\ast}.$ This
intersection must be tangential as discussed above with $q_{1}$ and $q_{2}$.
However, this cannot be the case when $f_{Rv}(s)$ takes values on both sides
of $\sigma_{1}^{2}.$ This proves the Claim 1:
\begin{align*}
\exists\delta &  >0\text{ such that }\\
\text{if }\eta_{Rv}(s_{1})  &  \in A_{q_{1}}\cap O(K_{1},R_{1}\mu)\text{ and
}s\in(s_{1},s_{1}+\delta)\text{ then }f_{Rv}(s)\geq\sigma_{1}^{2},\text{
hence,}\\
\forall s  &  \in(s_{1},s_{1}+\delta),\forall p\in A_{\gamma(s)}\cap
O(K_{1},R_{1}\mu),\left\Vert c_{1}-p\right\Vert \geq\sigma_{1}.
\end{align*}

Recall that $\forall p\in A_{q_{2}}$, $\left\|  c_{1}-p\right\|  \geq
\sigma_{1}$ and $A_{q_{2}}$ is tangent to $A_{q_{1}}$ at $p_{0}$. To avoid any
transversal intersections with $A_{q_{2}},$ $A_{\gamma(s)}$ must stay between
the codimension 1 submanifolds (sphere or plane) containing $A_{q_{1}}$ and
$A_{q_{2}},$ respectively$.$ This forces $A_{\gamma(s)}$ to be tangent to
$A_{q_{1}}$ at $p_{0}$ for $\forall s\in(s_{1},s_{1}+\delta), $ which is still
true on $[s_{1},s_{1}+\delta]$ by taking closure.

\textbf{Claim 2.} $A_{\gamma(s)}$ is tangent to $A_{q_{1}}$ at $p_{0}$ for
$\forall s\in\lbrack s_{1},s_{2}]$.

If $\mu^{\prime}>0$ on $[s_{1},s_{2}),$ then Claim 2 can be proved by a
standard topology argument. It is also possible to have the existence of
$s_{3}\in(s_{1},s_{2})$ with $\mu^{\prime}>0$ on $[s_{1},s_{3})$ and
$\mu^{\prime}(s_{3})=0.$ Then, Claim 2 holds on $[s_{1},s_{3}]$ by the same
argument. Let $q_{3}=\gamma(s_{3}).$ $A_{q_{3}}$ is a subset of a hyperplane
$H=\{x\in\mathbf{R}^{n}:x\cdot\gamma^{\prime}(s_{3})=a_{0}\}$ dividing
$\mathbf{R}^{n}$ into two half spaces and $A_{\gamma(s)}$ are tangent to
$A_{q_{3}}$ at $p_{0}$ for $\forall s\in\lbrack s_{1},s_{3}).$ The spheres
containing $A_{\gamma(s)}$ ($s\in\lbrack s_{1},s_{3})$) are on the same side
of $H$ as $A_{q_{1}}$, their centers are on the line $\ell$ perpendicular to
$H$ at $p_{0}$, and the set of their radii is $[\sigma_{1},\infty).$
$\mu^{\prime}(s_{2})\neq0$ and $A_{q_{2}}$ is a subset of a sphere, since
$A_{q_{2}}$ and $A_{q_{3}}$ are tangent at $p_{0}$. $A_{q_{1}}$ and $A_{q_{2}%
}$ must be on the opposite sides of $H$ since the center of $A_{q_{2}}$ is
also on $\ell,$ and the radius of $A_{q_{2}}$ is not less than the radius of
$A_{q_{1}}.$ By studying the function $g_{Rv}(s)=\gamma^{\prime}(s_{3}%
)\cdot\exp(\gamma(s),Rv(s)),$ and using the first characterization of
$F_{p}^{\prime\prime}$ in Lemma 5, in a similar proof to Claim 1, one can
obtain that
\[
\exists\delta^{\prime}>0,\forall s\in(s_{3},s_{3}+\delta^{\prime}),\forall
p\in A_{\gamma(s)}\cap O(K_{1},R_{1}\mu),\text{ }p\cdot\gamma^{\prime}%
(s_{3})\geq a_{0}.
\]
To avoid any transversal intersections with $A_{q_{2}},$ $A_{\gamma(s)}$ must
stay between the codimension 1 submanifolds (a sphere and a plane) containing
$A_{q_{2}}$ and $A_{q_{3}},$ respectively$.$ This forces $A_{\gamma(s)}$ to be
tangent to $A_{q_{3}}$ as well as $A_{q_{1}}$ at $p_{0} $ for $\forall
s\in(s_{3},s_{3}+\delta^{\prime}),$ which is still true on $[s_{1}%
,s_{3}+\delta^{\prime}]$ by taking closure and combining with above.
$\mu^{\prime}<0$ on $(s_{3},s_{3}+\delta^{\prime}]$, since (i) any zero of
$\mu^{\prime}$ will give a hyperplane tangent to $A_{q_{3}} $ which cannot
happen, and (ii) any positive value of $\mu^{\prime}$ will give a sphere whose
center is on $\ell$ but on the same side of $H$ as $A_{q_{1}},$ which cannot
happen by continuity and $A_{\gamma(s)}\cap A_{\gamma(s^{\prime})}=\{p_{0}\}$
for $s<s_{3}<s^{\prime}.$ One repeats the proof of Claim 1 by showing that
$f_{Rv}$ is decreasing with $\mu^{\prime}<0,$ and Lemma 5, to extend Claim 2
to $[s_{1},s_{2}].$

$p_{0}=\exp^{\mu}(\gamma(s),r(s)v(s))$ for some curve $(\gamma
(s),r(s)v(s)):[s_{1},s_{2}]\rightarrow NK_{1}.$ Hence, $r(s)=\left\Vert
\gamma(s)-p_{0}\right\Vert /\mu(s)\equiv r_{1}>0$ by the Corollary 2(ii),
$v(s)=N_{\gamma}(s)$ and $\left(  \mu^{\prime}\right)  ^{2}-\mu\mu
^{\prime\prime}=r_{1}^{-2}$ on $[s_{1},s_{2}]$ by Proposition 8. $\forall
s\in\lbrack s_{1},s_{2}],$ $q_{\gamma(s)}^{\ast}=p_{0}$, since $q_{\gamma
(s)}^{\ast}$ is unique. One can extend $[s_{1},s_{2}]$ to a maximal closed
interval by requiring $p_{0}\in A_{\gamma(s)}.$

To summarize, if $\exp^{\mu}(q_{1},r_{1}v_{1})=\exp^{\mu}(q_{2},r_{2}%
v_{2})=p_{0}$, for $r_{1}$, $r_{2}<UR(K,\mu)$ and $v_{i}\in UNK_{i}$ for
$i=1,2,$ then (i) $r_{1}=r_{2}$, (ii) $\exp^{\mu}(\gamma(s),r_{1}N_{\gamma
}(s))=p_{0}$, $\forall s\in\lbrack s_{1},s_{2}]$, and (iii) $v_{i}=N_{\gamma
}(s_{i})$ for $i=1,2$. However, it is essential to observe that this can be
done on one arc of $\gamma$ between $q_{1}$ and $q_{2},$ not both, since we
chose the interval $[s_{1},s_{2}]$ in a particular way above.

Observe that $q_{\gamma(s)}^{\ast}=p_{0}$, $\forall s\in\lbrack s_{1},s_{2}]$
or $[s_{2}-L,s_{1}]$, if $p_{0}\in A_{\gamma(s_{1})}\cap A_{\gamma(s_{2})}.$
This proves that
\begin{align*}
\exp^{\mu}(Sng_{i}^{NK})\cap\exp^{\mu}(NK_{i}\cap D(UR)-Sng_{i}^{NK})  &
=\varnothing\text{ and}\\
\exp^{\mu}(D(UR)-Sng^{NK})  &  =O(K,\mu UR)-Sng.
\end{align*}

\end{proof}

\begin{remark}
In the proof of Claim 1, it is essential that the fibers $A_{q}$ are subsets
of spheres and planes. $f_{x}(t)=x^{2}t-t^{3},$ satisfies that $f_{x}^{\prime
}(0)=x^{2}>0$ except $x=0,$ but \textquotedblleft$\forall x,$ $f_{x}%
(\varepsilon)\geq0=f_{x}(0)$\textquotedblright\ is false for all
$\varepsilon>0,$ since $f_{0}(t)=-t^{3}.$
\end{remark}

\begin{proposition}
Let $\gamma(s):\mathbf{R\rightarrow}K_{1}\subset\mathbf{R}^{n}$ be a unit
speed parametrization of a connected $K_{1}$ such that $\exp^{\mu}%
(\gamma(s),rN_{\gamma}(s))=p_{0}$, $\forall s\in\lbrack s_{1},s_{2}]$, for
$s_{1}<s_{2}$ and $r<UR(K_{1},\mu)$ as in Proposition 10. Then, $\kappa$ is a
positive constant on the interval $[s_{1},s_{2}]$ and
\begin{align*}
\left(  \mu^{\prime}\right)  ^{2}-\mu\mu^{\prime\prime}  &  =\frac{1}%
{r_{1}^{2}}\text{ and }\gamma^{\prime\prime\prime}+\kappa^{2}\gamma^{\prime
}=0,\\
\mu &  =\frac{2}{\kappa r_{1}}\cos\left(  \frac{\kappa s}{2}+a\right)  \text{
for some }a\in\mathbf{R.}%
\end{align*}
Therefore, Horizontal Collapsing Property occurs in a unique way only above
arcs of circles of curvature $\kappa$ and with a specific $\mu.$
$\gamma([s_{1},s_{2}])\neq K_{1}$, even if $[s_{1},s_{2}]$ is chosen to be a
maximal interval satisfying above.
\end{proposition}

\begin{proof}
By Propositions 8 and 10, $(\gamma(s),rN_{\gamma}(s))\in Sng^{NK}(K,\mu)$ and
\begin{equation}
\left(  \mu^{\prime}\right)  ^{2}-\mu\mu^{\prime\prime}=\frac{1}{r^{2}}\text{
and }\mu^{\prime\prime}+\frac{1}{4}\kappa^{2}\mu=0\text{ with }\kappa>0.
\end{equation}%
\begin{align*}
0  &  =\left(  \left(  \mu^{\prime}\right)  ^{2}-\mu\mu^{\prime\prime}\right)
^{\prime}=\left(  \left(  \mu^{\prime}\right)  ^{2}+\frac{1}{4}\kappa^{2}%
\mu^{2}\right)  ^{\prime}\\
0  &  =2\mu^{\prime}\mu^{\prime\prime}+\frac{1}{2}\kappa\kappa^{\prime}\mu
^{2}+\frac{1}{2}\kappa^{2}\mu\mu^{\prime}\\
0  &  =2\mu^{\prime}\left(  \mu^{\prime\prime}+\frac{1}{4}\kappa^{2}%
\mu\right)  +\frac{1}{2}\kappa\kappa^{\prime}\mu^{2}\\
0  &  =\frac{1}{2}\kappa\kappa^{\prime}\mu^{2}%
\end{align*}
$\kappa$ is constant, since $\kappa$ and $\mu>0.$ $\mu=\frac{2}{\kappa r}%
\cos\left(  \frac{\kappa s}{2}+a\right)  $ is the only solution of (6.2).
\begin{align*}
\sqrt{1-\left(  r\mu^{\prime}\right)  ^{2}}  &  =\frac{\kappa r\mu}{2}\text{
and }\gamma^{\prime\prime}=\kappa N_{\gamma}\\
p_{0}  &  =\exp^{\mu}(\gamma,rN_{\gamma})=\gamma-r^{2}\mu\mu^{\prime}%
\gamma^{\prime}+r\mu\sqrt{1-\left(  r\mu^{\prime}\right)  ^{2}}N_{\gamma}\\
0  &  =\left(  \gamma-r^{2}\mu\mu^{\prime}\gamma^{\prime}+\frac{1}{2}r^{2}%
\mu^{2}\gamma^{\prime\prime}\right)  ^{\prime}\\
0  &  =\left(  1-\left(  r\mu^{\prime}\right)  ^{2}-r^{2}\mu\mu^{\prime\prime
}\right)  \gamma^{\prime}+0\cdot\gamma^{\prime\prime}+\frac{1}{2}r^{2}\mu
^{2}\gamma^{\prime\prime\prime}\\
0  &  =\left(  \frac{1}{4}\kappa^{2}\mu^{2}-\mu\mu^{\prime\prime}\right)
r^{2}\gamma^{\prime}+\frac{1}{2}r^{2}\mu^{2}\gamma^{\prime\prime\prime}\\
0  &  =\frac{1}{2}r^{2}\kappa^{2}\mu^{2}\gamma^{\prime}+\frac{1}{2}\mu
^{2}r^{2}\gamma^{\prime\prime\prime}=\frac{1}{2}\mu^{2}r^{2}\left(  \kappa
^{2}\gamma^{\prime}+\gamma^{\prime\prime\prime}\right) \\
0  &  =\kappa^{2}\gamma^{\prime}+\gamma^{\prime\prime\prime}\\
p_{1}  &  =\kappa^{2}\gamma+\gamma^{\prime\prime}\text{ for some constant
}p_{1}\in\mathbf{R}^{n}\\
\left\Vert \frac{p_{1}}{\kappa^{2}}-\gamma\right\Vert  &  =\frac{1}{\kappa
^{2}}\left\Vert \gamma^{\prime\prime}\right\Vert =\frac{1}{\kappa}%
\end{align*}
$\gamma$ is an arc of a circle in $\mathbf{R}^{n},$ since $\gamma$ has
curvature $\kappa$ and lying on a sphere of radius $1/\kappa,$ it has to be a
great circle of that sphere. Since $\mu$ is not constant and $K$ is compact,
there are points where $\mu^{\prime\prime}\geq0$ on each component of $K$.
However, on $[s_{1},s_{2}],$ $\mu^{\prime\prime}=-\frac{1}{4}\kappa^{2}\mu<0.$
$\gamma([s_{1},s_{2}])\neq K_{1}.$
\end{proof}

\begin{proposition}
Let $\left\{  (K_{i},\mu_{i}):i=1,2,...\right\}  $ be a sequence where each
$K_{i}$ is a disjoint union of finitely many simple smooth closed curves in
$\mathbf{R}^{n}$ with $C^{2}$ weight functions, and similarly for $(K_{0}%
,\mu_{0})$. If $(K_{i},\mu_{i})\rightarrow(K_{0},\mu_{0})$ in $C^{2}$
topology, then
\[
\underset{i\rightarrow\infty}{\lim\sup}AIR(K_{i},\mu_{i})\leq AIR(K_{0}%
,\mu_{0}).
\]

\end{proposition}

\begin{proof}
Let $\gamma_{0}(s):domain(\gamma_{0})\mathbf{\rightarrow}K_{0}$ be a unit
speed onto parametrization. Let $R>FocRad^{-}(K_{0},\mu_{0})$ be given
arbitrarily. By Proposition 3, $\exists s_{0}\in domain(\gamma_{0})$ such that
either $\Lambda(\kappa_{0},\mu_{0})(s_{0})^{-\frac{1}{2}}<R$ with
$\Delta(\kappa_{0},\mu_{0})(s_{0})>0,$ or $\left\vert \mu_{0}^{\prime}%
(s_{0})\right\vert ^{-1}<R.$ By parametrizing all $K_{i}$ over a small common
open interval $I$ about $s_{0}$ with respect to arclength, we can assume that
$\mu_{i}^{\prime\prime}\rightarrow\mu_{0}^{\prime\prime}$ and $\kappa
_{i}\rightarrow\kappa_{0}$ uniformly on $I$. For sufficiently large $i$,
$\Lambda(\kappa_{i},\mu_{i})(s_{0})^{-\frac{1}{2}}<R$ with $\Delta(\kappa
_{i},\mu_{i})(s_{0})>0,$ or $\left\vert \mu_{i}^{\prime}(s_{0})\right\vert
^{-1}<R.$ Hence, $R>FocRad^{-}(K_{i},\mu_{i})$ for sufficiently large $i.$
\[
\underset{i\rightarrow\infty}{\lim\sup}FocRad^{-}(K_{i},\mu_{i})\leq
FocRad^{-}(K_{0},\mu_{0}).
\]

By Proposition 9, for all $(K,\mu):$%
\[
AIR(K,\mu)=UR(K,\mu)=\min\left(  \frac{1}{2}DCSD(K,\mu),FocRad^{-}%
(K,\mu)\right)  .
\]
Suppose that $\exists R_{0}$ such that $AIR(K_{0},\mu_{0})<$ $R_{0}%
<\underset{i\rightarrow\infty}{\text{ }\lim\sup}AIR(K_{i},\mu_{i})$.
\begin{align}
AIR(K_{0},\mu_{0})  &  <R_{0}<\underset{i\rightarrow\infty}{\text{ }\lim\sup
}FocRad^{-}(K_{i},\mu_{i})\leq FocRad^{-}(K_{0},\mu_{0})\\
AIR(K_{0},\mu_{0})  &  =\frac{1}{2}DCSD(K_{0},\mu_{0})<R_{0}\nonumber
\end{align}
$D(R_{0})\subset W(\exp^{\mu_{0}})\subset NK_{0}$ by (6.3)$.$ There exists a
double critical pair $(q_{0},q_{1})$ for $(K_{0},\mu_{0}),$ and a point $p $
on the line segment joining $q_{0}$ and $q_{1}$ such that$\left\Vert
p-q_{i}\right\Vert =R_{1}\mu_{0}(q_{i})$ and $p=\exp^{\mu_{0}}(q_{i}%
,R_{1}v_{i})$ with $v_{i}\in UN(K_{0})_{q_{i}}$ for $i=0,1$ where
$R_{1}=AIR(K_{0},\mu_{0})<R_{0}$. As in the proof of Proposition 7(iii), we
consider $\beta_{1}(s)=\exp^{\mu_{0}}(q_{1},sv_{1})$ for $s\in\left(
R_{1},R_{0}\right)  .$ There exists at most one singular point along
$\beta_{1}$ before $R_{0}$ by Proposition 2 and (6.3). By using Lemma 4 and
the arguments in the proof of Proposition 7(iii) with $\measuredangle
(\beta_{1}^{\prime}(R_{1}),u(p,q_{0}))$ $=\alpha(q_{1},p)-\frac{\pi}{2}%
<\frac{\pi}{2}$, choose $s_{1}\in\left(  R_{1},R_{0}\right)  $ such that
$\left\Vert \beta_{1}(s_{1})-q_{0}\right\Vert \mu_{0}(q_{0})^{-1}<R_{1}$ and
$\exp^{\mu_{0}}$ is not singular at $(q_{1},s_{1}v_{1}).$ There exists an open
connected set $V_{1}^{T}\subset D(R_{0})-D(R_{1})\subset NK_{0}$ such that

i. $(q_{1},s_{1}v_{1})\in V_{1}^{T},$

ii. $\exp^{\mu_{0}}\mid V_{1}^{T}$ is a diffeomorphism onto an open set
$V_{1}$ ($\subset\mathbf{R}^{n}$) containing $\beta_{1}(s_{1}),$

iii. $0<c_{1}\leq\inf\left\Vert d(\exp^{\mu_{0}}\mid V_{1}^{T})\right\Vert
\leq\sup\left\Vert d(\exp^{\mu_{0}}\mid V_{1}^{T})\right\Vert \leq
C_{1}<\infty,$

iv. $\left\Vert x-q_{0}\right\Vert \mu_{0}(q_{0})^{-1}<R_{1}$, $\forall x\in
V_{1},$ and

v. $\{q\in K_{0}:(q,w)\in V_{1}^{T}\}$ is an open arc whose length is much
shorter than the length of the component of $K_{0}$ containing $q_{1}.$

There exists a $\mu_{0}-$closest point $q_{2}\in K_{0}$ to $\beta_{1}(s_{1}),$
and $\beta_{1}(s_{1})=\exp^{\mu_{0}}(q_{2},R_{2}v_{2})$ where $R_{2}<R_{1}.$
By Proposition 1(ii, v), $q_{1}\neq q_{2},$ since $R_{1}<\left\vert
\mu^{\prime}(q_{1})\right\vert ^{-1}.$ Let $\beta_{2}(s)=\exp^{\mu_{0}}%
(q_{2},sv_{2}).$ There exists $s_{2}<R_{2}$ sufficiently close to $R_{2}$ such
that $\exp^{\mu_{0}}$ is not singular at $(q_{2},s_{2}v_{2})$ and $\exp
^{\mu_{0}}(q_{2},s_{2}v_{2})\in V_{1}.$ There exists an open set $V_{2}%
^{T}\subset D(R_{2})\subset NK_{0}$ such that $(q_{2},s_{2}v_{2})\in V_{2}%
^{T}$, $\exp^{\mu_{0}}\mid V_{2}^{T}$ is a diffeomorphism onto an open set
$V_{2}$ with $\beta_{2}(s_{2})\in V_{2}\subset V_{1}$, and satisfying the same
type conditions as (iii) and (v) above. $V_{1}^{T}\cap V_{2}^{T}\subset
V_{1}^{T}\cap D(R_{2})=\varnothing.$

Let $K_{0}^{\prime}$ be open subset of $K_{0}$ such that $V_{1}^{T}\cup
V_{2}^{T}\subset NK_{0}^{\prime}.$ Having chosen $V_{i}^{T}$ small, we can
assume that $K_{0}^{\prime}$ is a union of one or two short open arcs, neither
of which is a whole component of $K_{0}.$ Parametrize $\gamma_{0}%
:I_{0}\rightarrow K_{0}^{\prime}$ and for sufficiently large $i\geq i_{0},$
$\gamma_{i}:I_{0}\rightarrow K_{i}^{\prime}\subset K_{i}$ with unit speed $s$
so that $\left\{  \gamma_{i}|I_{0}\right\}  _{i=i_{0}}^{\infty}$ converges to
$\gamma_{0}|I_{0}$ uniformly in $C^{2}$ topology as $i\rightarrow\infty$. All
$NK_{i}^{\prime}$ are diffeomorphic to (and can be identified with) the fixed
$NK_{0}^{\prime}.$ Since $(K_{i},\mu_{i})\rightarrow(K_{0},\mu_{0})$ in
$C^{2}$ topology, $\exp^{(K_{i}^{\prime},\mu_{i})}:NK_{i}^{\prime}\simeq
NK_{0}^{\prime}\rightarrow\mathbf{R}^{n}$ converges to $\exp^{(K_{0}^{\prime
},\mu_{0})} $ in $C^{1}$ topology. $V_{1}^{T}\cap V_{2}^{T}=\varnothing,$ but
$\exp^{(K_{0}^{\prime},\mu_{0})}(V_{2}^{T})\subset\exp^{(K_{0}^{\prime}%
,\mu_{0})}(V_{1}^{T})$ where all are open sets, and $\exp^{(K_{0}^{\prime}%
,\mu_{0})}$ is a local diffeomorphism on $V_{1}^{T}\cup V_{2}^{T}$ satisfying
(iii). Therefore, for sufficiently large $i,$ $\exp^{(K_{i}^{\prime},\mu_{i}%
)}$ is a local diffeomorphism on $V_{1}^{T}\cup V_{2}^{T}\subset D(R_{0}) $
where $V_{1}^{T}\ $\ and $V_{2}^{T}$ are nonempty disjoint open sets, but
$\exp^{(K_{i}^{\prime},\mu_{i})}(V_{2}^{T})\cap\exp^{(K_{i}^{\prime},\mu_{i}%
)}(V_{1}^{T})\neq\varnothing$. Therefore, by the definition, $AIR(K_{i}%
,\mu_{i})\leq R_{0}$ for sufficiently large $i.$ This contradicts with the
conditions of the initial choice of $R_{0}.$ The nonexistence of such $R_{0}$
proves that\ $\lim\sup_{i\rightarrow\infty}AIR(K_{i},\mu_{i})\leq
AIR(K_{0},\mu_{0}).$
\end{proof}

\subsection{Proofs of the Theorems}

The proof of \textbf{Theorem 1} is provided by Propositions 4, 5, 7, 9, 12,
and Lemma 6. The proof of \textbf{Theorem 2} is provided by Propositions 6, 10
and 11. The proof of \textbf{Theorem 4} is provided by Propositions 8, 9 and 10.

\begin{proof}
\textbf{Theorem 3}

Assume that $R=TIR(K,\mu)<UR(K,\mu)$. Recall the proof of Proposition 4(i)
that (i) $\exp^{\mu}:D(R)\rightarrow O(K,\mu R)$ is a homeomorphism, and
$\forall R^{\prime}$ such that $R<R^{\prime}<UR(K,\mu),$ $\exp^{\mu}\mid
D(R^{\prime})$ is not injective. By Proposition 10(iii, iv), there exists
$p_{0}=\exp^{\mu}(\gamma(s),rN_{\gamma}(s))\in Sng(K,\mu)$ for some
parametrization $\gamma$ of $K,$ $\forall s\in\lbrack s_{1},s_{2}]$ for some
$s_{1}<s_{2}$, and $R\leq r<R^{\prime}$. By Proposition 11, $\gamma\left(
\lbrack s_{1},s_{2}]\right)  $ is a desired arc of a circle with compatible
$\mu$. Conversely, if such an arc of a circle exists, with compatible $\mu$,
then as it was discussed in Example 1, there exists a horizontal collapsing
curve $\exp^{\mu}(\gamma(s),r^{\prime}N_{\gamma}(s))=p_{0}^{\prime}$ with
$\forall s\in\lbrack s_{1}^{\prime},s_{2}^{\prime}]$ for some $s_{1}^{\prime
}<s_{2}^{\prime}$, which must satisfy $R\leq r^{\prime}.$ Therefore,
$TIR(K,\mu)$ is equal to the infimum of such $r.$ If the lengths of disjoint
collapsing curves converges to zero and their $\mu$-height decreases to $R$,
then it is possible that the infimum may not be attainable. If there are no
such circles, then $\exp^{\mu}:D(UR)\rightarrow O(K,\mu UR)$ is injective, and
hence it is a homeomorphism by repeating the proof of Proposition 4(i).
\end{proof}

\section{References}

[BS]\qquad G. Buck and J. Simon, \textit{Energy and lengths of knots,
}Lectures at Knots 96, 219-234.

[CKS]\qquad J. Cantarella, R. B. Kusner, and J. M. Sullivan, \textit{On the
minimum ropelength of knots and links, }Inventiones Mathematicae 150 (2002)
no. 2, p. 257-286.

[CE]\qquad J. Cheeger and D. G. Ebin, \textit{Comparison theorems in
Riemannian geometry, Vol 9, }North-Holland, Amsterdam, 1975.

[Di]\qquad Y. Diao, \textit{The lower bounds of the lengths of thick knots,
}Journal of Knot Theory and Its Ramifications, Vol. 12, No. 1 (2003) 1-16.

[DC]\qquad M. P. DoCarmo, \textit{Riemannian Geometry, }Birkhauser, Cambridge,
Massachusetts, 1992.\textit{\ }

[D1]\qquad O. C. Durumeric, \textit{Thickness formula and }$C^{1}%
-$\textit{compactness of }$C^{1,1}$\textit{\ Riemannian submanifolds,
}preprint, http://arxiv.org/abs/math.DG/0204050

[D2]\qquad O. C. Durumeric, \textit{Local structure of the ideal shapes of
knots, }Topology and its Applications, Volume \textbf{154}, Issue \textbf{17},
(15 September 2007), 3070-3089.

[D3]\qquad O. C. Durumeric, \textit{Local structure of the ideal shapes of
knots, II, Constant curvature case, preprint, Feb. 2006, }arXiv:0706.1037v1 [math.GT]

[GL]\qquad O. Gonzales and R. de La Llave, \textit{Existence of ideal knots,}
J. Knot Theory Ramifications, 12 (2003) 123-133.

[GM]\qquad O. Gonzales and H. Maddocks, \textit{Global curvature, thickness
and the ideal shapes of knots}, Proceedings of National Academy of Sciences,
\textbf{96 }(1999) 4769-4773.

[G]\qquad V. Guillemin and A. Pollack, \textit{Differential Topology,
}Prentice-Hall, Englewood Cliffs New Jersey, 1974

[Ka]\qquad V. Katrich, J. Bendar, D. Michoud, R.G. Scharein, J. Dubochet and
A. Stasiak, \textit{Geometry and physics of knots}, Nature, \textbf{384}
(1996) 142-145.

[LSDR]\qquad A. Litherland, J Simon, O. Durumeric and E. Rawdon,
\textit{Thickness of knots}, Topology and its Applications, \textbf{91}(1999) 233-244.

[M]\qquad J. Milnor, \textit{Topology from the Differentiable Viewpoint},
\textit{Princeton University Press}; Revised edition (November 24, 1997)

[N]\qquad\ A. Nabutovsky, \textit{Non-recursive functions, knots
\textquotedblleft with thick ropes\textquotedblright\ and self-clenching
\textquotedblleft thick\textquotedblright\ hyperspheres, }Communications on
Pure and Applied Mathematics\textit{, }\textbf{48} (1995) 381-428.

[T]\qquad J. A. Thorpe, \textit{Elementary Topics in Differential Geometry,
Springer Verlag, }New York, Heidelberg, Berlin, 1979\textit{.}

\end{document}